\documentclass[A4paper,reqno,11pt]{amsart}
\usepackage{amssymb}
\usepackage{graphicx}
\usepackage{amsmath}
\usepackage{esint}
\usepackage[margin=1.17in]{geometry}

\usepackage[usenames, dvipsnames]{color}
\usepackage{epstopdf}
\usepackage{verbatim}
\usepackage{mathrsfs}
\usepackage{bm}
\usepackage{cite}
\usepackage{color}
\allowdisplaybreaks[4]
\usepackage{amsmath, amssymb, pgfplots}
\pgfplotsset{compat = newest, width = 12cm}
\usepackage{subfigure}
\usepackage{caption}
\usepackage{diagbox}
\usepackage{multirow}

\numberwithin{equation}{section}

\newtheorem{theorem}{Theorem}[section]
\newtheorem{corollary}[theorem]{Corollary}
\newtheorem{lemma}[theorem]{Lemma}
\newtheorem{prop}[theorem]{Proposition}

\theoremstyle{definition}
\newtheorem{remark}[theorem]{Remark}

\theoremstyle{definition}
\newtheorem{definition}[theorem]{Definition}

\theoremstyle{definition}

\makeatletter
\def\dashint{\operatorname%
{\,\,\text{\bf-}\kern-.98em\DOTSI\intop\ilimits@\!\!}}
\makeatother

\def\\det{\text{\det}}

\def\ri{\mathrm{i}}

\def\.5{\frac{1}{2}}

\def\bR{\mathbb{R}}

\def\bN{\mathbb{N}}

\def\Im{\text{Im}\,}

\def\cA{\mathcal{A}}

\def\cR{\mathcal{R}}
\def\S{\mathcal{S}}
\def\cK{\mathcal{K}}

\def\cI{\mathcal{I}}

\def\rd{\mathrm{d}}
\def\pD{\partial D}

\newcommand{\RN}[1]{%
\textup{\uppercase\expandafter{\romannumeral#1}}%
}

\renewcommand{\epsilon}{\varepsilon}

\newcounter{marnote}

\begin{document}

\title[Finite systems of subwavelength resonators]{Resonant frequencies distribution for multiple closely spaced subwavelength resonators}

\author[H.G. Li]{Haigang Li}
\address[H.G. Li]{School of Mathematical Sciences, Beijing Normal University, Laboratory of Mathematics and Complex Systems, Ministry of Education, Beijing 100875, China.}
\email{hgli@bnu.edu.cn}

\author[J.H. Zhang]{Junhua Zhang}
\address[J.H. Zhang]{School of Mathematical Sciences, Beijing Normal University, Laboratory of Mathematics and Complex Systems, Ministry of Education, Beijing 100875, China.}
\email{jhzhang$\_$23@mail.bnu.edu.cn}

\date{\today} % delete this line to display the current date

%%% BEGIN DOCUMENT
\begin{abstract}
In this paper, we investigate a resonant system comprising $N$ closely packed spherical resonators ($N>2$). We analyze how the spatial arrangement of these resonators influences the distribution of resonant frequencies, focusing on leading-order terms. Furthermore, we characterize the asymptotic behavior of resonant modes linked to their respective frequencies. Our results demonstrate distinct trends across configurations: For single-row alignment, the system exhibits $N$ clearly separated resonant frequencies; For multi-row arrangements, the resonant frequency range broadens, though the total number of frequencies may diminish; while for ring configurations, comparable frequency ranges to chain arrangements emerge, but with fewer resonant frequencies. We derive explicit analytical expressions to quantify these frequency distributions. Regarding resonant modes, we identify that at specific frequencies, the gradient of these modes may exhibit different asymptotic behavior between different resonators.

\end{abstract}

\maketitle
%\tableofcontents

\section{Introduction}%\label{sec_intro}
In wave physics, there is significant interest in controlling waves at subwavelength scales, as small objects can exhibit subwavelength resonance and strongly scatter waves with comparatively large wavelengths. For instance, subwavelength acoustic resonators are structures capable of compressing and vibrating in response to wavelengths much larger than their own size. The first notable example of subwavelength resonance in acoustics was observed by Marcel Minnaert when studying the resonance of small air bubbles in water \cite{m}, where the high contrast material properties between air and water are identified as a crucial factor in this phenomenon \cite{afglz,mms,dhbl}. Similar phenomena have also been observed in various settings, including  Helmholtz resonators \cite{az}, plasmonic nanoparticles \cite{amrz,aryz}, and high-contrast dielectric particles \cite{adfms,alz}. In this paper, following the notion in \cite{adh}, we define subwavelength resonators as material inclusions that are bounded and have significant parameter differences from their surrounding medium, and exhibit subwavelength resonance.

The structures composed of subwavelength resonators, a type of metamaterial, have been widely used in waveguiding applications \cite{ad1,ad2,adhy,afglz2,aflyz,kdd,lbfwtd} due to their interaction with waves at subwavelength scales. The design of acoustic systems with specific resonance frequencies \cite{ad1,ad2} and particular spectral gaps \cite{adhy} is a crucial focus in acoustic material research. The subwavelength resonance has multiscale nature. Previous studies have demonstrated that the number, shape, arrangement, and spacing of resonators all influence the behavior of the resonance system \cite{ad1,ady,lz,fa}. Feppon, Cheng, and Ammari \cite{fca} proposed a mathematical theory in acoustic wave scattering in finite one-dimensional high-contrast media, where the system consists of a finite alternation of high-contrast segments with arbitrary lengths and interdistances. They showed the existence of subwavelength resonances, which are analogous to the well-known Minnaert resonances in three-dimensional systems. This work was later extended by Ammari, Barandun, Cao, and Feppon \cite{abcf} to infinite periodic chains of subwavelength resonators.

For two close-to-touching resonators, Ammari, Davies, and Yu \cite{ady} employed bispherical coordinates to analyze the behavior of the coupled subwavelength resonant modes of two spherical high-contrast acoustic resonators. They demonstrated that the pair of resonators exhibits two distinct subwavelength resonant modes, with frequencies having different leading-order asymptotic behavior. Li and Zhao \cite{lz} extended this work by studying resonant modes of arbitrarily shaped convex resonators, revealing how geometric properties, such as convexity, curvatures, and the small inter-distance between resonators govern the resonance behavior of the system. Ammari, Davies, and Hiltune \cite{adh} explored how the generalized capacitance matrix provides a rigorous and intuitive discrete approximation to subwavelength Helmholtz scattering and resonance problems. This framework provides leading-order asymptotic expressions for resonant modes and scattered solutions via the matrix’s eigenvalues and eigenvectors, accompanied by precise error bounds. Notably, as mentioned in \cite{adh}, concise representations for capacitance coefficients become intractable for larger resonator systems. However, in a dilute case where the resonators are small compared to the distance between them, Ammari et al. \cite{adhly} obtained explicit expressions for the asymptotic behaviour of the capacitance matrix. 

As a continuation of \cite{lz}, this paper employs the capacitance approximation method proposed in \cite{adh} to analyze resonant systems composed of multiple closely spaced spherical resonators in chain, matrix, and ring configurations. We investigate the impact of factors such as the number of resonators $N$ ($N>2$), their spatial arrangement, volume, and small inter-resonator distance on subwavelength resonance phenomena. Our analysis reveals how collective resonators fundamentally reshape spectral responses through geometry-dependent wave interactions. In linear chains, positional asymmetries progressively lift degeneracies to generate precisely $N$ discrete, well-separated resonances—a phenomenon that enables multifrequency filtering. Transitioning to matrix configurations introduces coupling along orthogonal dimensions; this cross-dimensional interaction collapses intermediate modes while broadening spectral coverage. Rings impose rotational constraints that enforce modal degeneracy within chain-like frequency bounds, yielding compact resonators with simplified spectral signatures. This demonstrates that spatial organization is a primary design parameter—spectral engineering is achieved through geometric arrangement without altering the intrinsic properties of the resonators, enabling dynamic reconfiguration of functionality.

The second contribution of this work resides in characterizing the asymptotic behavior of resonant modes. We demonstrate that at designated resonant frequencies, spatially selective suppression of gradient singularities occurs between specific resonator pairs, manifesting as reduced blow-up rates relative to other inter-resonator regions. This phenomenon establishes engineered energy concentration channels where field enhancement preferentially localizes along geometrically determined pathways. This represents a novel paradigm for spatially programmable energy trapping in metamaterials. It has profound implications for the design of multifrequency resonators with isolated energy storage channels, tunable dissipation pathways, and the ability to geometrically manipulate energy flux. Technically, this discovery emerges from our isospectral reduction framework, which transforms intractable capacitance matrix eigenvalue problems into analytically solvable structured operators (Toeplitz/circulant matrices) while preserving full spectral information, providing a straightforward and dependable approach for such research. 

Furthermore, the understanding of the gradient blow-up of resonant modes is crucial because the gradient of acoustic pressure elucidates the forces that resonators exert on each other in the presence of sound waves. These forces, referred to as secondary Bjerknes forces \cite{b,c,p2}, are explored in this work, especially in the case of closely interacting bubbles, offering new insights on their behavior. We would like to point out that with slight modifications, the conclusions of this paper can be applied to convex resonators of arbitrary shapes, as well as to other configurations.

The blow-up phenomenon between closely spaced inclusions has been extensively studied in various disciplines, such as electrostatics \cite{aklll,ackly,kl,ly,p,bly,dly,lyz,wen}, linear elasticity \cite{ky,bll,bll2,l1}, and Stokes flows \cite{akky,dll,lx,lxz,lxz2}. It has been proved that as the material parameters of inclusions degenerate to either infinity or zero, there is a significant amplification of the electric field or stress field when the distance between them tends to zero. Therefore, it is important to distinguish these field enhancement phenomena from resonance effects. It is worth noting that in this paper and in \cite{ady,lz}, the material parameters considered are positive, contrasting with the negative permittivities of electromagnetic inclusions for surface plasmons studied in \cite{bt,kbs,ya}.

\subsection{Formulation of the problem} %\label{sub1.1}

Let $N>2$, and $D_1,\dots,D_N$ be $N$ closely positioned convex bounded open sets in $\mathbb{R}^{3}$, representing material inclusions. Denote 
$D:=\cup_{l=1}^N D_l$ and $\Omega:=\bR^3\backslash \overline{D}.$
We assume a homogeneous background medium in $\Omega$ with density $\rho $ and bulk modulus $\kappa$. The density and bulk modulus in the resonators $D_{l}$ are denoted as $\rho_b$ and $\kappa_b$, respectively. We use $\omega$ to represent the frequency of the waves and introduce the auxiliary parameters:
$$v=\sqrt{\frac{\kappa}{\rho }},\quad v_b=\sqrt{\frac{\kappa_b}{\rho _b}},\quad k=\frac{\omega}{v},\quad k_b=\frac{\omega}{v_b},$$
then $v$ and $v_b$ represent the wave speeds in the background medium and the inclusions, respectively, while $k$ and $k_b$ are the corresponding wavenumbers. Furthermore, we introduce a dimensionless contrast parameter
$\delta=\frac{\rho _b}{\rho }$. 

We consider the following Helmholtz resonance problem
\begin{equation}\label{eq_helm}
\left\{\begin{aligned}
\Delta u\!+\! k^2u&=0&&\text{in}\ \Omega,\\
\Delta u\!+\! k_{b}^2u&=0&&\text{in}\ D,\\
u|_+\!-\! u|_-&=0&&\text{on}\ \partial D,\\
\delta\frac{\partial u}{\partial\nu}\Big|_+\!-\!\frac{\partial u}{\partial\nu}\Big|_-&=0&&\text{on}\ \partial D,
\end{aligned}
\right.
\end{equation}
along with the Sommerfeld outgoing radiation condition
$\lim\limits_{|x|\to\infty}|x|^{2}\Big(\frac{\partial}{\partial |x|}-\ri k\Big)(u-u^{in})=0$, to guarantees that energy is radiated outwards by the scattered solution. In this paper, the subscripts $+$ and $-$ indicate taking the limit from outside and inside the boundary $\partial D$, respectively. These Helmholtz equations, which are used to model acoustic and polarised electromagnetic waves, represent the simplest model for wave propagation that still exhibits the rich phenomena associated to subwavelength physics.

To achieve subwavelength resonance, we assume that the contrast parameter is small, with $0<\delta<< 1$, indicating a large contrast between the materials, while the wave speeds $v,v_b$ are of the order of 1. A classic example of material inclusions that satisfies these assumptions is a collection of air bubbles in water, known as Minnaert bubbles \cite{m}. In such cases, the contrast parameter $\delta$ is typically around $10^{-3}$.

\subsection{Main results}

We define a complex number $\omega(\delta)$ as a resonant frequency if the Helmholtz resonance system \eqref{eq_helm} admits a non--trivial solution $u$ with $u^{in}=0$. Such a non--trivial solution $u$ is called a resonant mode. Furthermore, a resonant frequency $\omega(\delta)$ is considered subwavelength if it satisfies $\omega(0)=0$ and depends continuously on the parameter $\delta$. The main purpose of this paper is to develop an analytical framework for characterizing subwavelength resonant frequencies and resonant modes in a finite system of subwavelength resonators in $\mathbb{R}^3$.

Consider $N$ identical spherical resonators $D_{1},D_{2},\dots,D_{N}$, closely spaced in $\mathbb{R}^3$, where each resonator $D_l=B_R(x_l)$ is defined as a ball of radius $R$ centered at $x_l=(x_l^1,x_l^2,x_l^3)$. The separation distance $\epsilon$ between adjacent resonators is chosen to depend on $\delta$, and we perform an asymptotic analysis with respect to $\delta$. Specifically, 
$\epsilon$ is set such that, for some $0<\beta<1$, 
\begin{equation}\label{def_epsilon}
\epsilon= e^{-\Lambda/\delta^{1-\beta}},
\end{equation}
where $\Lambda>0$ is a fixed constant. As shown later, this choice of $\epsilon$ ensures that the subwavelength resonant frequencies remain well-behaved (i.e. $\omega=\omega(\delta)\to 0$ as $\delta\to 0$) and enables the computation of asymptotic expansions in terms of  $\delta$. For $N>2$, new phenomena arise for different arrangements, distinct from the case $N=2$ studied in \cite{ly,ady}. Our results will be presented separately for three arrangements: chain, ring, and matrix, in the following order.

The primary tool used to determine the resonant properties of the system $D = D_{1} \cup\cdots\cup D_{N}$ is the Helmholtz single layer potential. By applying Lemma \ref{le_omega} (\!\!\cite{ady}) and \cite[Lemma 2.11]{adh}, we reduce the resonance problem \eqref{eq_helm} to a matrix eigenvalue problem for the generalized capacitance matrix $\tilde{\mathbb{C}}=(\tilde{C}_{ij})_{N\times N}$ as
$\tilde{C}_{ij}=\frac{\delta v_b^2}{|D_i|}C_{ij}$, and  $C_{ij}=\int_{\Omega}\nabla v_{i}\cdot\nabla v_{j}dx$, where $v_i$ is the unique solution to
\begin{equation}\label{eq_vi}
\begin{cases}
\Delta v_i=0 & \text{in} \ \Omega,\\
v_i=\delta_{ij} & \text{on} \ \partial D_i,\\
v_i=O(|x|^{-1}) & \text{as} \ |x|\to\infty,
\end{cases}1\leq i,j\leq N,
\end{equation}
where $\delta_{ij}$ is the Kronecker delta. We denote the average capacity of $N$ resonators
\begin{equation}\label{def_M}
M:=\frac{\mathrm{Cap}(D)}{N}=-\frac{1}{N}\sum_{i=1}^{N}\int_{B_{\tilde{R}}}\frac{\partial v_i}{\partial\nu}\Big|_+\rd\sigma, 
\end{equation}
where $\tilde{R}$ is a sufficiently large constant such that $\bar{D}\subset B_{\tilde{R}}\!:=\!B_{\tilde{R}}(0)$. Throughout the paper, we use the notation
$O(A)$ to denote a quantity that can be bounded by $CA$, where $C$ is some positive constant independent of $\epsilon$ and $\delta$. 

\subsubsection{A finite chain arrangement} 
We begin by considering a chain of spherical resonators with a radius of $R$ and centered at $x_l=((l-1)(2R+\epsilon),0,0)$ for $l=1,\dots,N$, as illustrated in Figure \ref{ca}. 
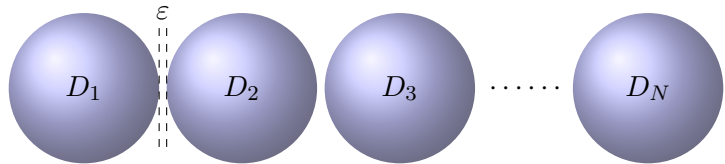
\begin{figure}[h]
\begin{tikzpicture}
\def\myds{1}
%\draw[->] (0,0) coordinate (O) -- (7,0) node[below]{$x$}; 
\shade[shading=ball, ball color=blue!30!white, opacity=0.8]  (0,0) circle[radius=\myds];
\shade[shading=ball, ball color=blue!30!white, opacity=0.8]  (2.1,0) circle[radius=\myds];
\shade[shading=ball, ball color=blue!30!white, opacity=0.8]  (4.2,0) circle[radius=\myds];
\node at (5.9,0) {$\cdots\cdots$};
\shade[shading=ball, ball color=blue!30!white, opacity=0.8] (7.5,0) circle[radius=\myds];
%%\node[label=above:$\omega_1$] at (0,1) {};
%%
%\node at (0,0) {1};
%\node at (2.1,0) {1};
%\node at (4.2,0) {1};
%\node at (6.3,0) {1};
\node at (0,0) {$D_1$};
\node at (2.1,0) {$D_2$};
\node at (4.2,0) {$D_3$};
\node at (7.5,0) {$D_N$};
\node at (1.05,1) {$\varepsilon$};
\draw [dashed] (1,0.8) -- (1,-0.8);
\draw [dashed] (1.1,0.8) -- (1.1,-0.8);
\end{tikzpicture}
\caption{A chain arrangement.}\label{ca}
\end{figure}
For a chain configuration of $N$ resonators, we have the following theorem on the distribution of resonant frequencies.

\begin{theorem}\label{Th_chain}
Let $D=\cup_{l=1}^N D_l$ be a chain of spherical resonators as above. Then, as the contrast of the density $\delta\to0$, there are $N$ subwavelength resonant frequencies which satisfy the asymptotic formula 
$$\omega_1=\sqrt{\frac{3v_b^2M\delta}{4\pi R^3}}\Big(1+o(1)\Big),~\mbox{and}~ 
\omega_i=\sqrt{\frac{3v_b^2a_N^i\delta}{4R^3}|\log\epsilon|}+O\Big(\sqrt{\frac{\delta}{|\log\epsilon|}}+\delta\Big),\quad i=2,\dots,N,$$
where \begin{equation}\label{def_anin}
a_N^i=2(1-\cos \theta_i), \quad\mbox{and}\quad\theta_i=\frac{(i-1)\pi}{N}.
\end{equation}
\end{theorem}
\begin{remark}
By making slight modifications, our approach remains applicable to resonant systems with non-uniformly spaced configurations. For example, when $N=3$, let $\epsilon_1$ denote the distance between $D_1$ and $D_2$, and $\epsilon_2$ denote the distance between $D_2$ and $D_3$. If there is a positive constant $C$ such that $\frac{1}{C}\le\frac{\log \epsilon_1}{\log \epsilon_2}\le C$, then the following statement is true:
$$\omega_1\!=\!\sqrt{\frac{3v_b^2M\delta}{4\pi R^3}}\Big(1+o(1)\Big),~\mbox{and}~ 
\omega_i\!=\!\sqrt{\frac{3v_b^2\delta}{4R^3}a_i(\epsilon_1,\epsilon_2)}+O\Big(\sqrt{\frac{\delta}{|\log\epsilon_1|}}+\sqrt{\frac{\delta}{|\log\epsilon_2|}}+\delta\Big),~ i=2,3,$$
where $a_i(\epsilon_1,\epsilon_2)=|\log\epsilon_1|+|\log\epsilon_2|+(-1)^{i-1}\sqrt{|\log\epsilon_1|^2+|\log\epsilon_2|^2-|\log\epsilon_1||\log\epsilon_2|}$.
\end{remark}

\begin{remark}
When $N=2$, the conclusions obtained in Theorem \ref{Th_chain} are consistent with the results in \cite{ady,lz} by choosing $\epsilon$ as in \eqref{def_epsilon}, yielding 
$$\omega_1=\sqrt{\frac{3v_b^2M\delta}{4\pi R^3}}\big(1+o(1)\big),\quad\mbox{and}~\omega_2=\sqrt{\frac{3\Lambda v_b^2}{2R^3}}\delta^{\beta/2}+O(\delta^{1-\beta/2}).$$ 
Here we would like to point out that when $N=2$, the resonant frequencies are the eigenvalues of a second-order capacitance matrix. These frequencies can be explicitly calculated using the quadratic formula, as they arise from the roots of a quadratic polynomial. For $N>2$, however, the characteristic polynomial in the capacitance matrix becomes degree $N$. By the Abel-Ruffini theorem, no general algebraic solution exists for polynomials of degree $N\ge 5$. Consequently, while the explicit form of the characteristic polynomial remains intractable, we instead analyze its dominant terms. Using the fundamental theorem of algebra and the theorem of existence of zeros, we establish the existence and distribution of its roots, thereby inferring the distribution of resonant frequencies. A critical aspect of this approach is its reliance on the assumption of small interparticle distances (governed by $\epsilon$).
\end{remark}

Furthermore, for its resonant modes, we have
\begin{theorem}\label{Th_mode}
Under the assumptions of Theorem \ref{Th_chain}, let $u^i$, $1\le i\le N$, denote the subwavelength resonant modes associated with the resonant frequencies $\omega_i$ for $D$. Normalize $u^i$ such that
$\lim\limits_{\delta\to0}|u^i|\sim1~ \text{for}\ x\in\partial D.$ 
Then, for the values of the resonant modes $u^i$ on $\pD_l$, $l=1,\dots,N$, we have, as $\delta\to 0$:
\begin{align*}
u^1(x)&=1+O(\omega_1+|\log\epsilon|^{-1}),\\
u^i(x)&=\frac{\sin l\theta_i-\sin(l-1)\theta_{i}}{\sin\theta_i}+O(\omega_i+|\log\epsilon|^{-1}),\quad2\le i\le N.
\end{align*}
\end{theorem}
\begin{remark}\label{Re} Denote the narrow regions between $D_l$ and $D_{l+1}$ for $1\le l\le N-1$ in the chain arrangement by 
\begin{equation}\label{def_omgl}
\Omega_{l,l+1}^{r}:=B_r(\frac{x_l+x_{l+1}}{2})\cap\Omega,\quad\mbox{for some}~0<r<R/4.
\end{equation}
From Theorem \ref{Th_mode}, it follows that
$$\max\limits_{x\in\Omega}|\nabla u^1|,\ \max\limits_{x\in\Omega_{l,l+1}^{r}}|\nabla u^i|\lesssim\frac{1}{\epsilon|\log\epsilon|},\quad\mbox{for}~2\le i\le N,~1\le l\le N-1,$$
where $i$ and $l$ satisfy $l(i-1)=Nt$ for some $t\in\bN_+$ (positive integers), and 
$\max\limits_{x\in\Omega}|\nabla u^i|\sim1/\epsilon$, for $2\le  i\le N$.
\end{remark}

\subsubsection{A finite ring arrangement} 
Now let us discuss the second arrangement: the ring arrangement. Suppose that $\cR$ is a regular polygon with $N$ sides in the $x^1x^2$-plane, with each side length of $2R+\epsilon$. We position the centers of $N$ spherical resonators of radius $R$ at the vertices of $\cR$ and label them in a counterclockwise manner. For instance, the ring arrangements for $N=3$ and $N=6$ can be seen in Figure \ref{ra}. For the three-spheres examples, Ammari and Yu \cite{ya} studied the high field concentration in the narrow gap regions between nanospheres by using transformation optics approach and a modified hybrid numerical scheme.

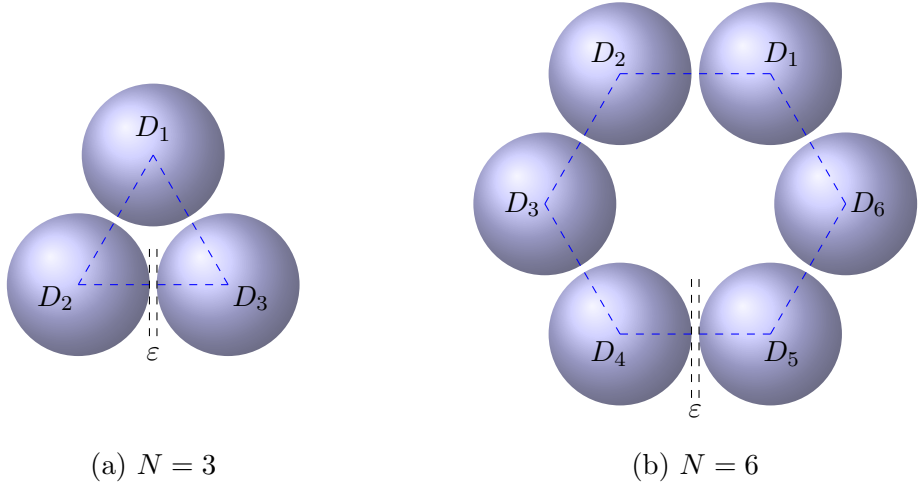
\begin{figure}[h]
\begin{tikzpicture}
\foreach \x in {1,2,3} 
{

\shade[shading=ball, ball color=blue!30!white, opacity=0.8]   ({1.15*cos (120*\x-30)},{1.15*sin (120*\x-30)}) circle[radius=0.95];
\node at ({1.5*cos (120*(\x)-30)},{1.5*sin (120*(\x)-30)}) {$D_{\x}$};

}
\foreach \x in {1,2,3}{
	\draw[dashed,color=blue] ({1.15*cos (120*\x-30)},{1.15*sin (120*\x-30)}) -- ({1.15*cos (120*(1+\x)-30)}, {1.15*sin(120*(1+\x)-30)});
	}
\draw[dashed] (-0.05,-0.1)--(-0.05,-1.3);
\draw[dashed] (0.05,-0.1)--(0.05,-1.3);
\node at (0,-1.5) {$\varepsilon$};
\node at (0,-3) {(a) $N=3$};
\node at (4,0){};
\end{tikzpicture}
\begin{tikzpicture}
\foreach \x in {1,2,3,4,5,6} 
{
\shade[shading=ball, ball color=blue!30!white, opacity=0.8]   ({2*cos (60*\x)},{2*sin (60*\x)}) circle[radius=0.95];
\node at ({2.3*cos (60*(\x))},{2.3*sin (60*(\x))}) {$D_{\x}$};
}
\foreach \x in {1,...,6}{
	\draw[dashed,color=blue] ({2*cos (60*\x)},{2*sin (60*\x)}) -- ({2*cos (60*(1+\x))}, {2*sin(60*(1+\x))});
}
\draw[dashed] (-0.05,-1)--(-0.05,-2.6);
\draw[dashed] (0.05,-1)--(0.05,-2.6);
\node at (0,-2.75) {$\varepsilon$};
\node at (0,-3.5) {(b) $N=6$};
\end{tikzpicture}
\caption{Two ring arrangements.}\label{ra}
\end{figure}

Set
\begin{equation}\label{def_alpN}
\alpha(N):=\begin{cases}
(N+1)/2,&N\text{ is odd,}\\
N/2,&N\text{ is even.}\\
\end{cases}
\end{equation}
For the ring arrangement, we have the distribution of resonant frequencies in the following.
\begin{theorem}\label{Th_ring}
Let $D$ be a ring arrangement illustrated as above. Then, as $\delta\to0$, the resonant frequencies of $D$ are given by $\omega_1=\sqrt{\frac{3v_b^2M\delta}{4\pi R^3}}\Big(1+o(1)\Big)$, and 
$$\omega_{2i-2},\omega_{2i-1}=\sqrt{\frac{3v_b^2\hat{a}_N^i}{4R^3}\delta|\log\epsilon|}+O\Big(\sqrt{\frac{\delta}{|\log\epsilon|}}+\delta\Big),\quad i=2,\dots,\alpha(N),$$
where 
\begin{equation}\label{def_hai}
\hat{a}_N^i=2(1-\cos \hat{\theta}_i),\quad\mbox{and}~\hat{\theta}_i=2\theta_i.
\end{equation}
Furthermore, if $N$ is even, $\omega_N=\sqrt{\frac{3v_b^2}{R^3}\delta|\log\epsilon|}+O\Big(\sqrt{\frac{\delta}{|\log\epsilon|}}+\delta\Big).$
\end{theorem}
\begin{theorem}\label{Th_mode_r}
Under the assumptions of Theorem \ref{Th_ring}, let $u^i$, $1\le i\le N$, denote the subwavelength resonant modes for $D$.  Normalize $u^i$ such that
$\lim\limits_{\delta\to0}|u^i|\sim1$ for $x\in\partial D.$
Then, for resonant modes $u^i$ associated with $\omega_i$, as $\delta\to 0$, we have
\begin{equation*}
u^i(x)=\alpha_{l}^i+O(\omega_i+|\log\epsilon|^{-1})\quad\text{on}\ \partial D_l, 1\le i,l\le N,
\end{equation*}
where $\alpha_l^1=1 $, and for $t=2,\dots, \alpha(N)$,
\begin{equation*}
\begin{aligned}
\alpha^{2t-2}_l&=k_1^{2t-2}\frac{\sin (l-2)\hat{\theta}_{2t-2}}{\sin\hat{\theta}_{2t-2}}+k_2^{2t-2}\frac{\sin (l-1)\hat{\theta}_{2t-1}}{\sin\hat{\theta}_{2t-1}},\\
\alpha^{2t-1}_l&=k_1^{2t-1}\frac{\sin (l-2)\hat{\theta}_{2t-2}}{\sin\hat{\theta}_{2t-2}}+k_2^{2t-1}\frac{\sin (l-1)\hat{\theta}_{2t-1}}{\sin\hat{\theta}_{2t-1}},
\end{aligned}
\end{equation*}
with constants $k_1^{2t-2}$, $k_2^{2t-2}$, $k_1^{2t-1}$ and $k_2^{2t-1}$. Moreover, $\alpha^N_l=(-1)^l$, if $N$ is even.
\end{theorem}

\subsubsection{A finite matrix arrangement} Now we consider the third case. Assume $N = m \times n$, where $n \geq m \geq 2$ and $m, n \in \mathbb{Z}$. For the ball located in row $\gamma$ and column $\alpha$, we can count it as $i = (\gamma-1) n + \alpha$, where $1 \leq \gamma \leq m$ and $1 \leq \alpha \leq n$. Denote its center as  
$x_i = ((\alpha-1)(2R+\epsilon), (\gamma-1)(2R+\epsilon), 0).$
This configuration of $N$ resonators is referred to as a matrix arrangement, as illustrated in Figure \ref{ma}.

\begin{figure}[h]
\begin{tikzpicture}
\def\myds{1}
\shade[shading=ball, ball color=blue!30!white, opacity=0.8]  (0,0) circle[radius=\myds];
\shade[shading=ball, ball color=blue!30!white, opacity=0.8]  (2.1,0) circle[radius=\myds];
\shade[shading=ball, ball color=blue!30!white, opacity=0.8]  (4.2,0) circle[radius=\myds];
\node  at   (5.55,0){$\cdots$};
\shade[shading=ball, ball color=blue!30!white, opacity=0.8]  (6.8,0) circle[radius=\myds];
\shade[shading=ball, ball color=blue!30!white, opacity=0.8]  (0,2.1) circle[radius=\myds];
\shade[shading=ball, ball color=blue!30!white, opacity=0.8]  (2.1,2.1) circle[radius=\myds];
\shade[shading=ball, ball color=blue!30!white, opacity=0.8] (4.2,2.1) circle[radius=\myds];
\node  at   (5.55,2.1){$\cdots$};
\shade[shading=ball, ball color=blue!30!white, opacity=0.8] (6.8,2.1) circle[radius=\myds];
\node  at   (0,3.5){$\vdots$};
\node  at   (2.1,3.5){$\vdots$};
\node  at   (4.2,3.5){$\vdots$};
\node  at   (6.8,3.5){$\vdots$};
\shade[shading=ball, ball color=blue!30!white, opacity=0.8]  (0,4.7) circle[radius=\myds];
\shade[shading=ball, ball color=blue!30!white, opacity=0.8]  (2.1,4.7) circle[radius=\myds];
\shade[shading=ball, ball color=blue!30!white, opacity=0.8]  (4.2,4.7) circle[radius=\myds];
\node  at   (5.55,4.7){$\cdots$};
\shade[shading=ball, ball color=blue!30!white, opacity=0.8]  (6.8,4.7) circle[radius=\myds];

\node at (0,0) {$D_1$};
\node at (2.1,0) {$D_2$};
\node at (4.2,0) {$D_3$};
\node at (6.8,0) {$D_n$};
\node at (0,2.1) {$D_{n+1}$};
\node at (2.1,2.1) {$D_{n+2}$};
\node at (4.2,2.1) {$D_{n+3}$};
\node at (6.8,2.1) {$D_{2n}$};
\node at (0,4.7) {$D_{(m-1)n+1}$};
\node at (2.1,4.7) {$D_{(m-1)n+2}$};
\node at (4.2,4.7) {$D_{(m-1)n+3}$};
\node at (6.8,4.7) {$D_{N}$};

\node at (1.05,-1.){$\varepsilon$};
\draw [dashed](1,0.8) -- (1,-0.8);
\draw [ dashed](1.1,0.8) -- (1.1,-0.8);
\node  at (-1,1.05){$\varepsilon$};
\draw [dashed, line width=0.3pt](0.8,1) -- (-0.8,1);
\draw [dashed, line width=0.3pt](0.8,1.1) -- (-0.8,1.1);
\end{tikzpicture}
\caption{Matrix arrangements.}\label{ma}
\end{figure}
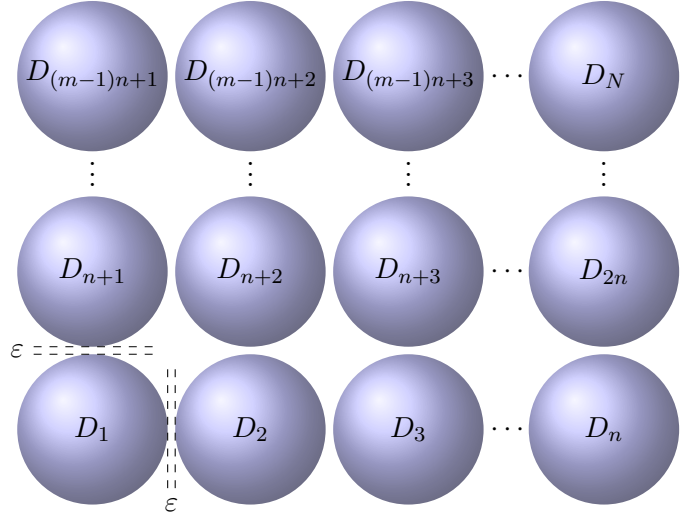

When the resonators are arranged in a matrix format, the distribution of their resonant frequencies becomes more intricate. Let $a_{m}^{\gamma}$ and $a_{n}^{\alpha}$ be defined in \eqref{def_anin}, and denote $\bar{a}_N^1\le\dots\le\bar{a}_N^N$ as all the roots of the polynomial in $a$ given by
\begin{equation}\label{amianj}
\prod_{\gamma=1}^{m}\prod_{\alpha=1}^{n}\big(a-(a_{m}^{\gamma}+a_{n}^{\alpha})\big).
\end{equation}

\begin{theorem}\label{Th_matrix}
Let $D$ be a matrix arrangement illustrated as above and $\bar{a}_N^i$ be the roots of the polynomial in \eqref{amianj}. Then, as $\delta\to0$, the resonant frequencies of $D$ are given by $$\omega_1=\sqrt{\frac{3v_b^2M\delta}{4\pi R^3}}\Big(1+o(1)\Big),~\mbox{and}~ \omega_i=\sqrt{\frac{3v_b^2\bar{a}_N^i\delta}{4R^3}|\log\epsilon|}+O\Big(\sqrt{\frac{\delta}{|\log\epsilon|}}+\delta\Big),\quad i=2,\dots,N.$$
\end{theorem}

\begin{remark}
We now compare the three configurations (chain, ring, and matrix). Define $\eta=\sqrt{(3\Lambda v_b^2)/(4R^3)}\delta^{\beta/2}$. 
In the case of a chain arrangement, with $\epsilon$ chosen as in \eqref{def_epsilon}, the resonant frequencies are  $\omega_1=\sqrt{\frac{3v_b^2M\delta}{4\pi R^3}}\big(1+o(1)\big)$ and $\omega_i=\sqrt{a_N^i}\eta+O(\delta^{1-\beta/2})$ for $ i=2,\dots, N$. At leading order, there are $N$ distinct resonant frequencies. Since $0\le a_N^i < 4$, these frequencies satisfy $0<\omega_i<2\eta$ for $ i=1,\dots, N$.

For ring arrangements, $0\le \hat{a}_N^i \le 4$, so the resonant frequencies (up to an $O(\delta^{1-\beta/2})$ error) satisfy $0<\omega_i\le2\eta$ for $i=1,\dots,N$. At leading order, the number of distinct frequencies is $(N+1)/2$ if $N$ is odd, and $(N+2)/2$ if $N$ is even.

In the case of matrix arrangements, the frequencies are $\omega_1=\sqrt{\frac{3v_b^2M\delta}{4\pi R^3}}\big(1+o(1)\big)$ and $\omega_i=\sqrt{\bar{a}_N^i}\eta+O(\delta^{1-\beta/2})$ for $i=2,\dots,N$. Let $s(m,n)$ denote the number of distinct roots of \eqref{amianj}. We find $s(m,n)\le N$, implying $s(m,n)$ distinct leading-order resonant frequencies. Since $0\le \bar{a}_N^i <8$, the frequencies span $0<\omega_i<2\sqrt{2}\eta$ for $i=1,\dots,N$.

There are two key observations. Multi-row (matrix) configurations exhibit a broader frequency range ($0<\omega_i<2\sqrt{2}\eta$) but fewer distinct frequencies. Chain and ring configurations share the same frequency range ($0<\omega_i\le2\eta$), yet the ring configuration has fewer distinct frequencies. These results are summarized in Tables \ref{tab1} and \ref{tab2} blow.

\begin{table}[h]
\centering
\begin{tabular}{|c|c|c|}
\hline
\diagbox{Arrangements\!\!\!\!}{Frequencies\\ distribution} & \quad\quad\quad Span \quad\quad\quad\quad~~~& \quad\quad\quad Number\quad \quad\quad\quad~~~\\ \hline\rule{0pt}{13pt}
Chain & (0,\,2$\eta$) & $N$ \\ \hline \rule{0pt}{13pt}
\multirow{2}{*}{Ring} & \multirow{2}{*}{(0,\,2$\eta$]} & $(N+1)/2$, if $N$ is odd\\ 
\cline{3-3}\rule{0pt}{13pt}
 &&$(N+2)/2$, if $N$ is even\\ \hline\rule{0pt}{13pt}
Matrix & (0,\,2$\sqrt{2}\eta$) & $s(m,n)$ \\ \hline
\end{tabular}
\caption{Frequencies distribution span and number for the three different arrangements.}\label{tab1}
\end{table}
\begin{table}[h]
\centering
\begin{tabular}{|c|c|c|c|}
\hline
\diagbox{Arrangements\!\!\!\!\!\!\!\!\!\!\!}{Number of\\ \ resonators} & \quad$N=4$\quad\quad &\quad$N=5$\quad\quad& $N=16$ \\ \hline\rule{0pt}{13pt}
Chain & 4 & 5 &16\rule{0pt}{13pt} \\ \hline \rule{0pt}{13pt}
Ring & 3 &3&9   \\\hline\rule{0pt}{13pt}
\multirow{2}{*}{Matrix} &  \multirow{2}{*}{3} & \multirow{2}{*}{\diagbox{ }{ }} & \quad 15 ($m=2,n=8$)\quad\quad \\
\cline{4-4}\rule{0pt}{13pt}
&&&\, 9  ($m=4,n=4$)\\
\hline
\end{tabular}
\caption{For $N=$ 4, 5, and 16, the number of distinct leading-order terms of frequencies associated with the three different arrangements.}\label{tab2}
\end{table}
 \end{remark}

\begin{remark}
In \cite{fa}, the authors used the properties of circulant matrices to express the resonant frequencies of a ring configuration in terms of the entries of the capacitance matrix. We emphasize that the formulas derived in Theorems \ref{Th_ring} and \ref{Th_mode_r}, which characterize the resonant frequencies and their corresponding modes, are more concise and intuitively interpretable than existing results. Furthermore, our approach associates each resonant frequency with its corresponding mode, enabling a clearer physical interpretation.
\end{remark}

This paper is organized as follows. In Section \ref{sec_layer} we employs layer potential techniques to formulate the problem, reducing the analysis of resonant mode characteristics to the calculation of capacitance coefficients. In Section \ref{sec_chain} we analyzes a chain configuration, where spherical resonators are aligned in a single row. Section \ref{sec_matrix} explores multi-row (matrix) configurations of spherical resonators. Section \ref{sec_ring} investigates the spectral properties of ring-shaped resonator arrangements. In Section \ref{sec_exam}, we  provides several examples to illustrate our findings in Sections \ref{sec_chain}, \ref{sec_ring}, and \ref{sec_matrix}. In Section \ref{sec_con} we concludes the paper with some remarks.

\section{Preliminary}\label{sec_layer}
The main tool that allows us to reveal the resonant
properties of the system $D =\cup_{l=1}^{N}D_{l}$ is the Helmholtz single layer potential. In this section, we give some preliminary results for the single layer potential theory and reduce the resonance problem \eqref{eq_helm} to a matrix eigenproblem for the generalized capacitance matrix. 

Let $G^k(x,y)$ be the Helmholtz Green's function
$$G^k(x,y):=-\frac{e^{\mathrm{i}k|x-y|}}{4\pi|x-y|},\quad x,y\in\bR^3, k\ge 0,$$
and $\S_D^k:L^2(\partial D)\to H_{\mathrm{loc}}^1(\bR^3)$ be the corresponding single layer potential
\begin{equation*}
\S_D^k[\phi](x) := \int_{\partial D} G^k(x,y) \phi(y) \, \rd y, \quad x\in\bR^3,\phi \in L^2(\partial D).
\end{equation*}
The Neumann--Poincar\'e operator $\cK_D^{k,*}:L^2(\pD)\to L^2(\pD)$ is defined by
$$\cK_D^{k,*}[\phi](x):=\int_{\partial D}\frac{\partial}{\partial\nu_x}G^k(x,y)\phi(y)\rd y,\quad x\in\pD,$$
where $\partial/\partial\nu_x$ denotes the outward normal derivative at the point $x$. Then we have the following lemma, which is from \cite{adh}.

\begin{lemma}(\!\cite{adh})
In the regime $\omega\to 0$, the Helmholtz problem \eqref{eq_helm} is equivalent to finding $(\psi,\phi)\in L^2(\pD)\times L^2(\pD)$ such that
\begin{equation}\label{eq_phipsi}
\cA(\omega,\delta)\left(\begin{matrix}
\phi\\
\psi
\end{matrix}\right)=\left(\begin{matrix}
u^{in}\\
\delta\frac{\partial u^{in}}{\partial\nu_x}
\end{matrix}\right),
\end{equation}
where
\begin{equation*}\label{def_cA}
\cA(\omega,\delta):=\left(\begin{matrix}
\S_D^{k_b}&-\S_D^{k}  \\
-\frac{1}{2}\cI+\cK_D^{k_b,*}&-\delta(\frac{1}{2}\cI+\cK_D^{k,*}) 
\end{matrix}\right),
\end{equation*}
and $\cI$ is the identity operator on $L^2(\pD)$. 
\end{lemma}
\begin{proof}
It is easy to see that the function $u$ defined by
\begin{equation}\label{us}
u=\begin{cases}
u^{in}+\S_D^k[\psi](x),&x\in\Omega,\\
\S_D^{k_b}[\phi](x),&x\in D,
\end{cases}
\end{equation}
satisfies the first two equations in \eqref{eq_helm} for surface potentials $(\phi,\psi)\in L^2(\pD)\times L^2(\pD)$. It is known from \cite{ck} that for the single layer potential $\S_D^k$ and the Neumann--Poincar\'e operator $\cK_D^{k,*}$, the following relation holds:
$\frac{\partial}{\partial\nu_x}\S_D^k[\phi]\big|_\pm=(\pm\frac{1}{2}\cI+\cK_D^{k,*})[\phi]$.
To ensure that $u$ is the solution to equation \eqref{eq_helm}, we simply need to choose $(\phi,\psi)$ that satisfies the transmission condition across $\pD$, meaning that $(\phi,\psi)$ solves equation \eqref{eq_phipsi}.
\end{proof}

\begin{definition}
We refer to a subwavelength resonance as a complex frequency $\omega=\omega(\delta)$ that satisfies $\omega(\delta)\to 0$ as $\delta\to 0$, and for which \eqref{eq_phipsi} has a non-zero solution $u(\omega,\delta)$ for zero incoming field $u^{in}$:
\begin{equation*}
\cA(\omega,\delta)\left(\begin{matrix}
\phi\\
\psi
\end{matrix}\right)=\left(\begin{matrix}
0\\
0
\end{matrix}\right).
\end{equation*}
For each resonant frequency $\omega$, we define the corresponding resonant mode (or eigen-mode) as 
\begin{equation}\label{def_eigenmode}
u=\begin{cases}
\S_D^{k}[\psi](x),&x\in \Omega,\\
\S_D^{k_b}[\phi](x),&x\in D.
\end{cases}
\end{equation}
\end{definition}

Under the assumption that $\delta<< 1$, we analyze \eqref{eq_phipsi} through expansion. For sufficiently small $k>0$, the Helmholtz Green's function can be expanded as follows:
\begin{align*}
G^k(x,y)=\sum_{n=0}^\infty -\frac{\mathrm{i}^n}{4\pi n!}|x-y|^{n-1}k^n
=:\sum_{n=0}^{\infty}G_n(x,y)k^n.
\end{align*}
By this expansion, we obtain an expansion for $\S_D^k$:
\begin{align}\label{expan_s}
\S_D^k=\S_D+\sum_{n=1}^\infty k^n\S_{D,n},
\end{align}
where 
$\S_{D,n}[\phi]=\int_{\partial D}G_n(x,y)\phi(y)\rd y,~ x\in\bR^3,\phi\in L^2(\pD),$
for $n=0,1,\dots$, and $\S_D=\S_{D,0}$ represents the Laplace single layer potential. Similarly, for $\cK_D^{k,*}$, we have
\begin{align}\label{expan_k}
\cK_D^{k,*}=\cK_D^*+\sum_{n=1}^\infty k^n\cK_{D,n},
\end{align}
where
$\cK_{D,n}[\phi]=\int_{\partial D}\frac{\partial}{\partial\nu_x}G_n(x,y)\phi(y)\rd y,~ x\in\bR^3,\phi\in L^2(\pD),$ for $n=0,1,\dots$, 
and $\cK_D^*=\cK_{D,0}$ is the Neumann--Poincar\'e operator corresponding to the Laplacian. As in \cite{ady}, we could write $\S_D^k=\S_D+O(k)$ and $\cK_D^{k,*}=\cK_D^*+O(k)$ in the relevant operator norms, respectively.

We are interested in how the kernel of  
\begin{equation*}%\label{def_cA00}
\cA(0,0)=\left(\begin{matrix}
\S_D&-\S_D  \\
-\frac{1}{2}I+\cK_D^{*}&0
\end{matrix}\right),
\end{equation*}
is perturbed when $\delta$ and $\omega$ are nonzero \cite{adh}. We need the following lemma, which is from \cite{adh}. 

\begin{lemma}(\!\!\cite{adh})\label{le_sk}
Consider a system of  $N$ subwavelength resonators $D_1,\dots,D_N$ in $\bR^3$. Then

(\romannumeral 1) the Laplace single layer potential $\S_D: L^2(\pD)\to H^1(\pD)$ is invertible;

(\romannumeral 2) $\ker(-\frac{1}{2}I+\cK_D^{*})=\mathrm{span}\{\psi_1,\dots,\psi_N\}$, where $\psi_l:=\S_D^{-1}[\chi_{\pD_l}]$, and  $\chi_{\pD_l}$ denotes the characteristic function of $\pD_l$, for $l=1,\dots,N$.
\end{lemma}
\begin{definition}(\!\!\cite{adh})
For a system of $N$ resonators $D_1,\dots,D_N$ in $\bR^3$, we define the capacitance matrix $\mathbb{C}=(C_{ij})_{N\times N}$ as the square matrix given by 
\begin{equation*}
C_{ij}=-\int_{\partial D_i}\psi_j\rd\sigma,\quad i,j=1,\dots,N,
\end{equation*}
and the generalized capacitance matrix $\tilde{\mathbb{C}}=(\tilde{C}_{ij})_{N\times N}$ with
$$\tilde{C}_{ij}=\frac{\delta v_b^2}{|D_i|}C_{ij}.$$
We denote the eigenvalues of matrix $\mathbb{C}$ and $\tilde{\mathbb{C}}$ as $\lambda_i$ and $\tilde{\lambda}_i$ respectively, where $i=1,\dots,N$.
\end{definition}

Let $v_i$ be defined in \eqref{eq_vi}, then we have $v_i=\mathcal{S}_D[\psi_i]$. Therefore, we can express the capacitance entries 
\begin{equation*}%\label{eq_Cij}
    C_{ij}=-\int_{\partial D_i}\frac{\partial v_j}{\partial\nu}\big|_+\rd\sigma=\int_{\Omega}\nabla v_{i}\cdot\nabla v_{j}dx.
\end{equation*}
The following proposition describes the symmetry and positivity properties of the capacitance matrix.

\begin{prop}(\!\!\cite{fa})\label{prop_C}
The capacitance matrix $\mathbb{C}$ satisfies the following properties:

(\romannumeral 1) $C_{ij}=C_{ji}$ for any $1\le i,j\le N$, and $\mathbb{C}$ is positive definite;

(\romannumeral 2) $C_{ii}>0$ for any $1\le i\le N$, while $C_{ij}<0$ for any $1\le i,j\le N$, $i\neq j$;

(\romannumeral 3) $C$ is diagonally dominant:
$C_{ii}>\sum_{j\neq i}|C_{ij}|$, for any $1\le i\le N.$
\end{prop}

\begin{lemma}(\!\!\cite{ady})\label{le_omega}
The subwavelength resonant frequencies of $N$ resonators $D_1,\dots,D_N$ are given by
$\omega_i=\sqrt{\tilde{\lambda}_i}+O(\delta),~i=1,\dots,N,~\mbox{as}~\delta\to0.$
\end{lemma} 

\begin{proof}
Suppose that $(\phi,\psi)$ is a solution to \eqref{eq_phipsi} for small $\omega(\delta)$. From the expansion \eqref{expan_s} and \eqref{expan_k}, we obtain
\begin{align}
\S_D[\phi-\psi]+k_b\S_{D,1}[\phi]-k\S_{D,1}[\psi]&=O(\omega^2),\label{eq_1}\\
\Big(-\frac{1}{2}I+\cK_D^*+k_b^2\cK_{D,2}\Big)[\phi]-\delta\Big(\frac{1}{2}I+\cK_D^*\Big)[\psi]&=O(\delta\omega+\omega^3).\label{eq_2}
\end{align}
From \eqref{eq_1} and Lemma \ref{le_sk} (\romannumeral 1), we have $\phi=\psi+O(\omega)$. By virtue of \eqref{eq_2}, we obtain $(-\frac{1}{2}I+\cK_D^*)[\phi]=O(\omega^2+\delta)$. It follows from Lemma \ref{le_sk} (\romannumeral 2) that 
$\phi=\sum_{j=1}^N\alpha_j\psi_j+O(\omega^2+\delta),$
for constants $\alpha_{j}=O(1)$, $j=1,\dots,N$. This concludes that
\begin{equation}\label{comp_psi}
\psi=\sum_{j=1}^N\alpha_j\psi_j+O(\omega+\delta).
\end{equation}

For any $\varphi\in L^2(\pD)$, we have
\begin{equation}\label{eq_prop}
\begin{aligned}
\int_{\pD_i}\Big(-\frac{1}{2}I+\cK_D^*\Big)[\varphi]\rd\sigma=0,\quad& \int_{\pD_i}\Big(\frac{1}{2}I+\cK_D^*\Big)[\varphi]\rd\sigma=\int_{\pD_i}\varphi\rd\sigma,\\
\int_{\pD_i}\cK_{D,2}[\varphi]\rd\sigma&=-\int_{D_i}\S_D[\varphi]\rd x,
\end{aligned}
\end{equation}
where $i=1,\dots,N$. Integrating \eqref{eq_2} over $\pD_i$ and using \eqref{eq_prop} yields
\begin{equation}\label{eq_3}
-k_b^2\int_{D_i}\S_D[\psi]\rd\sigma-\delta\int_{\pD_i}\psi\rd\sigma=O(\delta\omega+\omega^3).
\end{equation}
Substituting \eqref{comp_psi} into \eqref{eq_3} leads to
\begin{equation*}
-|D_i|k_b^2\alpha_{i}-\sum_{j=1}^N\delta\alpha_j\int_{\pD_i}\psi_j\rd\sigma=O(\delta^2+\delta\omega+\omega^3),\quad i=1,\dots ,N.
\end{equation*}
By employing $k_b=\omega/v_b$, 
\begin{equation}\label{z}
\sum_{j=1}^N\frac{\delta v_b^2}{|D_i|}C_{ij}\alpha_j=\omega^2\alpha_{i}+O(\delta^2+\delta\omega+\omega^3),\quad i=1,\dots ,N.
\end{equation}

Recalling the definition of $\tilde{\mathbb{C}}$, up to an error of order $O(\delta^2+\delta\omega+\omega^3)$, the equation system \eqref{z} is equivalent to the eigenvalue problem 
$\tilde{\mathbb{C}}\boldsymbol{\alpha}=\omega^2\boldsymbol{\alpha},$
where $\boldsymbol{\alpha}=(\alpha_1,\dots,\alpha_n)^{\mathrm{T}}$. So $\omega_i=\sqrt{\tilde{\lambda}_i}+O(\delta)$, $i=1,\dots,N$.
\end{proof}
\begin{remark}%\label{re}
We emphasize that the resonant frequencies $\omega_i$ are not purely real-valued and their imaginary components are included in the $O(\delta)$ term. For resonators in an unbounded domain, energy radiates into the far field, resulting in resonant frequencies with negative imaginary parts. Specifically, as demonstrated in \cite[Theorem 2.22]{adh}, the following asymptotic expansion holds: \begin{align}\label{eq_ima}
\omega_i=\sqrt{\tilde{\lambda
}_i}-\ri\delta\frac{v_b^2}{8\pi v}\frac{(\boldsymbol{\alpha}^i)^\mathrm{T}\mathbb{C}\mathbb{J}\mathbb{C}\boldsymbol{\alpha}^i}{\|\boldsymbol{\alpha}^i\|_D^2}+O(\delta^{3/2}),
\end{align}
where $\mathbb{J}$ denotes the $N\times N$ all-ones matrix, $\boldsymbol{\alpha}^i$ is the eigenvector associated to $\lambda_i$ and the norm $\|x\|_D:=(\sum_{l=1}^N|D_l|x_l^2)^{1/2}$ is used. While this paper primarily focuses on analyzing the distribution of the real parts of the resonant frequencies, we provide explicit calculations for the imaginary parts in a special case (see Remark \ref{Re_ima}). Furthermore, as noted in \cite{ady}, singularities in the leading-order terms of the asymptotic expansions \eqref{expan_s} and \eqref{expan_k} arise exclusively when the resonators are closely spaced. Higher-order expansions lie beyond the scope of this study.
\end{remark}

\section{The chain arrangement}\label{sec_chain}

In this section, we investigate the resonant frequency and its resonant mode for the chain arrangement, as depicted in Figure \ref{ca}. We will introduce a simplified matrix to fully characterize the eigenvalues and eigenvectors of the capacitance matrix $\mathbb{C}$. According to Lemma \ref{le_omega}, the subwavelength resonant frequencies $\omega_i$ are determined totally by the eigenvalues $\tilde{\lambda}_i$ of the generalized capacitance matrix $\tilde{\mathbb{C}}$. Furthermore, we will also establish the characterization of the scattered solution.

\subsection{Resonant frequency}

Let $\mathbb{I}$ denote the $N\times N$ identity matrix. Under the assumptions on $D$, the resonator volumes satisfy  $|D_1|=\dots=|D_N|=4\pi R^3/3$. Since
\begin{equation}\label{tillam}
\det (\tilde{\mathbb{C}}-\tilde{\lambda}\mathbb{I})=\Big(\frac{\delta v_b^2}{|D_1|}\Big)^N\det\Big(\mathbb{C}-\frac{|D_1|}{\delta v_b^2}\tilde{\lambda}\mathbb{I}\Big),\quad\mbox{and}~ 
\tilde{\lambda}_i=\frac{3\delta v_b^2}{4\pi R^3}\lambda_i, \quad i=1,\dots,N,
\end{equation}
it suffices to calculate the eigenvalues $\lambda_i$ of matrix $\mathbb{C}$. Based on the characteristics of the chain arrangement, we derive the following properties of the capacitance matrix $\mathbb{C}$.

\begin{prop}\label{prop_C_train}
For the capacitance matrix $\mathbb{C}$ of chain arrangement, there exist some constants $M_{ij}$ independent of $\epsilon$ such that

(\romannumeral 1) $C_{11}=C_{NN}=\pi|\log\epsilon|+M_{11}+o(1),$ $C_{ii}=2\pi|\log\epsilon|+M_{ii}+o(1)$, for $2\le i\le N-1$;

(\romannumeral 2) $C_{i,i+1}=C_{i+1,i}=-\pi|\log\epsilon|+M_{i,i+1}+o(1)$, for $1\le i\le N-1$;

(\romannumeral 3)  $C_{ij}=O(1)$, if $|i-j|\ge 2$; and  $\frac{1}{N}\big(\sum_{|i-j|\le1}M_{ij}+\sum_{|i-j|\ge2}C_{ij}\big)=M+o(1)$, where $M$ is given by \eqref{def_M}.
\end{prop}
\begin{proof}Let $\Omega_{l,l+1}^{r}$ be defined in \eqref{def_omgl}.

(\romannumeral 1) Since $v_1$ satisfies the equation \eqref{eq_vi}, by virtue of \cite[Proposition 1.5]{lz}, we have
$$C_{11}=\int_{\Omega_{1,2}^r}|\nabla v_1|^2\mathrm{d}x+M_{11}+o(1)=\pi|\log\epsilon|+M_{11}+o(1).$$
By symmetry, we have $C_{NN}=C_{11}$. For $C_{ii}$ with $2\le i\le N-1$, we obtain
$$C_{ii}=\int_{\Omega_{i-1,i}^r}|\nabla v_i|^2\mathrm{d}x+\int_{\Omega_{i,i+1}^r}|\nabla v_i|^2\mathrm{d}x+M_{ii}+o(1)=2\pi|\log\epsilon|+M_{ii}+o(1).$$

(\romannumeral 2) Due to the symmetry of matrix $\mathbb{C}$, it is sufficient to calculate $C_{i,i+1}$. Referring to \cite{lly}, we have 
\begin{align*}
C_{i,i+1}=-\int_{\bar{\Omega}_{i,i+1}^r\cap \partial D_{i}}\!\!\frac{\partial v_{i+1}}{\partial\nu}\rd\sigma+M_{i,i+1}+o(1)
&=\int_{\bar{\Omega}_{i,i+1}^r\cap \partial D_{i+1}}\!\!\frac{\partial v_{i+1}}{\partial\nu}\rd\sigma+M_{i,i+1}+o(1)\\
&=-\pi|\log\epsilon|+M_{i,i+1}+o(1),
\end{align*}
for $1\le i\le N-1$. 

(\romannumeral 3) It follows from \cite[Lemma 2.5]{bly2} that $C_{ij}=O(1)$, if $|i-j|\ge 2$. By the definition of $v_i$ and using Green's formula, we have 
$$M=-\frac{1}{N}\sum_{i=1}^{N}\int_{\pD}\frac{\partial v_i}{\partial\nu}\Big|_+\rd\sigma=\frac{1}{N}\sum_{i,j=1}^NC_{ij}=\frac{1}{N}\big(\sum_{|i-j|\le1}M_{ij}+\sum_{|i-j|\ge2}C_{ij}\big)+o(1).$$
\end{proof}

We begin by determining the eigenvalues of magnitude  $O(|\log\epsilon|)$. For this purpose, define $$\rho(\epsilon):=\pi|\log\varepsilon|~\quad\mbox{and}~ \lambda:=a\rho(\epsilon).$$ Let $\mathbb{A}$ denote an $N\times N$ tridiagonal Toeplitz matrix with perturbations of magnitude $-1$ added to the diagonal entries at the top-left and bottom-right corners. Namely,
\begin{equation}\label{def_A}
\mathbb{A}=\left(\begin{array}{ccccc}
1 & -1 &  &  &  \\
-1 & 2 & \ddots &  &  \\
& \ddots & \ddots & \ddots &  \\
&  & \ddots & 2 & -1 \\
&  &  & -1 & 1
\end{array}\right).
\end{equation}
By Proposition \ref{prop_C_train}, we obtain
\begin{equation}\label{def_Fa}
F(a):=\det(\mathbb{C}-a\rho(\epsilon)\mathbb{I})=f_N(a)\rho(\epsilon)^{N}+f_{N-1}(a)\rho(\epsilon)^{N-1}+O(\rho(\epsilon)^{N-2}),
\end{equation}
where 
\begin{equation}\label{def_detA}
f_N(a)=\det(\mathbb{A}-a\mathbb{I})
\end{equation}
is a polynomial in $a$ of degree $N$ and $f_{N-1}(a)$ is a polynomial in $a$ of degree $N\!-\!1$. Our goal is to determine the roots of $F(a)$, denoted by $\tilde{a}_N^1,\dots,\tilde{a}_N^N$.

For $a\le0$, Proposition \ref{prop_C_train} implies that $\mathbb{C}-a\rho(\epsilon)\mathbb{I}$ is a diagonally dominant matrix, ensuring $F(a)=\det(\mathbb{C}-a\rho(\epsilon)\mathbb{I})>0$. For $a\ge 4$, the matrix $-(\mathbb{C}-a\rho(\epsilon)\mathbb{I})$ becomes diagonally dominant. In this case, $F(a)=(-1)^N\det(-\mathbb{C}+a\rho(\epsilon)\mathbb{I})$ yielding $F(a)>0$ if $N$ is even, and $F(a)<0$ if $N$ is odd. Consequently, we restrict our analysis to $0<a<4$ for the remainder of this section. Denote 
\begin{equation}\label{def_gn}
g_N(a):=\left|\begin{array}{ccccc}
2-a & -1 &  &  &  \\
-1 & 2-a & \ddots &  &  \\
& \ddots & \ddots & \ddots &  \\
&  & \ddots & 2-a & -1 \\
&  &  & -1 & 2-a
\end{array}\right|,
\end{equation}
then 
$$f_N(a)=g_N(a)-2g_{N-1}(a)+g_{N-2}(a).$$

\begin{lemma}
Let $f_N(a)$ be defined in \eqref{def_detA} for $a\in(0,4)$. Then
\begin{equation}\label{eq_fn}
f_N(a)=-\frac{a\sin N\theta}{\sin\theta},\quad\mbox{where}~\theta=\arccos \frac{2-a}{2}.
\end{equation}

\end{lemma}
\begin{proof}
Let $\cos\theta=(2-a)/2$, $a\in(0,4)$. Then $g_N(a)-2\cos\theta g_{N-1}(a)+g_{N-2}(a)=0$. By recursion, we obtain 
$g_N(a)=\frac{\sin (N+1)\theta}{\sin\theta}.$
Consequently,
\begin{align*}
f_N(a)=\frac{\sin(N+1)\theta-2\sin N\theta+\sin(N-1)\theta}{\sin\theta}=\frac{(2\cos\theta-2)\sin N\theta}{\sin\theta}=-\frac{a\sin N\theta}{\sin\theta}.
\end{align*}
\end{proof}
\begin{remark}
Since
$$\sin N\theta=\sum_{k=0}^{\alpha(N)-1}\frac{(-1)^k N!}{(2k+1)!(n-2k-1)!}\cos^{N-2k-1}\theta\sin^{2k+1}\theta,\quad\mbox{for}~N\geq2,$$
where $\alpha(N)$ is defined in \eqref{def_alpN}. From \eqref{eq_fn}, we rewrite $f_N(a)$ into polynomial form
\begin{equation}\label{eq_fn_pol}
f_N(a)=-\frac{a}{2^{N-1}}\sum_{k=0}^{\alpha(N)-1}\frac{(-1)^k N!}{(2k+1)!(n-2k-1)!}a^k(4-a)^k(2-a)^{N-2k-1}.
\end{equation}
From \eqref{eq_fn_pol}, it is easy to see that the coefficient of $a^N$ is $(-1)^N$.
\end{remark}
Let $f_N(a)=0$. By $\sin N\theta=0$, we have $\theta=(i-1)\pi/N$, $i=2,\dots,N$. Consequently, $a=2\big(1-\cos((i-1)\pi/N)\big)$. If we denote $\theta_i$ and $a_N^i$ as in \eqref{def_anin}, and let $a_N^1=0$, then combining this with \eqref{eq_fn_pol}, we derive the following equation:
\begin{equation}\label{fn_de}
f_N(a)=(-1)^N(a-a_N^1)(a-a_N^2)\cdots(a-a_N^N),\quad a\in(0,4).
\end{equation}
Due to our determination of the first coefficient and $N$ roots of $f_N(a)$, \eqref{fn_de} holds for all $a\in\bR$. Denote 
\begin{equation}\label{def_fj}
f^i(a):=\frac{f_N(a)}{a-a_N^i},\quad 1\le i\le N.
\end{equation}
\begin{lemma}\label{le_lam2toN}
Let $\tilde{\lambda}_i$ be defined in \eqref{tillam}, where $1\le i\le N$. There holds
$$\tilde{\lambda}_i=\frac{3\delta v_b^2a_N^i}{4R^3}|\log\epsilon|+O(\delta).$$
\end{lemma}
\begin{proof}
For $2\le i\le N$, combining \eqref{def_Fa} and \eqref{fn_de} yields
\begin{align*}
F(a_N^i\!+\!C_i\rho(\epsilon)^{-1})&\!\!=\!C_if^i(a_N^i\!+\!C_i\rho(\epsilon)^{-1})\rho(\epsilon)^{N-1}\!\!\!+\!\!f_{N-1}(a_N^i\!+\!C_i\rho(\epsilon)^{-1})\rho(\epsilon)^{N-1}\!\!\!+\!O(\rho(\epsilon)^{N-2})\\
&\!\!=\!\big(C_if^i(a_N^i\!+\!C_i\rho(\epsilon)^{-1})\!+\!f_{N-1}(a_N^i\!+\!C_i\rho(\epsilon)^{-1})\big)\rho(\epsilon)^{N-1}\!+\!O(\rho(\epsilon)^{N-2}),
\end{align*}
and
\begin{align*}
F(a_N^i\!-C_i\rho(\epsilon)^{-1})\!=\!\big(\!-C_if^i(a_N^i\!-\!C_i\rho(\epsilon)^{-1})\!+\!f_{N-1}(a_N^i\!-\!C_i\rho(\epsilon)^{-1})\big)\rho(\epsilon)^{N-1}\!+\!O(\rho(\epsilon)^{N-2}),
\end{align*}
where $C_i$ are constants, to be determined later. Without loss of generality, assume $C_i<\rho(\epsilon)$. From \eqref{fn_de}, it follows that $|f^i(a_N^i\pm C_i\rho(\epsilon)^{-1})|\to \gamma_i$ for some constant $\gamma_i>0$ as $\epsilon\to 0$. Observe that $C_if^i(a_N^i\pm C_i\rho(\epsilon)^{-1})$ is a polynomial in $C_i$ of degree $N$, whereas $f_{N-1}(a_N^i\pm C_i\rho(\epsilon)^{-1})$ is a polynomial in $C_i$ of degree $N-1$. We then select $C_i$ sufficiently large such that 
$$\big(C_if^i\!(a_N^i+C_i\rho(\epsilon)^{-1})+\!f_{N-1}(a_N^i+C_i\rho(\epsilon)^{-1})\!\big)\!\big(\!-C_if^i\!(a_N^i-C_i\rho(\epsilon)^{-1})+f_{N-1}(a_N^i-C_i\rho(\epsilon)^{-1})\!\big)\!\!<\!0.$$
Since $\epsilon<<1$, we have $F(a_N^i-C_i\rho(\epsilon)^{-1})F(a_N^i+C_i\rho(\epsilon)^{-1})<0$. By the intermediate value theorem, there exists $\tilde{a}_N^i$ in the interval $a_N^i-C_i\rho(\epsilon)^{-1}<\tilde{a}_N^i<a_N^i+C_i\rho(\epsilon)^{-1}$ such that $F(\tilde{a}_N^i)=0$. Thus,
$a_N^i\rho(\epsilon) -C_i<\lambda_i<a_N^i\rho(\epsilon) +C_i$.
Hence, substituting into the scaling relation for $\tilde{\lambda}_i$, we obtain
$$\frac{3\delta v_b^2}{4\pi R^3}(a_N^i\rho(\epsilon)-C_i)<\tilde{\lambda}_i<\frac{3\delta v_b^2}{4\pi R^3}(a_N^i\rho(\epsilon)+C_i).$$

We now prove that there exists a constant $C_1$ such that $0<\tilde{\lambda}_1<C_1\delta$. First, observe that 
$F(0)=\det \mathbb{C}>0$. If $0<a<a_N^2$, we have $0<\arccos \frac{2-a}{2}=\theta<\pi/N$, then
\begin{equation}\label{111}
f_N(a)=-\frac{a\sin N\theta}{\sin\theta}<0,\quad a\in(0,a_N^2).
\end{equation}
Let $a=C_1\rho(\epsilon)^{-1}$, where $C_1$ is a constant to be determined later, and ensure $0<a<a_N^2$. Substituting into $F(a)$, we obtain
\begin{align*}
F(C_1\rho(\epsilon)^{-1})=\big(C_1f^1(C_1\rho(\epsilon)^{-1})+f_{N-1}(C_1\rho(\epsilon)^{-1})\big)\rho(\epsilon)^{N-1}+O(\rho(\epsilon)^{N-2}).
\end{align*} 
By \eqref{111}, $C_1f^1(C_1\rho(\epsilon)^{-1})=f_N(C_1\rho(\epsilon)^{-1})\rho(\epsilon)<0$. Similarly, choose $C_1$ sufficiently large such that $C_1f^1(C_1\rho(\epsilon)^{-1})+f_{N-1}(C_1\rho(\epsilon)^{-1})<0$. This implies
$F(0)F(C_1\rho(\epsilon)^{-1})<0$. By the intermediate value theorem, there exists $\tilde{a}_N^1\in(0,C_1\rho(\epsilon)^{-1})$ satisfying $F(\tilde{a}_N^1)=0$. Hence,
$0<\tilde{\lambda}_1=\frac{3\delta v_b^2\tilde{a}_N^1}{4\pi R^3}\rho (\epsilon)<C_1\delta,$
for some constant $C_1$. This completes the proof.
\end{proof}

We next derive the specific expression for the principal term of $\tilde{\lambda}_1$. Based on our previous derivations, we know that $\lambda_1$ is of $O(1)$ order.

\begin{lemma}\label{le_omg1_chain}
Let $\tilde{\lambda}_1$ be defined in \eqref{tillam}. Then,
$$\tilde{\lambda}_1=\frac{3v_b^2M}{4\pi R^3}\delta\big(1+o(1)\big),$$
where $M$ is defined in \eqref{def_M}.
\end{lemma}

\begin{proof}
By applying Proposition \ref{prop_C_train} again, we expand the determinant $\det(\mathbb{C}-\lambda\mathbb{I})$ as
\begin{equation*}
\tilde{f}(\lambda):=\det(\mathbb{C}-\lambda\mathbb{I})=\tilde{f}_0(\lambda)\rho(\epsilon)^N+\tilde{f}_{1}(\lambda)\rho(\epsilon)^{N-1}+\tilde{f}_2(\lambda)\rho(\epsilon)^{N-2}+O(\rho(\epsilon)^{N-3}),
\end{equation*}
where each $\tilde{f}_k(\lambda)$ is a polynomial in $\lambda$ of degree $k$ for $k=0,1,2$. It is easy to verify that $\tilde{f}_0(\lambda)=f_N(0)=0$. Direct calculation yields 
\begin{equation*}
\tilde{f}_1(\lambda)=\left|\begin{array}{cccccc}
\tilde{f}_{11}(\lambda)& \tilde{f}_{12}(\lambda) & \tilde{f}_{13}(\lambda) &\cdots &\tilde{f}_{1,N-1}(\lambda) &\tilde{f}_{1N}(\lambda)\\
-1 & 2  &-1&\cdots  &0 & 0 \\
0&-1&2&\cdots&0&0\\
 \vdots&\vdots  &\vdots& &\vdots&\vdots \\
0&0&0&\cdots  & 2 & -1 \\
0& 0 &0&  \cdots& -1 & 1
\end{array}\right|,
\end{equation*}
with $\tilde{f}_{1i}(\lambda)=\tilde{M}_i-\lambda$, where $\tilde{M}_i=\sum_{|j-i|\le1}M_{ij}+\sum_{|j-i|\ge2}C_{ij}$ for $1\le i\le N$. Observe that $\tilde{f}_1(\frac{1}{N}\sum_{i=1}^N\tilde{M}_{i})=0$, which implies $$\tilde{f}\big(\frac{1}{N}\sum_{i=1}^N\tilde{M}_{i}+\rho(\epsilon)^{-s}\big)\tilde{f}\big(\frac{1}{N}\sum_{i=1}^N\tilde{M}_{i}-\rho(\epsilon)^{-s}\big)<0, \quad\mbox{for}~ s\in(0,1).$$ Consequently, we obtain $\lambda_1=\frac{1}{N}\sum_{i=1}^N\tilde{M}_{i}+O\big(\rho(\epsilon)^{-s}\big)$. By virtue of Proposition \ref{prop_C_train} (\romannumeral 3), we have $\frac{1}{N}\sum_{i=1}^N\tilde{M}_{i}=M+o(1)$. Thus, the proof is completed by combining this with \eqref{tillam}.
\end{proof}

\begin{proof}[Proof of Theorem \ref{Th_chain}]
By applying Lemmas \ref{le_omega}, \ref{le_lam2toN}, and \ref{le_omg1_chain}, the theorem follows.
\end{proof}

\subsection{Resonant mode}

The resonant mode is known to remain approximately constant within each resonator. If these constants differ, then as the two resonators are brought closer together, the field gradient between them will blow up at a rate of $O(\epsilon^{-1})$. Recalling definition \eqref{def_eigenmode} and the decomposition \eqref{comp_psi}, it suffices to compute the eigenvector of matrix $\mathbb{C}$ associated with the eigenvalue $\lambda_i$, where $i=1,\dots,N$.

Let $\tilde{\boldsymbol{\alpha}}^i:=\left(
\tilde{\alpha}_1^i,\tilde{\alpha}_2^i,\dots,\tilde{\alpha}_N^i
\right)^{\mathrm{T}}.$
By virtue of \eqref{tillam}, the eigenvalue equation
$\tilde{\mathbb{C}}\tilde{\boldsymbol{\alpha}}^i=\tilde{\lambda}_i\tilde{\boldsymbol{\alpha}}^i$
is equivalent to 
\begin{equation}\label{zzz}
\mathbb{C}\tilde{\boldsymbol{\alpha}}^i=\lambda_i\tilde{\boldsymbol{\alpha}}^i.
\end{equation}
Set $\mathbb{A}_i:=\mathbb{A}-a_N^i\mathbb{I}$. Using Proposition \ref{prop_C_train}, we rewrite \eqref{zzz} as 
$\rho(\epsilon)  \mathbb{A}_i\tilde{\boldsymbol{\alpha}}^i+O(1)=0,$
which implies
$\mathbb{A}_i\tilde{\boldsymbol{\alpha}}^i=O(\rho(\epsilon)^{-1})$.

Let 
\begin{equation}\label{alpi}
\boldsymbol{\alpha}^i=(\alpha^i_1,\dots,\alpha^i_N)^\mathrm{T}.
\end{equation}
If $i=1$, we have $a_N^1=0$, implying $\mathbb{A}_1=\mathbb{A}$. From the definition of $\mathbb{A}$ in \eqref{def_A}, it follows that the vector $\boldsymbol{\alpha}^1=(1,1,\cdots,1)^\mathrm{T}$ spans the solution space of the equation $\mathbb{A}_1\boldsymbol{\alpha}^1=0$. For $2\le i\le N$, by recursion, we obtain
\begin{equation}\label{def_alpji}
\begin{aligned}
\alpha_l^i=\frac{\sin l\theta_i-\sin(l-1)\theta_i}{\sin\theta_i},
\end{aligned}
\end{equation}
where $1\le l\le N$, and the vector $\boldsymbol{\alpha}^i$ spans the solution space of $\mathbb{A}_i\boldsymbol{\alpha}^i=0$. This yields the eigenvectors
\begin{equation}\label{tila}
\tilde{\boldsymbol{\alpha}}^i=\boldsymbol{\alpha}^i+O(\rho(\epsilon)^{-1}),\quad i=1,\dots,N.
\end{equation}

The resonant modes are given by
\begin{equation}\label{def_ui}
u^i(x)=\S_D[\phi^i](x)+O(\omega_i),\quad i=1,\dots,N.
\end{equation}
By using \eqref{comp_psi} in Lemma \ref{le_omega}, we have 
\begin{equation}\label{def_phi_i} \phi^i:=\sum_{l=1}^N\tilde{\alpha}_{l}^i\psi_l.
\end{equation}
Note that $v_l=\S_D(\psi_l)$, then
$\S_D[\phi^i](x)=\sum_{l=1}^N\tilde{\alpha}_{l}^i\S_D(\psi_l)+O(\omega_i)
=\sum_{l=1}^N\tilde{\alpha}_{l}^iv_l+O(\omega_i).$
Hence, 
\begin{equation}%\label{ui}
u^i(x)=\tilde{\alpha}_l^i+O(\omega_i)\quad\text{on}\ \partial D_l, 1\le i,l\le N,
\end{equation}
where $\tilde{\alpha}_{l}^i$ is defined in \eqref{tila}. 

For $2\le i\le N$ and $1\le l\le N-1$, \eqref{def_alpji} yields
\begin{equation}
\tilde{\alpha}_{l+1}^i-\tilde{\alpha}_{l}^i=-\frac{2\big(1-\cos\theta_i\big)}{\sin\theta_i}\sin l\theta_i+O(\rho(\epsilon)^{-1}).
\end{equation}
When $l(i-1)=Nt$ holds for some $t\in\bN_+$, we find $\tilde{\alpha}_{l+1}^i-\tilde{\alpha}_{l}^i=O(\rho(\epsilon)^{-1})$. By \cite[Theorem 4.1]{lz}, we obtain 
\begin{equation}\label{nabu}
\|\nabla u^1\|_{L^\infty(\Omega)}, \|\nabla u^i\|_{L^\infty(\Omega_{l,l+1}^{r})}\lesssim\frac{1}{\epsilon|\log\epsilon|},
\end{equation}
where $i$ and $l$ satisfy $l(i-1)=Nt$ for $t\in\bN_+$.

\begin{proof}[Proof of Theorem \ref{Th_mode} and Remark \ref{Re}]
By combining with \eqref{def_alpji}--\eqref{nabu}, we complete the proof.
\end{proof}

\subsection{Scattered solution} 

By Theorem \ref{Th_chain} and Theorem \ref{Th_mode}, the scattered field of a finite chain of resonators under an incident plane wave $u^{in}$ is characterized as follows.
\begin{theorem}
Let $\mathbb{V}=(\boldsymbol{\alpha}^1,\cdots,\boldsymbol{\alpha}^N)^\mathrm{T}$ be defined in \eqref{alpi}, and let $u^i$ denote the subwavelength resonant modes defined in \eqref{def_ui} for $i=1,\dots,N$. Then the solution $u$ to the scattering problem \eqref{eq_helm} with incident wave $u^{in}$ at frequency $\omega$ satisfies, for $x\in\Omega$,
\begin{equation}\label{6}
u(x)=u^{in}-\S_D^k[\S_D^{-1}[u^{in}]]+\sum_{i=1}^Na_iu^i+O(\omega), 
\end{equation}
where $(a_1,\dots,a_N)^{\mathrm{T}}$ solves the system 
\begin{equation}\label{5}
\mathbb{V}\left(\begin{matrix}
\omega_1^2\!-\omega^2&&\\
&\!\ddots\!&\\
&&\omega_N^2\!-\omega^2
\end{matrix}\right)\left(\begin{matrix}
a_1\\
\vdots\\
a_N
\end{matrix}\right)\!=\!-\frac{3\delta v_b^2}{4\pi R^3}\left(\begin{matrix}
\int_{\pD_1}\!\S_D^{-1}[u^{in}]\rd \sigma\\
\vdots\\
\int_{\pD_N}\!\S_D^{-1}[u^{in}]\rd \sigma
\end{matrix}\right)\!+O(\delta^{2-\beta}\!+\delta^{1-\beta}\omega^2\!+\omega^3).
\end{equation}
\end{theorem}

\begin{proof}
Let $(\phi,\psi)$ be solutions to the scattering problem \eqref{eq_phipsi}. From the boundary integral formulation, we derive
\begin{align}
\S_{D}[\phi-\psi]&=u^{in}+O(\omega),\label{11}\\ 
\big(-\frac{1}{2}\cI+\cK_D^*+k_b^2\cK_{D,2}\big)[\phi]-\delta\big(\frac{1}{2}\cI+\cK_D^*\big)[\psi]&=O(\delta\omega+\omega^3).\label{22}
\end{align}
From \eqref{11}, we obtain
\begin{equation}\label{0}
\psi=\phi-\S_D^{-1}[u^{in}]+O(\omega).
\end{equation}
Substituting this into \eqref{22} yields 
\begin{equation}\label{1}
\big(-\frac{1}{2}\cI+\cK_D^*+k_b^2\cK_{D,2}\big)[\phi]-\delta\big(\frac{1}{2}\cI+\cK_D^*\big)[\phi]=-\delta\big(\frac{1}{2}\cI+\cK_D^*\big)\S_D^{-1}[u^{in}]+O(\delta\omega+\omega^3).
\end{equation}
We perform a decomposition 
\begin{equation}\label{2}
\phi=\sum_{i=1}^Na_i\phi^i+O(\delta+\omega^2),
\end{equation}
where $\phi^i$ are defined in \eqref{def_phi_i}, and $a_i$ are constants of order $O(1)$, see \cite[Theorem 6.1]{ady}. By combining \eqref{1} with \eqref{2}, it follows that
\begin{equation}\label{3}
-\sum_{i=1}^N\frac{\omega^2}{v_b^2}\int_{\pD_l}\S_{D}\big[a_i\phi^i\big]\rd x-\delta\sum_{i=1}^N\int_{\pD_l}a_i\phi^i\rd \sigma=-\delta\int_{\pD_l}\S_D^{-1}[u^{in}]\rd \sigma+O(\delta\omega+\omega^3).
\end{equation}
By construction, $\phi^i$ satisfy \eqref{eq_3} exactly at resonant frequencies $\omega=\omega_i$. Thus, 
\begin{equation}\label{4}
-\sum_{i=1}^N\frac{\omega_i^2}{v_b^2}\int_{D_l}\S_{D}\big[a_i\phi^i\big]\rd x-\delta\sum_{i=1}^N\int_{\pD_l}a_i\phi^i\rd \sigma=O(\delta\omega+\omega^3).
\end{equation}
Subtracting \eqref{4} from \eqref{3} gives
$$\sum_{i=1}^N\frac{a_i(\omega_i^2-\omega^2)|D_l|\alpha_l^i}{v_b^2}=-\delta\int_{\pD_l}\S_D^{-1}[u^{in}]\rd \sigma+O(\delta^{2-\beta}+\delta^{1-\beta}\omega^2+\omega^3),$$
which is equivalent to \eqref{5}. Combining \eqref{us}, \eqref{0} and \eqref{1} yields \eqref{6}.
\end{proof}

\section{The ring arrangement}\label{sec_ring}

In this section, we analyze a ring-shaped configuration of $N$ resonators. Unlike the linear chain case, we will show that the function $f_N(a)$ associated with this geometry admits multiple roots (i.e., not all roots are simple). To account for this multiplicity, we extend the analytical framework developed in Section \ref{sec_chain}, establishing a comprehensive approach to characterize the eigenvalues of the capacitance matrix.

\subsection{Resonant frequency} Similarly as Proposition \ref{prop_C_train}, we derive the following properties for the capacitance matrix in the ring configuration.

\begin{prop}\label{prop_C_ring}
For the capacitance matrix $\mathbb{C}$ of a ring arrangement, we have

(\romannumeral 1) $\mathbb{C}$ is a circulant matrix, satisfying the periodicity conditions: $C_{ij}=C_{i-1,j-1}$ and $C_{i1}=C_{i-1,N}$ for $i,j=2,\dots,N$;

(\romannumeral 2) The leading entries satisfy: $C_{11}=2\rho(\epsilon)+O(1)$, $C_{12},C_{1N}=-\rho(\epsilon)+O(1)$;

(\romannumeral 3) For $N\ge4$, the off-diagonal entries decay as: $C_{1j}=O(1)$, $j=3,\dots,N-1$.
\end{prop}

Let $\hat{\mathbb{A}}$ be an $N$-th order circulant matrix, given by
\begin{equation*}
\hat{\mathbb{A}}=\left(\begin{array}{ccccc}
2 & -1 &\cdots  &  0&-1  \\
-1 & 2 & \cdots & 0 &0  \\
\vdots& \vdots & & \vdots & \vdots \\
0&0  & \cdots & 2 & -1 \\
-1& 0 & \cdots & -1 & 2
\end{array}\right),
\end{equation*}
then there holds
\begin{equation}\label{def_Far}
F(a)=f_N(a)\rho(\epsilon)^N+f_{N-1}(a)\rho(\epsilon)^{N-1}+O(\rho(\epsilon)^{N-2}),
\end{equation}
where 
$f_N(a)=\det(\hat{\mathbb{A}}-a\mathbb{{I}})$. 

Similarly as before, we restrict our consideration to $0<a<4$. One can verify that $f_N(a)=g_N(a)-g_{N-2}(a)-2$, where $g_N(a)$ is defined in \eqref{def_gn}. This yields
$$f_N(a)=\frac{\sin(N+1)\theta}{\sin\theta}-\frac{\sin(N-1)\theta}{\sin\theta}-2=2(\cos N\theta-1),$$
where $\theta=\arccos(2-a)/2$ for $a\in(0,4)$. Setting $f_N(a)=0$, we have $\theta=2(t-1)\pi/N$ for $t=2,\dots,\alpha(N)$, with $\alpha(N)$ defined in \eqref{def_alpN}. Consequently, 
$$a=2\Big(1-\cos \frac{2(t-1)\pi}{N}\Big),\quad t=2,\dots,\alpha(N).$$ 

Let $\hat{a}_N^t$ and $\hat{\theta}_t$ be defined as in \eqref{def_hai} for $t=2,\dots,\alpha(N)$. Observing that
$$(\cos N\hat{\theta}_t-1)=0,\quad\frac{d}{d\theta}(\cos N\hat{\theta}_t-1)=0,\quad \frac{d^2}{d\theta^2}(\cos N\hat{\theta}_t-1)\neq0,$$
and $\hat{a}_N^t$ are double roots of $f_N(a)$. Thus, accounting for multiplicities, $f_N(a)$ has $2(\alpha(N)-1)$ roots. For odd $N$, $\hat{a}_N^1=0$ is a simple root of $f_N(a)$. For even $N$, both $\hat{a}_N^1=0$ and $\hat{a}_N^{\alpha(N)+1}=4$ are simple roots of $f_N(a)$. Ignoring the negative sign, we factorize $f_N(a)$ as
\begin{equation}\label{far}
f_N(a)=\begin{cases}
(a-\hat{a}_N^1)(a-\hat{a}_N^2)^2\cdots(a-\hat{a}_N^{(N+1)/2})^2, &\ N \text{ is odd},\\
(a-\hat{a}_N^1)(a-\hat{a}_N^2)^2\cdots(a-\hat{a}_N^{N/2})^2(a-\hat{a}_N^{N/2+1}), &\ N \text{ is even}.
\end{cases}
\end{equation}

For simplicity, we focus on the odd $N$ case hereafter, as the even case follows analogously. Following the notation in \eqref{def_fj}, we denote 
\begin{equation}\label{def_fNt_r}
f^1(a):=\frac{f_N(a)}{a-\hat{a}_N^1},\quad\text{and}~ f^t(a):=\frac{f_N(a)}{(a-\hat{a}_N^t)^2},\quad\mbox{for}~t=2,\dots,(N+1)/2.
\end{equation}

\begin{lemma}\label{le_lam1_r}
Let $\tilde{\lambda}_1$ be defined in \eqref{tillam}. Then $\tilde{\lambda}_1=O(\delta)>0$.
\end{lemma}

\begin{proof}
Since $\hat{a}_N^1=0$ is a simple root of $f_N(a)$, the result follows directly by applying the methodology of Theorem \ref{Th_chain} to the ring configuration.
\end{proof}

From Lemma \ref{le_lam1_r}, we conclude that $\tilde{a}_N^1=O(\rho(\epsilon)^{-1})$ is the smallest root of $F(a)$. The next lemma identifies the intervals containing the remaining roots of $F(a)$.

\begin{lemma}\label{le_inter_r}
Let $F(a)$ be defined as in \eqref{def_Far}. The remaining roots $\tilde{a}_N^i$ of $F(a)$, for $2\le i\le N$, lie within the union of intervals: $$\cup_{t=2}^{(N+1)/2}\big(\hat{a}_N^{t}-C/\rho (\epsilon)^{1/2},\hat{a}_N^{t}+C/\rho (\epsilon)^{1/2}\big),$$ 
for some constant $C>0$.
\end{lemma}
\begin{proof}
Let 
$a=\hat{a}_N^t\pm C_t/\rho(\epsilon)^{1/2}, 2\le t\le (N+1)/2,$
for some constants $C_t$.
It follows from \eqref{def_Far}, \eqref{far} and \eqref{def_fNt_r} that
\begin{align*}
F(\hat{a}_N^t\!\pm\! C_t/\rho(\epsilon)^{1/2})\!\!=\!\!\Big(\!C_t^{2}f^t(\hat{a}_N^t\!\pm\! C_t/\rho(\epsilon)^{1/2})\!+\!\!f_{N-1}(\hat{a}_N^t\!\pm\! C_t/\rho(\epsilon)&^{1/2})\!\Big)\rho(\epsilon)^{N-1}\!+\!O(\rho(\epsilon)^{N-2}).
\end{align*}
we have $C_t^{2}f^t(\hat{a}_N^t\pm C_t/\rho(\epsilon)^{1/2})$ is a degree-$N$ polynomial in $C_t$, $f_{N-1}(\hat{a}_N^t\pm C_t/\rho(\epsilon)^{1/2})$ is a degree-$N-1$ polynomial in $C_t$, and $|f^t(\hat{a}_N^t\pm C_t/\rho(\epsilon)^{1/2})|\to \gamma_t$ for some $\gamma_t>0$ as $\epsilon\to 0$. By selecting $C_t$ sufficiently large such that $|F(a)|>\alpha_t$ for some $\alpha_t>0$ in the gaps
$$\hat{a}_N^{t-1}+C_{t-1}/\rho (\epsilon)^{1/2}\le a \le \hat{a}_N^t- C_t/\rho(\epsilon)^{1/2}.$$

Similarly, choosing $C_{(N+3)/2}$ large enough ensures $|F(a)|>\alpha_{(N+3)/2}$ for some $\alpha_{(N+3)/2}>0$ for
$a\ge \hat{a}_N^{(N+1)/2}+ C_{(N+3)/2}/\rho(\epsilon) ^{1/2}.$
Thus, all roots $\tilde{a}_N^i$ ($2\le i\le N$) must lie within 
$\cup_{t=2}^{(N+1)/2}\big(\hat{a}_N^t-C/\rho(\epsilon)^{1/2},\hat{a}_N^t+C/\rho(\epsilon)^{1/2}\big),$
where $C=\max\{C_2,\dots,C_{(N+3)/2}\}$.
\end{proof}

\begin{lemma}\label{le_atmost_r}
Let $F(a)$ be defined in \eqref{def_Far}. Then, when accounting for multiplicity, $F(a)$ has at most two roots within the interval $\big(\hat{a}_N^t-C/\rho(\epsilon)^{1/2},\hat{a}_N^t+C/\rho(\epsilon)^{1/2}\big)$, for $2\le t\le (N+1)/2$.
\end{lemma}

\begin{proof}
Since $\hat{a}_N^t$ is a simple root of $f_N'(a)$, arguments analogous to those in Lemma  \ref{le_lam2toN} show that $f_N'(a)$ has exactly one root within  
$\big(\hat{a}_N^t-C/\rho(\epsilon)^{1/2},\hat{a}_N^t+C/\rho(\epsilon)^{1/2}\big)$. Consequently, $f_N(a)$ admits at most two roots in this interval. 
\end{proof}

\begin{corollary}\label{le_exact}
Let $F(a)$ be defined in \eqref{def_Far}. When accounting for multiplicity $F(a)$ has exactly two roots within
$\big(\hat{a}_N^t-C/\rho(\epsilon)^{1/2},\hat{a}_N^t+C/\rho(\epsilon)^{1/2}\big)$, for $2\le t\le (N+1)/2$.
\end{corollary}

\begin{proof}
By Proposition \ref{prop_C}, the capacitance matrix $\mathbb{C}$ is symmetric and positive definite, ensuring all $N$ eigenvalues of $\mathbb{C}$ are real. Thus, $F(a)$ must have $N$ real roots (counting multiplicity). 

Assume for contradiction that an interval 
$(\hat{a}_N^t-C/\rho(\epsilon)^{1/2},\hat{a}_N^t+C/\rho(\epsilon)^{1/2})$, 
contains $r$ roots of $F(a)$, where $0\le r<2$. By the pigeonhole principle, there must exist another interval $\big(\hat{a}_N^{t'}-C/\rho (\epsilon)^{1/2},\hat{a}_N^{t'}+C/\rho(\epsilon) ^{1/2}\big)$ for some $t'\neq t$ ($2\le t'\le (N+1)/2$) containing at least three roots, contradicting Lemma \ref{le_atmost_r}.
\end{proof}

Through refined analysis, we establish a tighter bound for the roots within the interval $\big(\hat{a}_N^t-C/\rho(\epsilon)^{1/2},\hat{a}_N^t+C/\rho(\epsilon)^{1/2}\big)$.

\begin{lemma}\label{le_equal_r}
Let $F(a)$ be defined as in \eqref{def_Far}. For $2\le t\le (N+1)/2$, $F(a)$ has two roots within the refined interval
$\big(\hat{a}_N^t-C/\rho(\epsilon),\hat{a}_N^t+C/\rho(\epsilon)\big).$
\end{lemma}

\begin{proof}
Let $\tilde{a}_N^t$ denote a root of $F(a)$ in 
$(\hat{a}_N^t-C/\rho(\epsilon)^{1/2},\hat{a}_N^t+C/\rho(\epsilon)^{1/2})$. Assume $\tilde{a}_N^t=\hat{a}_N^t+\gamma
/\rho(\epsilon)^{1/s}$ for constants $s\in(1, 2]$ and $\hat{C}$ such that $1/
\hat{C}\le\gamma\le \hat{C}$. By definition, $\tilde{a}_N^t\rho(\epsilon)$ is an eigenvalue of capacitance matrix $\mathbb{C}$. Let $\boldsymbol{\alpha}^t$ be the eigenvector of matrix $\mathbb{C}$ associated with eigenvalue $\tilde{a}_N^t\rho(\epsilon)$. Thus, we have
$\mathbb{C}\boldsymbol{\alpha}^t=\tilde{a}_N^t\rho(\epsilon)\boldsymbol{\alpha}^t.$
This implies
\begin{equation}\label{ar}
(\mathbb{C}-\hat{a}_N^t\rho(\epsilon)\mathbb{I})\boldsymbol{\alpha}^t=\gamma\rho(\epsilon)^{1-\frac{1}{s}}\boldsymbol{\alpha}^t.
\end{equation}

Define
\begin{equation}\label{def_At}
\hat{\mathbb{A}}_t:=\hat{\mathbb{A}}-\hat{a}_N^t\mathbb{I}.
\end{equation}
From Proposition \ref{prop_C_ring} and \eqref{ar},
\begin{equation}\label{br}
\rho(\epsilon)\hat{\mathbb{A}}_t\boldsymbol{\alpha}^t=(\hat{\mathbb{C}}+\gamma\rho(\epsilon)^{1-\frac{1}{s}}\mathbb{I})\boldsymbol{\alpha}^t,
\end{equation}
where $\hat{\mathbb{C}}$ is an $N\times N$ matrix with $O(1)$ entries. We decompose 
\begin{equation}\label{cr}
\boldsymbol{\alpha}^t=\boldsymbol{\beta}^t+\boldsymbol{\gamma}^t,
\end{equation}
where $\boldsymbol{\beta}^t\!\in\!\ker\hat{\mathbb{A}}_t\!:=\!\{\boldsymbol{\beta}\in\bR^3| \hat{\mathbb{A}}_t\boldsymbol{\beta}=0\}$ and $\boldsymbol{\gamma}^t\neq0$ is orthogonal to $\boldsymbol{\beta}^t$. By substituting \eqref{cr} into \eqref{br}, we obtain
\begin{equation}\label{aaa}
(\hat{\mathbb{A}}_t-\rho(\epsilon)^{-1}\hat{\mathbb{C}}-\gamma\rho(\epsilon)^{-\frac{1}{s}}\mathbb{I})\boldsymbol{\gamma}^t=(\rho(\epsilon)^{-1}\hat{\mathbb{C}}+\gamma\rho(\epsilon)^{-\frac{1}{s}}\mathbb{I})\boldsymbol{\beta}^t,
\end{equation}
which implies $\boldsymbol{\gamma}^t=O(\rho(\epsilon)^{-1/s})$.

Let $\boldsymbol{\gamma}^t=\rho(\epsilon)^{-\frac{1}{s}}\hat{\boldsymbol{\gamma}}^t$ with $\hat{\boldsymbol{\gamma}}=O(1)$. Then, \eqref{aaa} becomes
\begin{equation}\label{bbb}
\hat{\mathbb{A}}_t\hat{\boldsymbol{\gamma}}^t-\gamma\boldsymbol{\beta}^t=\rho(\epsilon)^{-1+\frac{1}{s}}\hat{\mathbb{C}}\boldsymbol{\beta}^t+\big(\gamma\rho(\epsilon)^{-\frac{1}{s}}\mathbb{I}+\rho(\epsilon)^{-1}\hat{\mathbb{C}}\big)\hat{\boldsymbol{\gamma}}^t.
\end{equation}
Since $\hat{\mathbb{A}}_t$ is symmetric, $\ker\hat{\mathbb{A}}_t\cap\sigma(\hat{\mathbb{A}}_t)=\{0\}$, where $\sigma(\hat{\mathbb{A}}_t)=\{\hat{\mathbb{A}}_t\boldsymbol{\gamma}|\boldsymbol{\gamma}\in\bR^3\}$. This implies that the left--hand side of \eqref{bbb} is of order $O(1)$, while the right side of \eqref{bbb} is of order $O(\rho(\epsilon)^{-1+\frac{1}{s}})$, resulting in a contradiction. Thus, $\tilde{a}_N^t=\hat{a}_N^t+O(\rho(\epsilon)^{-1})$, completing the proof.
\end{proof}

Similar to Lemma \ref{le_omg1_chain}, we have the following lemma.
\begin{lemma}\label{le_omg1_ring}
Let $\tilde{\lambda}_1$ be defined in \eqref{tillam}. Under the ring arrangement, $\tilde{\lambda}_1=\frac{3v_b^2M}{4\pi R^3}\delta\big(1+o(1)\big)$,where $M$ is defined in \eqref{def_M}.
\end{lemma}

\begin{proof}[Proof of Theorem \ref{Th_ring}]
Combining Lemmas \ref{le_omega}, \ref{le_lam1_r}, \ref{le_equal_r}, and \ref{le_omg1_ring}, with the relation 
$\tilde{\lambda}_i=\frac{3\delta v_b^2\tilde{a}_N^i}{4R^3}|\log\epsilon|$, the result follows.
\end{proof}

\subsection{Resonant mode}

In contrast to the chain arrangement, the simplified matrix $\hat{\mathbb{A}}$ associated with the capacitance matrix $\mathbb{C}$ in a ring configuration possesses equal eigenvalues, complicating the determination of $\mathbb{C}$'s eigenvectors. An exact characterization of $\mathbb{C}$'s eigenvectors is formalized below.

\begin{lemma}\label{le_modes_ring}
Let $\mathbb{C}$ be defined as in Proposition \ref{prop_C_ring}. For $t=2,\dots,\alpha(N)$, let $\boldsymbol{\alpha}^{2t-2}$ and $\boldsymbol{\alpha}^{2t-1}$ denote eigenvectors of $\mathbb{C}$ associated with eigenvalues $\lambda_{2t-2}$, $\lambda_{2t-1}$, respectively. Then, there exist constants $k_1^{2t-2}$, $k_2^{2t-2}$, $k_1^{2t-1}$ and $k_2^{2t-1}$ such that
\begin{align*}
\boldsymbol{\alpha}^{2t-2}&=k_1^{2t-2}\boldsymbol{\beta}_1^t+k_2^{2t-2}\boldsymbol{\beta}_2^t+\boldsymbol{\gamma}^{2t-2},\\
\boldsymbol{\alpha}^{2t-1}&=k_1^{2t-1}\boldsymbol{\beta}_1^t+k_2^{2t-1}\boldsymbol{\beta}_2^t+\boldsymbol{\gamma}^{2t-1},
\end{align*}
where $\boldsymbol{\beta}_1^t$, $\boldsymbol{\beta}_2^t$ are linearly independent eigenvectors of $\hat{\mathbb{A}}$ associated with eigenvalue $\hat{a}_N^t$, and $\boldsymbol{\gamma}^{2t-2}$, $\boldsymbol{\gamma}^{2t-1}$ are vectors of order $O(\rho(\epsilon)^{-1})$.
\end{lemma}
\begin{proof}
We only prove this for $\boldsymbol{\alpha}^{2t-2}$. By Lemma \ref{le_equal_r}, there exist a constant $\gamma$ such that $\lambda_{2t-2}=\hat{a}_N^t\rho(\epsilon)+\gamma$. Then, $\mathbb{C}\boldsymbol{\alpha}^{2t-2}=(\hat{a}_N^t\rho(\epsilon)+\gamma)\boldsymbol{\alpha}^{2t-2}$. Thus, we have $\hat{\mathbb{A}}_t\boldsymbol{\alpha}^{2t-2}=\rho(\epsilon)^{-1}(\gamma\mathbb{I}+\hat{\mathbb{C}})\boldsymbol{\alpha}^{2t-2}$, where $\hat{\mathbb{C}}$ is an $N\times N$ matrix with $O(1)$ entries. We decompose 
$\boldsymbol{\alpha}^{2t-2}=\boldsymbol{\beta}^t+\boldsymbol{\gamma}^{2t-2},$
where $\boldsymbol{\beta}^t\!\in\ker\hat{\mathbb{A}}_t$ and $\boldsymbol{\gamma}^{2t-2}\neq0$ is orthogonal to $\boldsymbol{\beta}^t$. Hence, we have $\hat{\mathbb{A}}_t\boldsymbol{\gamma}^{2t-2}=\rho(\epsilon)^{-1}(\gamma\mathbb{I}+\hat{\mathbb{C}})(\boldsymbol{\beta}^t+\boldsymbol{\gamma}^{2t-2})$, which implies $\boldsymbol{\gamma}^{2t-2}=O(\rho(\epsilon)^{-1})$. 

Let $\boldsymbol{\gamma}^{2t-2}=\rho(\epsilon)^{-1}\hat{\boldsymbol{\gamma}}^{2t-2}$. Then 
\begin{equation}\label{ccc}
\hat{\mathbb{A}}_t\hat{\boldsymbol{\gamma}}^{2t-2}=(\gamma\mathbb{I}+\hat{\mathbb{C}})\boldsymbol{\beta}^t+\rho(\epsilon)^{-1}(\gamma\mathbb{I}+\hat{\mathbb{C}})\hat{\boldsymbol{\gamma}}^{2t-2}.
\end{equation}
Since $\mathbb{C}$ is a symmetric matrix, we know that there exist $\boldsymbol{\beta}^t=k_1^{2t-2}\boldsymbol{\beta}_1^t+k_2^{2t-2}\boldsymbol{\beta}_2^t$ and $\hat{\boldsymbol{\gamma}}^{2t-2}$ satisfies \eqref{ccc}, where $k_1^{2t-2}$ and $k_2^{2t-2}$ are some constants. This completes the proof.
\end{proof}

 Now, we determine the eigenvectors associated with each eigenvalues. Let $\hat{\mathbb{A}}_t$ be defined as in \eqref{def_At} and suppose  $\hat{\mathbb{A}}_t\boldsymbol{\beta}^t=0$ for $t=2,\dots,\alpha(N)$. Recursion reveals the solution space basis $\boldsymbol{\beta}_1^t=(\beta^t_{11},\dots,\beta^t_{1N})^\mathrm{T}$ and $\boldsymbol{\beta}_2^t=(\beta^t_{21},\dots,\beta^t_{2N})^\mathrm{T}$, where for $l=1,\dots,N$,
\begin{equation}\label{aa}
\begin{aligned}
\beta^t_{1l}&=2\cos\hat{\theta}_t\cos l\hat{\theta}_t-\frac{\cos2\hat{\theta}_t}{\sin\hat{\theta}_t}\sin l\hat{\theta}_t=-\frac{\sin (l-2)\hat{\theta}_t}{\sin\hat{\theta}_t},\\
\beta^t_{2l}&=(1+2\cos2\hat{\theta}_t)\cos l\hat{\theta}_t+\frac{\cos\hat{\theta}_t(1-2\cos2\hat{\theta}_t)}{\sin\hat{\theta}_t}\sin l\hat{\theta}_t=-\frac{\sin (l-1)\hat{\theta}_t}{\sin\hat{\theta}_t}
\end{aligned}
\end{equation}
It is straightforward to verify that
\begin{equation}
\boldsymbol{\beta}^1=(1,\dots,1)^\mathrm{T}
\end{equation}
is the eigenvector for $\hat{a}_N^1$. For even $N$, direct computation gives
\begin{equation}\label{bb}
\boldsymbol{\beta}^N=(1,-1,\dots,1,-1)^\mathrm{T}
\end{equation}
as the eigenvector for $\hat{a}_N^{N/2+1}$.

\begin{proof}[Proof of Theorem \ref{Th_mode_r}]
Using Lemma \ref{le_modes_ring}, and \eqref{aa}--\eqref{bb}, the result follows by extending the methodology of Theorem \ref{Th_mode} to the ring configuration.
\end{proof}

\section{The matrix arrangement}\label{sec_matrix}

In this section we consider a matrix arrangement of $N$ resonators, which can be regarded as an extension of the chain configuration while exhibiting a more intricate and fascinating frequency distribution. Let $N=m\times n$ with $n\ge m\ge2$, and partition the capacitance matrix $\mathbb{C}$ into $m\times m$ block matrix of $n\times n$ submatrices:
$$\mathbb{C}=\left(\begin{matrix}
\mathbb{A}_{11}&\cdots&\mathbb{A}_{1m}\\
\vdots& \ddots&\vdots\\
\mathbb{A}_{m1}&\cdots&\mathbb{A}_{mm}
\end{matrix}\right),$$
where  each block $\mathbb{A}_{ij}=(a_{ij}^{kl})_{n\times n}$ with $a_{ij}^{kl}=C_{(i-1)n+k,{(j-1)n+l}}$, for $1\le i,j\le m$ and $1\le k,l\le n$.

The following proposition is similar to Proposition \ref{prop_C_train} to the matrix arrangement. The proof follows analogously and is omitted.

\begin{prop}\label{prop_C_matrix}
For capacitance matrix $\mathbb{C}$ of the matrix arrangement, we have

(\romannumeral 1) Corner and edge diagonal entries: $a_{11}^{11}=a_{11}^{nn}=a_{mm}^{11}=a_{mm}^{nn}=2\rho(\epsilon) +O(1)$, while for interior diagonal entries ($2\le k\le n-1$): $a_{11}^{kk}, a_{mm}^{kk}=3\rho(\epsilon) +O(1)$;

(\romannumeral 2) Intermediate blocks ($m\ge 3$): for $2\le k\le n-1$, $a_{ii}^{11}=a_{ii}^{nn}=3\rho(\epsilon) +O(1)$, and and for interior entries ($2\le i\le m-1$): $a_{ii}^{kk}=4\rho(\epsilon) +O(1)$;

(\romannumeral 3) Adjacent couplings: For $1\le i\le m$ and $1\le l\le n-1$, $a_{ii}^{l,l+1}=a_{ii}^{l+1,l}=-\rho(\epsilon) +O(1)$; For $1\le j\le m-1$ and $1\le k\le n$, $a_{j,j+1}^{kk}=a_{j+1,j}^{kk}=-\rho(\epsilon) +O(1)$; 

(\romannumeral 4) Remaining entries: All other entries of $\mathbb{C}$ are of order $O(1)$.
\end{prop}

Let $\mathbb{I}_n$ denote the $n\times n$ identity matrix and let  $\mathbb{A}$ be the matrix defined in \eqref{def_A}. Consider an $N$-th order tridiagonal block Toeplitz matrix $\tilde{\mathbb{A}}$ with perturbations $(-\mathbb{I}_n, -\mathbb{I}_n)$ in its diagonal corners, that is, 
\begin{equation*}\label{def_A_m}
\tilde{\mathbb{A}}=\left(\begin{array}{ccccc}
\mathbb{A}\!+\! \mathbb{I}_n & -\mathbb{I}_n  &  &  &  \\
-\mathbb{I}_n & \mathbb{A}\!+\!2\mathbb{I}_n & \ddots &  &  \\
& \ddots & \ddots & \ddots &  \\
&  & \ddots & \mathbb{A}\!+\!2\mathbb{I}_n  & -\mathbb{I}_n \\
&  &  & -\mathbb{I}_n & \mathbb{A}\!+\!\mathbb{I}_n
\end{array}\right).
\end{equation*}
From Proposition \ref{prop_C_matrix}, it follows that
\begin{equation}\label{def_Fa_m}
F(a)=\det(\mathbb{C}-a\rho (\epsilon)\mathbb{I})=f_N(a)\rho(\epsilon)^{N}+f_{N-1}(a)\rho(\epsilon)^{N-1}+O(\rho(\epsilon)^{N-2}),
\end{equation}
where $f_N(a)=\det(\tilde{\mathbb{A}}-a\mathbb{I})$. Similarly as before, we seek the roots of $F(a)$. If $a\le0$, according to Proposition \ref{prop_C_matrix}, the matrix $\mathbb{C}-a\rho(\epsilon)\mathbb{I}$ is a diagonally dominant, and hence $F(a)=\det(\mathbb{C}-a\rho(\epsilon)\mathbb{I})>0$. When $a\ge 8$, the matrix $-(\mathbb{C}-a\rho(\epsilon)\mathbb{I})$ becomes diagonally dominant. Therefore, we restrict our analysis to $0<a<8$ in this section.

Since $\tilde{\mathbb{A}}$ is a block matrix, the recursive method used to compute $\det\mathbb{A}$ in Section \ref{sec_chain} is not directly applicable. However, elementary transformations remain a powerful tool for block matrices. Below, we illustrate how to use elementary transformations and our previous calculation results of $\det\mathbb{A}$ in section \ref{sec_chain} to calculate $m\times m$ determinant:
\begin{equation*}
\det\mathbb{B}=\left|\begin{array}{ccccc}
1+a & -1 &  &  &  \\
-1 & 2+a & \ddots &  &  \\
& \ddots & \ddots & \ddots &  \\
&  & \ddots & 2+a & -1 \\
&  &  & -1 & 1+a
\end{array}\right|.
\end{equation*}
First, swap the first and second rows of the determinant $\det\mathbb{B}$, and add $1+a $ times the new first row to the new second row. Similarly, swap the second and third rows of the updated determinant $\det\mathbb{B}$, and add $a^2+3a+1$ times the new second row to the new third row. Repeat this process iteratively for subsequent rows. After these transformations, the determinant becomes upper triangular
\begin{equation}\label{z1}
\det\mathbb{B}=(-1)^{m-1}\left|\begin{array}{cccccc}
-1 & 2+a &-1  &  &  \\
& -1&2+a&-1 &  &  \\
&&-1&2+a&\ddots&\\
&& & \ddots & \ddots &-1  \\
&&  &  & -1 & 1+a \\
&&  &  &  & h(a)
\end{array}\right|,
\end{equation}
where $h(a)$ is a degree-$m$ polynomial in $a$ to be determined.  From \eqref{fn_de} and \eqref{z1}, we conclude that
$h(a)=\det\mathbb{B}=f_m(-a)=\prod_{\gamma=1}^{m}(a+a_{m}^{\gamma}),$
where $a_{m}^{\gamma}$ is defined in \eqref{def_anin}.

To compute the determinant of the block matrix $\tilde{\mathbb{A}}$, we apply the same transformation approach used before. Specifically, we interchange the first $n$ rows with the second $n$ rows of the determinant of $\tilde{\mathbb{A}}$. Then, we add  $\mathbb{I}_n+\mathbb{A}$ times the new first $n$ rows to the updated second $n$ rows. Repeat analogous operations (as in the scalar case for $\mathbb{B}$) iteratively across subsequent block rows. This process yields an upper triangular block determinant:

\begin{equation*}
\det\tilde{\mathbb{A}}=(-1)^{n(m-1)}\left|\begin{array}{cccccc}
-\mathbb{I}_n & \mathbb{A}+2\mathbb{I}_n &-\mathbb{I}_n  & &  \\
& -\mathbb{I}_n &\mathbb{A}+2\mathbb{I}_n&-\mathbb{I}_n &  &  \\
&&-\mathbb{I}_n&\mathbb{A}+2\mathbb{I}_n&\ddots\\
&&  & \ddots & \ddots &-\mathbb{I}_n  \\
&&  &  & -\mathbb{I}_n & \mathbb{A}+\mathbb{I}_n \\
&&  &  &  & h(\mathbb{A})
\end{array}\right|.
\end{equation*}
Thus, we have 
\begin{equation}\label{eq_detba}
\det\tilde{\mathbb{A}}=\det h(\mathbb{A})=\det(\mathbb{A}+a_m^1\mathbb{I}_n)(\mathbb{A}+a_m^2\mathbb{I}_n)\cdots(\mathbb{A}+a_m^m\mathbb{I}_n).
\end{equation}

\begin{lemma} \label{Le_det_A+a}
Let $\mathbb{A}$ be defined in \eqref{def_A} and let $a_{m}^{\gamma}$, $a_{n}^{\alpha}$ be defined as in \eqref{def_anin} with $1\le i\le m$, $1\le j\le n$. Then
$$\det(\mathbb{A}+a_{m}^{\gamma}\mathbb{I}_n)=(-1)^n\prod_{\alpha=1}^{n}\big(a-(a_{m}^{\gamma}+a_{n}^{\alpha})\big).$$
\end{lemma}
\begin{proof}
Let $\tilde{a}=a-a_{m}^{\gamma}$. Since $f_n(\tilde{a})=(-1)^n(\tilde{a}-a_n^1)(\tilde{a}-a_n^2)\cdots(\tilde{a}-a_n^n)$, then
$\det(\mathbb{A}+a_{m}^{\gamma}\mathbb{I}_n)=f_n(a-a_{m}^{\gamma})=(-1)^n\prod_{\alpha=1}^{n}\big(a-(a_{m}^{\gamma}+a_{n}^{\alpha})\big).$
\end{proof}
Combining \eqref{eq_detba} with Lemma \ref{Le_det_A+a}, and seting $N=m\times n$, we immediately have
\begin{equation}\label{fN_de}
f_N(a)=\det\tilde{\mathbb{A}}=(-1)^{N}\prod_{\gamma=1}^{m}\prod_{\alpha=1}^{n}\big(a-(a_{m}^{\gamma}+a_{n}^{\alpha})\big).
\end{equation}

Let $0=\bar{b}_N^1<\dots<\bar{b}_N^s$ denote the distinct roots of $f_N(a)$ where $1\le s\le N$, and let $r_t$ be the multiplicity of the root $\bar{b}_N^t $, $1\le t\le s$. We factorize $f_N(a)$ in \eqref{fN_de} as $f_N(a)=(-1)^N\prod_{t=1}^s(a-\bar{b}_N^t)^{r_t}$.
The following lemmas are similar to Lemmas  \ref{le_equal_r}, and \ref{le_omg1_ring}. So we omit their proofs.
\begin{lemma}\label{le_omg1_matrix}
Let $\tilde{\lambda}_1$ be defined in \eqref{tillam}. Then $\tilde{\lambda}_1=\frac{3v_b^2M}{4\pi R^3}\delta\big(1+o(1)\big)$,where $M$ is defined in \eqref{def_M}.
\end{lemma}

\begin{lemma}%\label{le_equal}
Let $F(a)$ be defined in \eqref{def_Fa_m}. Then for $2\le t\le s$, there are $r_t$ roots of $F(a)$ within the interval
$\big(\bar{b}_N^t-C/\rho(\epsilon),\bar{b}_N^t+C/\rho(\epsilon)\big).$
\end{lemma}
\begin{proof}[Proof of Theorem \ref{Th_matrix}]
The conclusion follows directly by adapting the proof methodology of Theorem \ref{Th_ring} to the matrix case.
\end{proof}

\section{Examples}\label{sec_exam}

In this section, we present several examples to illustrate the results obtained in Sections \ref{sec_chain}, \ref{sec_ring}, and \ref{sec_matrix}. Let $u_l^i:=u^i|_{\partial D_l}$, $1\le i,l\le N$. Let $\gamma=\min\{\beta/2, 1-\beta\}$, and define $\eta=\sqrt{(3\Lambda v_b^2)/(4R^3)}\delta^{\beta/2}$.
\subsection{For $N=4$}
\subsubsection{Chain arrangement}By Theorem \ref{Th_chain}, the eigenvalues are
$a_4^2=2-\sqrt{2}, a_4^3=2, a_4^4=2+\sqrt{2}$.
The corresponding resonant frequencies are
\begin{align*}
&\omega_1=\sqrt{\frac{3v_b^2M}{4\pi R^3}\delta}\big(1+o(1)\big),&&\hspace{-20mm}\omega_2=\sqrt{2-\sqrt{2}}\eta+O(\delta^{1-\beta/2}),\\
&\omega_3=\sqrt{2}\eta+O(\delta^{1-\beta/2}),&&\hspace{-20mm}\omega_4=\sqrt{2+\sqrt{2}}\eta+O(\delta^{1-\beta/2}).
\end{align*}

From Theorem \ref{Th_mode}, the resonant mode $u^1$, associated with the resonant frequency $\omega_1$, satisfies $u^1_l=1$ up to an error of order $O(\delta^{\gamma})$. The modes $u^2$, $u^3$, and $u^4$, associated with $\omega_2$, $\omega_3$, and $\omega_4$, satisfy the following boundary conditions:
\begin{align*}
&u^2_1=1,&&\hspace{-8mm}u^2_2=-1+\sqrt{2},&&\hspace{-8mm}u^2_3=1-\sqrt{2}, &&\hspace{-8mm}u^2_4=-1;\\
&u^3_1=1,&&\hspace{-8mm}u^3_2=-1,&&\hspace{-8mm}u^3_3=-1,&&\hspace{-8mm}u^3_4=1;\\
&u^4_1=1,&&\hspace{-8mm}u^4_2=-1-\sqrt{2},&&\hspace{-8mm}u^4_3=1+\sqrt{2},&&\hspace{-8mm}u^4_4=-1.
\end{align*}
From these, we can observe that $\nabla u^3$ exhibits a lower blow--up rate between $D_2$ and $D_3$. Figure $\ref{f_c4}$ provides an intuitive visualization of these resonant modes (error terms omitted for clarity).

\begin{remark}\label{Re_ima}It follows from \eqref{eq_ima} that 
\begin{align*}
\Im\omega_i=-\delta\frac{3v_b^2}{32v\pi ^2R^3}\frac{\big(\sum_{l=1}^Nu_l^i\int_{\partial B_{\tilde{R}}}\frac{\partial v_l}{\partial\nu}\big|_+\rd\sigma\big)^2}{\sum_{l=1}^{N}(u_l^i)^2}+O(\delta^{3/2}).
\end{align*}
Then, for the chain arrangement of $N=4$, we have  $\Im\omega_2,~~\Im\omega_4=O(\delta^{\gamma'})$, and up to an error of order $O(\delta^{\gamma'})$,
\begin{align*}
\Im\omega_1\!=\!-\delta\frac{3v_b^2}{32v\pi ^2R^3}\Big(\int_{B_{\tilde{R}}}\!\!\frac{\partial(v_1+v_2)}{\partial\nu}\Big|_+\!\rd\sigma\Big)^2,~
\Im\omega_3\!=\!-\delta\frac{3v_b^2}{32v\pi ^2R^3}\Big(\int_{B_{\tilde{R}}}\!\!\frac{\partial(v_1-v_2)}{\partial\nu}\Big|_+\!\rd\sigma\Big)^2,
\end{align*}
where $\gamma'=\min\{3/2,1+\gamma\}$.
\end{remark}

\begin{figure}[h]
\begin{tikzpicture}
\def\a{0.8}
\def\b{1}

\shade[shading=ball, ball color=blue!30!white, opacity=0.8]   (0,0) circle[radius=\a];
\shade[shading=ball, ball color=blue!30!white, opacity=0.8]   (2.1*\a,0) circle[radius=\a];
\shade[shading=ball, ball color=blue!30!white, opacity=0.8]   (4.2*\a,0) circle[radius=\a];
\shade[shading=ball, ball color=blue!30!white, opacity=0.8]   (6.3*\a,0) circle[radius=\a];

\begin{scope}
	\clip (0.5*\a,-0.4*\a)rectangle(1.6*\a,0.4*\a);
	\filldraw[color=orange!40,opacity=0.975](0.01,1*\a)arc(90:-90:1*\a)--(0.01,-1*\a)--(2.087*\a,-1*\a)--(2.087*\a,-1*\a)arc(-90:-270:1*\a)--cycle;
\end{scope}
\begin{scope}
	\clip (2.5*\a,-0.4*\a)rectangle(3.6*\a,0.4*\a);
	\filldraw[color=orange!40,opacity=0.975](2.112*\a,1*\a)arc(90:-90:1*\a)--(2.112*\a,-1*\a)--(4.187*\a,-1*\a)--(4.187*\a,-1*\a)arc(-90:-270:1*\a)--cycle;
\end{scope}
\begin{scope}
	\clip (4.5*\a,-0.4*\a)rectangle(5.6*\a,0.4*\a);
	\filldraw[color=orange!40,opacity=0.975](4.212*\a,1*\a)arc(90:-90:1*\a)--(4.212*\a,-1*\a)--(6.287*\a,-1*\a)--(6.287*\a,-1*\a)arc(-90:-270:1*\a)--cycle;
\end{scope}
\node [scale=1.1*\a] at (0,0) {1};
\node [scale=1.1*\a] at (2.1*\a,0) {1};
\node [scale=1.1*\a] at (4.2*\a,0) {1};
\node [scale=1.1*\a] at (6.3*\a,0) {1};
\node [scale=\a] at (0,-0.7*\a) {$D_1$};
\node [scale=\a] at (2.1*\a,-0.7*\a) {$D_2$};
\node [scale=\a] at (4.2*\a,-0.7*\a) {$D_3$};
\node [scale=\a] at (6.3*\a,-0.7*\a) {$D_4$};

\node at (-0.9*\b,-0.9*\b) {};
\node at (5.3*\b,1*\b) {};
\node at (3.05*\a,-1.6*\a) {(a) $u^1$};
\end{tikzpicture}
\begin{tikzpicture}
\def\a{0.8}
\def\b{1}

\shade[shading=ball, ball color=blue!30!white, opacity=0.8]   (0,0) circle[radius=\a];
\shade[shading=ball, ball color=blue!30!white, opacity=0.8]   (2.1*\a,0) circle[radius=\a];
\shade[shading=ball, ball color=blue!30!white, opacity=0.8]   (4.2*\a,0) circle[radius=\a];
\shade[shading=ball, ball color=blue!30!white, opacity=0.8]   (6.3*\a,0) circle[radius=\a];

\begin{scope}
	\clip (0.5*\a,-0.4*\a)rectangle(1.6*\a,0.4*\a);
	\filldraw[color=red!60,opacity=0.975](0.01,1*\a)arc(90:-90:1*\a)--(0.01,-1*\a)--(2.087*\a,-1*\a)--(2.087*\a,-1*\a)arc(-90:-270:1*\a)--cycle;
\end{scope}
\begin{scope}
	\clip (2.5*\a,-0.4*\a)rectangle(3.6*\a,0.4*\a);
	\filldraw[color=red!60,opacity=0.975](2.112*\a,1*\a)arc(90:-90:1*\a)--(2.112*\a,-1*\a)--(4.187*\a,-1*\a)--(4.187*\a,-1*\a)arc(-90:-270:1*\a)--cycle;
\end{scope}
\begin{scope}
	\clip (4.5*\a,-0.4*\a)rectangle(5.6*\a,0.4*\a);
	\filldraw[color=red!60,opacity=0.975](4.212*\a,1*\a)arc(90:-90:1*\a)--(4.212*\a,-1*\a)--(6.287*\a,-1*\a)--(6.287*\a,-1*\a)arc(-90:-270:1*\a)--cycle;
\end{scope}

\node [scale=1.1*\a] at (0,0) {1};
\node [scale=1.1*\a] at (2.1*\a,0) {$-1\!+\sqrt{2}$};
\node [scale=1.1*\a] at (4.2*\a,0) {$1-\sqrt{2}$};
\node [scale=1.1*\a] at (6.3*\a,0) {$-1$};
\node [scale=\a] at (0,-0.7*\a) {$D_1$};
\node [scale=\a] at (2.1*\a,-0.7*\a) {$D_2$};
\node [scale=\a] at (4.2*\a,-0.7*\a) {$D_3$};
\node [scale=\a] at (6.3*\a,-0.7*\a) {$D_4$};

\node at (-0.9*\b,-0.9*\b) {};
\node at (5.3*\b,1*\b) {};
\node at (3.05*\a,-1.6*\a) {(b) $u^2$};
\end{tikzpicture}
\begin{tikzpicture}
\def\a{0.8}
\def\b{1}

\shade[shading=ball, ball color=blue!30!white, opacity=0.8]   (0,0) circle[radius=\a];
\shade[shading=ball, ball color=blue!30!white, opacity=0.8]   (2.1*\a,0) circle[radius=\a];
\shade[shading=ball, ball color=blue!30!white, opacity=0.8]   (4.2*\a,0) circle[radius=\a];
\shade[shading=ball, ball color=blue!30!white, opacity=0.8]   (6.3*\a,0) circle[radius=\a];
\begin{scope}
	\clip (0.5*\a,-0.4*\a)rectangle(1.6*\a,0.4*\a);
	\filldraw[color=red!60,opacity=0.975](0.01,1*\a)arc(90:-90:1*\a)--(0.01,-1*\a)--(2.087*\a,-1*\a)--(2.087*\a,-1*\a)arc(-90:-270:1*\a)--cycle;
\end{scope}
\begin{scope}
	\clip (2.5*\a,-0.4*\a)rectangle(3.6*\a,0.4*\a);
	\filldraw[color=orange!40,opacity=0.975](2.112*\a,1*\a)arc(90:-90:1*\a)--(2.112*\a,-1*\a)--(4.187*\a,-1*\a)--(4.187*\a,-1*\a)arc(-90:-270:1*\a)--cycle;
\end{scope}
\begin{scope}
	\clip (4.5*\a,-0.4*\a)rectangle(5.6*\a,0.4*\a);
	\filldraw[color=red!60,opacity=0.975](4.212*\a,1*\a)arc(90:-90:1*\a)--(4.212*\a,-1*\a)--(6.287*\a,-1*\a)--(6.287*\a,-1*\a)arc(-90:-270:1*\a)--cycle;
\end{scope} 

\node [scale=1.1*\a] at (0,0) {1};
\node [scale=1.1*\a] at (2.1*\a,0) {$-1$};
\node [scale=1.1*\a] at (4.2*\a,0) {$-1$};
\node [scale=1.1*\a] at (6.3*\a,0) {1};
\node [scale=\a] at (0,-0.7*\a) {$D_1$};
\node [scale=\a] at (2.1*\a,-0.7*\a) {$D_2$};
\node [scale=\a] at (4.2*\a,-0.7*\a) {$D_3$};
\node [scale=\a] at (6.3*\a,-0.7*\a) {$D_4$};

\node at (-0.9*\b,-0.9*\b) {};
\node at (5.3*\b,1*\b) {};
\node at (3.05*\a,-1.6*\a) {(c) $u^3$};
\end{tikzpicture}
\begin{tikzpicture}
\def\a{0.8}
\def\b{1}

\shade[shading=ball, ball color=blue!30!white, opacity=0.8]   (0,0) circle[radius=\a];
\shade[shading=ball, ball color=blue!30!white, opacity=0.8]   (2.1*\a,0) circle[radius=\a];
\shade[shading=ball, ball color=blue!30!white, opacity=0.8]   (4.2*\a,0) circle[radius=\a];
\shade[shading=ball, ball color=blue!30!white, opacity=0.8]   (6.3*\a,0) circle[radius=\a];

\begin{scope}
	\clip (0.5*\a,-0.4*\a)rectangle(1.6*\a,0.4*\a);
	\filldraw[color=red!60,opacity=0.975](0.01,1*\a)arc(90:-90:1*\a)--(0.01,-1*\a)--(2.087*\a,-1*\a)--(2.087*\a,-1*\a)arc(-90:-270:1*\a)--cycle;
\end{scope}
\begin{scope}
	\clip (2.5*\a,-0.4*\a)rectangle(3.6*\a,0.4*\a);
	\filldraw[color=red!60,opacity=0.975](2.112*\a,1*\a)arc(90:-90:1*\a)--(2.112*\a,-1*\a)--(4.187*\a,-1*\a)--(4.187*\a,-1*\a)arc(-90:-270:1*\a)--cycle;
\end{scope}
\begin{scope}
	\clip (4.5*\a,-0.4*\a)rectangle(5.6*\a,0.4*\a);
	\filldraw[color=red!60,opacity=0.975](4.212*\a,1*\a)arc(90:-90:1*\a)--(4.212*\a,-1*\a)--(6.287*\a,-1*\a)--(6.287*\a,-1*\a)arc(-90:-270:1*\a)--cycle;
\end{scope}

\node [scale=1.1*\a] at (0,0) {1};
\node [scale=1.1*\a] at (2.1*\a,0) {$-1\!-\sqrt{2}$};
\node [scale=1.1*\a] at (4.2*\a,0) {$1+\sqrt{2}$};
\node [scale=1.1*\a] at (6.3*\a,0) {$-1$};
\node [scale=\a] at (0,-0.7*\a) {$D_1$};
\node [scale=\a] at (2.1*\a,-0.7*\a) {$D_2$};
\node [scale=\a] at (4.2*\a,-0.7*\a) {$D_3$};
\node [scale=\a] at (6.3*\a,-0.7*\a) {$D_4$};

\node at (-0.9*\b,-0.9*\b) {};
\node at (5.3*\b,1*\b) {};
\node at (3.05*\a,-1.6*\a) {(d) $u^4$};
\end{tikzpicture}
\caption{Resonant modes of a chain arrangement for $N=4$. In this figure and subsequent ones, the orange regions between the spheres indicate that the upper--bounded norm of the gradient of the resonant mode is $\frac{C}{\epsilon|\log\epsilon|}$ in those areas. Conversely, the red regions denote that the gradient (and thus its norm) will exhibit a blow--up behavior with a rate of $\frac{C}{\epsilon}$ within those intervals.}\label{f_c4}
\end{figure}
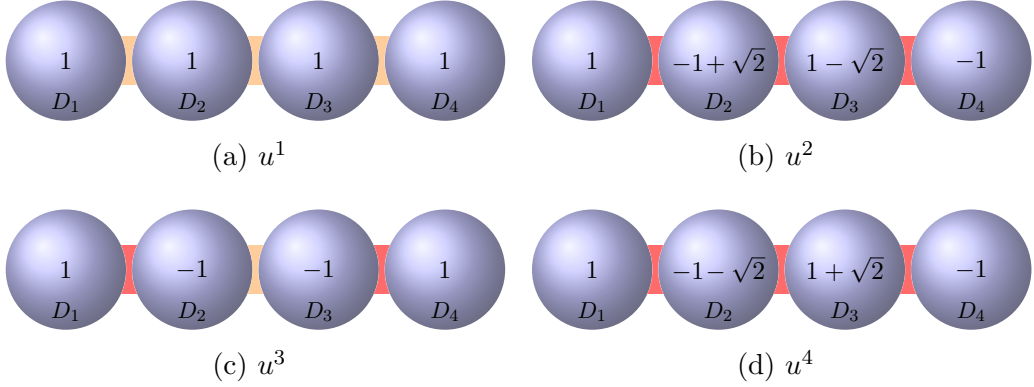

\subsubsection{Ring arrangement}

From Theorem \ref{Th_ring}, we obtain the eigenvalues
$\hat{a}_4^2=2$, $\hat{a}_4^3=4.$
Consequently, the resonant frequencies are 
$$\omega_1=\sqrt{\frac{3v_b^2M}{4\pi R^3}\delta}\big(1+o(1)\big), \quad\omega_2,\omega_3=\sqrt{2}\eta+O(\delta^{1-\beta/2}),\quad\omega_4=2\eta+O(\delta^{1-\beta/2}).$$
By Theorem \ref{Th_mode_r}, the resonant modes satisfy the following boundary conditions (up to an error of order $O(\delta^\gamma)$): for the mode $u^1$, associated with  $\omega_1$, $u_l^1=1$; for the modes $u^i$, associated with  $\omega_i$, $2\le i\le 4$, 
\begin{align*}
&u^2_1=k_1^2, &&\hspace{-12mm}u^2_2=k_2^2, &&\hspace{-12mm}u^2_3=-k_1^2, &&\hspace{-12mm}u^2_4=-k_2^2;\\
&u^3_1=k_1^3, &&\hspace{-12mm}u^3_2=k_2^3, &&\hspace{-12mm}u^3_3=-k_1^3, &&\hspace{-12mm}u^3_4=-k_2^3;\\
&u^4_1=1, &&\hspace{-12mm}u^4_2=-1, &&\hspace{-12mm}u^4_3=1, &&\hspace{-12mm}u^4_4=-1,
\end{align*}
where $k_t^s$ are some constants for $t=1,2$ and $s=2,3$. We observe that, except for $u^1$, the gradient of resonant modes $u^2$, $u^3$ and $u^4$ may blow up with a rate of $\epsilon^{-1}$, as illustrated in Figure \ref{f_r4}.
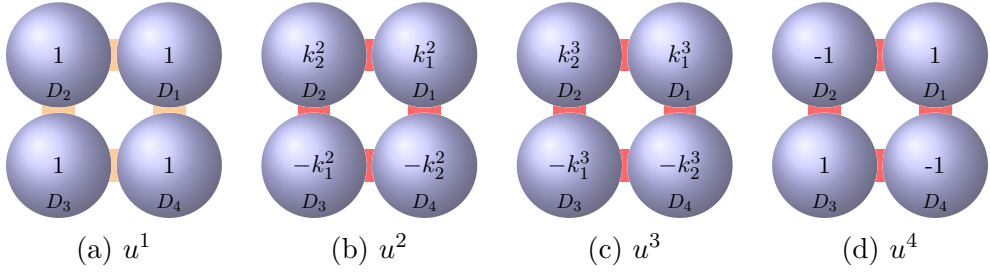
\begin{figure}[h]
\begin{tikzpicture}
\def\a{0.7}

\shade[shading=ball, ball color=blue!30!white, opacity=0.8]    (0,0) circle[radius=\a];
\shade[shading=ball, ball color=blue!30!white, opacity=0.8]   (2.1*\a,0) circle[radius=\a];
\shade[shading=ball, ball color=blue!30!white, opacity=0.8]    (0,2.1*\a) circle[radius=\a];
\shade[shading=ball, ball color=blue!30!white, opacity=0.8]   (2.1*\a,2.1*\a) circle[radius=\a];

\begin{scope}
	\clip (0.5*\a,-0.3*\a)rectangle(1.6*\a,0.3*\a);
	\filldraw[color=orange!40](0.01,1*\a)arc(90:-90:1*\a)--(0.01,-1*\a)--(2.086*\a,-1*\a)--(2.086*\a,-1*\a)arc(-90:-270:1*\a)--cycle;
\end{scope}
\begin{scope}
	\clip (0.5*\a,1.8*\a)rectangle(1.6*\a,2.4*\a);
	\filldraw[color=orange!40](0.01,3.1*\a)arc(90:-90:1*\a)--(0.01,1.1*\a)--(2.085*\a,1.1*\a)--(2.085*\a,1.1*\a)arc(-90:-270:1*\a)--cycle;
\end{scope}
\begin{scope}
	\clip (-0.3*\a,0.5*\a)rectangle(0.3*\a,1.6*\a);
	\filldraw[color=orange!40](1*\a,0.01)arc(0:180:1*\a)--(-1*\a,0.01)--(-1*\a,2.086*\a)--(-1*\a,2.086*\a)arc(-180:0:1*\a)--cycle;
\end{scope}
\begin{scope}
	\clip (1.8*\a,0.5*\a)rectangle(2.4*\a,1.6*\a);
	\filldraw[color=orange!40](3.1*\a,0.01)arc(0:180:1*\a)--(1.1*\a,0.01)--(1.1*\a,2.085*\a)--(1.1*\a,2.085*\a)arc(-180:0:1*\a)--cycle;
\end{scope}

\node [scale=1.2*\a] at (0,0) {1};
\node [scale=1.2*\a] at ({2.1*\a},0) {1};
\node [scale=1.2*\a] at (0,{2.1*\a}) {1};
\node [scale=1.2*\a] at (2.1*\a,2.1*\a) {1};
\node [scale=\a] at (0,-0.7*\a) {$D_3$};
\node [scale=\a] at (2.1*\a,-0.7*\a) {$D_4$};
\node [scale=\a] at (0,1.4*\a) {$D_2$};
\node [scale=\a] at (2.1*\a,1.4*\a) {$D_1$};

\node at  (-1.1*\a,-1.1*\a){};
\node at  (3.2*\a,3.2*\a){};
\node at (1.05*\a,-1.6*\a) {(a) $u^1$};
\end{tikzpicture}
\begin{tikzpicture}
\def\a{.7}

\shade[shading=ball, ball color=blue!30!white, opacity=0.8]    (0,0) circle[radius=\a];
\shade[shading=ball, ball color=blue!30!white, opacity=0.8]    (2.1*\a,0) circle[radius=\a];
\shade[shading=ball, ball color=blue!30!white, opacity=0.8]    (0,2.1*\a) circle[radius=\a];
\shade[shading=ball, ball color=blue!30!white, opacity=0.8]    (2.1*\a,2.1*\a) circle[radius=\a];

\begin{scope}
	\clip (0.5*\a,-0.3*\a)rectangle(1.6*\a,0.3*\a);
	\filldraw[color=red!60](0.01,1*\a)arc(90:-90:1*\a)--(0.01,-1*\a)--(2.086*\a,-1*\a)--(2.086*\a,-1*\a)arc(-90:-270:1*\a)--cycle;
\end{scope}
\begin{scope}
	\clip (0.5*\a,1.8*\a)rectangle(1.6*\a,2.4*\a);
	\filldraw[color=red!60](0.01,3.1*\a)arc(90:-90:1*\a)--(0.01,1.1*\a)--(2.085*\a,1.1*\a)--(2.085*\a,1.1*\a)arc(-90:-270:1*\a)--cycle;
\end{scope}
\begin{scope}
	\clip (-0.3*\a,0.5*\a)rectangle(0.3*\a,1.6*\a);
	\filldraw[color=red!60](1*\a,0.01)arc(0:180:1*\a)--(-1*\a,0.01)--(-1*\a,2.086*\a)--(-1*\a,2.086*\a)arc(-180:0:1*\a)--cycle;
\end{scope}
\begin{scope}
	\clip (1.8*\a,0.5*\a)rectangle(2.4*\a,1.6*\a);
	\filldraw[color=red!60](3.1*\a,0.01)arc(0:180:1*\a)--(1.1*\a,0.01)--(1.1*\a,2.085*\a)--(1.1*\a,2.085*\a)arc(-180:0:1*\a)--cycle;
\end{scope}

\node [scale=1.2*\a] at (0,0) {$-k_1^2$};
\node [scale=1.2*\a] at (2.1*\a,0) {$-k_2^2$};
\node [scale=1.2*\a] at (0,2.1*\a) {$k_2^2$};
\node [scale=1.2*\a] at (2.1*\a,2.1*\a) {$k_1^2$};
\node [scale=\a] at (0,-0.7*\a) {$D_3$};
\node [scale=\a] at (2.1*\a,-0.7*\a) {$D_4$};
\node [scale=\a] at (0,1.4*\a) {$D_2$};
\node [scale=\a] at (2.1*\a,1.4*\a) {$D_1$};

\node at  (-1.1*\a,-1.1*\a){};
\node at  (3.2*\a,3.2*\a){};
\node at (1.05*\a,-1.6*\a) {(b) $u^2$};
\end{tikzpicture}
\begin{tikzpicture}
\def\a{.7}

\shade[shading=ball, ball color=blue!30!white, opacity=0.8]    (0,0) circle[radius=\a];
\shade[shading=ball, ball color=blue!30!white, opacity=0.8]    (2.1*\a,0) circle[radius=\a];
\shade[shading=ball, ball color=blue!30!white, opacity=0.8]    (0,2.1*\a) circle[radius=\a];
\shade[shading=ball, ball color=blue!30!white, opacity=0.8]    (2.1*\a,2.1*\a) circle[radius=\a];

\begin{scope}
	\clip (0.5*\a,-0.3*\a)rectangle(1.6*\a,0.3*\a);
	\filldraw[color=red!60](0.01,1*\a)arc(90:-90:1*\a)--(0.01,-1*\a)--(2.086*\a,-1*\a)--(2.086*\a,-1*\a)arc(-90:-270:1*\a)--cycle;
\end{scope}
\begin{scope}
	\clip (0.5*\a,1.8*\a)rectangle(1.6*\a,2.4*\a);
	\filldraw[color=red!60](0.01,3.1*\a)arc(90:-90:1*\a)--(0.01,1.1*\a)--(2.085*\a,1.1*\a)--(2.085*\a,1.1*\a)arc(-90:-270:1*\a)--cycle;
\end{scope}
\begin{scope}
	\clip (-0.3*\a,0.5*\a)rectangle(0.3*\a,1.6*\a);
	\filldraw[color=red!60](1*\a,0.01)arc(0:180:1*\a)--(-1*\a,0.01)--(-1*\a,2.086*\a)--(-1*\a,2.086*\a)arc(-180:0:1*\a)--cycle;
\end{scope}
\begin{scope}
	\clip (1.8*\a,0.5*\a)rectangle(2.4*\a,1.6*\a);
	\filldraw[color=red!60](3.1*\a,0.01)arc(0:180:1*\a)--(1.1*\a,0.01)--(1.1*\a,2.085*\a)--(1.1*\a,2.085*\a)arc(-180:0:1*\a)--cycle;
\end{scope}

\node [scale=1.2*\a] at (0,0) {$-k_1^3$};
\node [scale=1.2*\a] at (2.1*\a,0) {$-k_2^3$};
\node [scale=1.2*\a] at (0,2.1*\a) {$k_2^3$};
\node [scale=1.2*\a] at (2.1*\a,2.1*\a) {$k_1^3$};
\node [scale=\a] at (0,-0.7*\a) {$D_3$};
\node [scale=\a] at (2.1*\a,-0.7*\a) {$D_4$};
\node [scale=\a] at (0,1.4*\a) {$D_2$};
\node [scale=\a] at (2.1*\a,1.4*\a) {$D_1$};

\node at  (-1.1*\a,-1.1*\a){};
\node at  (3.2*\a,3.2*\a){};
\node at (1.05*\a,-1.6*\a) {(c) $u^3$};
\end{tikzpicture}
\begin{tikzpicture}
\def\a{.7}

\shade[shading=ball, ball color=blue!30!white, opacity=0.8]    (0,0) circle[radius=\a];
\shade[shading=ball, ball color=blue!30!white, opacity=0.8]    (2.1*\a,0) circle[radius=\a];
\shade[shading=ball, ball color=blue!30!white, opacity=0.8]    (0,2.1*\a) circle[radius=\a];
\shade[shading=ball, ball color=blue!30!white, opacity=0.8]   (2.1*\a,2.1*\a) circle[radius=\a];

\begin{scope}
	\clip (0.5*\a,-0.3*\a)rectangle(1.6*\a,0.3*\a);
	\filldraw[color=red!60](0.01,1*\a)arc(90:-90:1*\a)--(0.01,-1*\a)--(2.086*\a,-1*\a)--(2.086*\a,-1*\a)arc(-90:-270:1*\a)--cycle;
\end{scope}
\begin{scope}
	\clip (0.5*\a,1.8*\a)rectangle(1.6*\a,2.4*\a);
	\filldraw[color=red!60](0.01,3.1*\a)arc(90:-90:1*\a)--(0.01,1.1*\a)--(2.085*\a,1.1*\a)--(2.085*\a,1.1*\a)arc(-90:-270:1*\a)--cycle;
\end{scope}
\begin{scope}
	\clip (-0.3*\a,0.5*\a)rectangle(0.3*\a,1.6*\a);
	\filldraw[color=red!60](1*\a,0.01)arc(0:180:1*\a)--(-1*\a,0.01)--(-1*\a,2.086*\a)--(-1*\a,2.086*\a)arc(-180:0:1*\a)--cycle;
\end{scope}
\begin{scope}
	\clip (1.8*\a,0.5*\a)rectangle(2.4*\a,1.6*\a);
	\filldraw[color=red!60](3.1*\a,0.01)arc(0:180:1*\a)--(1.1*\a,0.01)--(1.1*\a,2.085*\a)--(1.1*\a,2.085*\a)arc(-180:0:1*\a)--cycle;
\end{scope}

\node [scale=1.2*\a] at (0,0) {1};
\node [scale=1.2*\a] at (2.1*\a,0) {-1};
\node [scale=1.2*\a] at (0,2.1*\a) {-1};
\node [scale=1.2*\a] at (2.1*\a,2.1*\a) {1};
\node [scale=\a] at (0,-0.7*\a) {$D_3$};
\node [scale=\a] at (2.1*\a,-0.7*\a) {$D_4$};
\node [scale=\a] at (0,1.4*\a) {$D_2$};
\node [scale=\a] at (2.1*\a,1.4*\a) {$D_1$};

\node at  (-1.1*\a,-1.1*\a){};
\node at  (3.2*\a,3.2*\a){};
\node at (1.05*\a,-1.6*\a) {(d) $u^4$};
\end{tikzpicture}
\caption{Resonant modes of ring arrangement for $N=4$.}\label{f_r4}
\end{figure}
\subsubsection{Matrix arrangement} When $N=4$, the matrix arrangement is the same as the ring arrangement, with only the order of the resonator numbers being different.
\subsection{For $N=6$}
\subsubsection{Chain arrangement}
By Theorem \ref{Th_chain}, we deduce that
$a_6^2=2-\sqrt{3}$, $a_6^3=1$, $a_6^4=2$, $a_6^5=3$, $a_6^6=2+\sqrt{3}$,
which leads to the resonant frequencies
\begin{align*}
&\omega_1=\sqrt{\frac{3v_b^2M}{4\pi R^3}\delta}\big(1+o(1)\big),&&\omega_2=\sqrt{2-\sqrt{3}}\eta+O(\delta^{1-\beta/2}),&&\omega_3=\eta+O(\delta^{1-\beta/2}),\\
&\omega_4=\sqrt{2}\eta+O(\delta^{1-\beta/2}),&&\omega_5=\sqrt{3}\eta+O(\delta^{1-\beta/2}),&&\omega_6=\sqrt{2+\sqrt{3}}\eta+O(\delta^{1-\beta/2}).
\end{align*}
By Theorem \ref{Th_mode}, with an error of order $O(\delta^{\gamma})$, the resonant mode $u^1$ associated with the resonant frequency $\omega_1$ satisfies $u_l^1=1$. Similarly, with an error of order $O(\delta^{\gamma})$, the resonant modes $u^i$, associated with the resonant frequencies $\omega_i$, $2\le i\le 6$, satisfy the following boundary conditions:
\begin{align*}
&u^2_1=1,&&u^2_2=-1+\sqrt{3},&&u^2_3=2-\sqrt{3},&&u^2_4=-2+\sqrt{3},&&u^2_5=1-\sqrt{3},&&u^2_6=-1;\\
&u^3_1=1,&&u^3_2=0,&&u^3_3=-1,&&u^3_4=-1,&&u^3_5=0,&&u^3_6=1;\\
&u^4_1=1,&&u^4_2=-1,&&u^4_3=-1,&&u^4_4=1,&&u^4_5=1,&&u^4_6=-1;\\
&u^5_1=1,&&u^5_2=-2,&&u^5_3=1,&&u^5_4=1,&&u^5_5=-2,&&u^5_6=1;\\
&u^6_1=1,&&u^6_2=-1-\sqrt{3},&&u^6_3=2+\sqrt{3},&&u^6_4=-2-\sqrt{3},&&u^6_5=1+\sqrt{3},&&u^6_6=-1.
\end{align*}
Therefore, we have that $\nabla u^3$ and $\nabla u^5$  exhibit a lower blow--up rate between $D_3$ and $D_4$, while $\nabla u^4$  exhibits a lower blow--up rate between $D_2$ and $D_3$, as well as $D_4$ and $D_5$, see Figure \ref{f_c6}.

\begin{figure}[h]
\begin{tikzpicture}
\def\a{0.57}

\foreach \x in {1,2,3,4,5,6}
{
\shade[shading=ball, ball color=blue!30!white, opacity=0.8]   ({2.1*(\x-1)*\a},0) circle[radius=\a];
}
\foreach \x in {1,2,3,4,5}
{
	\begin{scope}
		\clip ({2.1*(\x-1)*\a},-0.3*\a)rectangle({2.1*\x*\a},0.3*\a);
		\filldraw[color=orange!40]({(2.1*(\x-1)+0.017)*\a},1*\a)arc(90:-90:1*\a)--({(2.1*(\x-1)+0.017)*\a},-1*\a)--({(2.1*\x-0.17)*\a},-1*\a)--({(2.1*\x-0.017)*\a},-1*\a)arc(-90:-270:1*\a)--cycle;
	\end{scope}
}
\foreach \x in {1,2,3,4,5,6}
{
	\node [scale=\a] at ({2.1*(\x-1)*\a},-0.7*\a) {$D_\x$};
}

\node [scale=1.2*\a] at (0,0) {1};
\node [scale=1.2*\a] at (2.1*\a,0) {1};
\node [scale=1.2*\a] at (4.2*\a,0) {1};
\node [scale=1.2*\a] at (6.3*\a,0) {1};
\node [scale=1.2*\a] at (8.4*\a,0) {1};
\node [scale=1.2*\a] at (10.5*\a,0) {1};

\node  at  (-1.1*\a,-1.2*\a){};
\node  at  (11.6*\a,1.2*\a){};

\node  at (5.25*\a,-1.5*\a) {(a) $u^1$};
\end{tikzpicture}
\begin{tikzpicture}
\def\a{0.57}

\foreach \x in {1,2,3,4,5,6}
{
	\shade[shading=ball, ball color=blue!30!white, opacity=0.8]   ({2.1*(\x-1)*\a},0) circle[radius=\a];
}
\foreach \x in {1,2,3,4,5}
{
	\begin{scope}
		\clip ({2.1*(\x-1)*\a},-0.3*\a)rectangle({2.1*\x*\a},0.3*\a);
		\filldraw[color=red!60]({(2.1*(\x-1)+0.017)*\a},1*\a)arc(90:-90:1*\a)--({(2.1*(\x-1)+0.017)*\a},-1*\a)--({(2.1*\x-0.17)*\a},-1*\a)--({(2.1*\x-0.017)*\a},-1*\a)arc(-90:-270:1*\a)--cycle;
	\end{scope}
}

\foreach \x in {1,2,3,4,5,6}
{
	\node [scale=\a] at ({2.1*(\x-1)*\a},-0.7*\a) {$D_\x$};
}

\node [scale=1.2*\a] at (0,0) {1};
\node [scale=1.2*\a] at (2.1*\a,0) {$-1+\sqrt{3}$};
\node [scale=1.2*\a] at (4.2*\a,0) {$2-\sqrt{3}$};
\node [scale=1.2*\a] at (6.3*\a,0) {$-2+\sqrt{3}$};
\node [scale=1.2*\a] at (8.4*\a,0) {$1-\sqrt{3}$};
\node [scale=1.2*\a] at (10.5*\a,0) {$-1$};

\node  at  (-1.1*\a,-1.2*\a){};
\node  at  (11.6*\a,1.2*\a){};

\node  at (5.25*\a,-1.5*\a) {(b) $u^2$};
\end{tikzpicture}
\begin{tikzpicture}
\def\a{0.57}

\foreach \x in {1,2,3,4,5,6}
{
\shade[shading=ball, ball color=blue!30!white, opacity=0.8]   ({2.1*(\x-1)*\a},0) circle[radius=\a];
}
\foreach \x in {1,2,4,5}
{
	\begin{scope}
		\clip ({2.1*(\x-1)*\a},-0.3*\a)rectangle({2.1*\x*\a},0.3*\a);
		\filldraw[color=red!60]({(2.1*(\x-1)+0.017)*\a},1*\a)arc(90:-90:1*\a)--({(2.1*(\x-1)+0.017)*\a},-1*\a)--({(2.1*\x-0.17)*\a},-1*\a)--({(2.1*\x-0.017)*\a},-1*\a)arc(-90:-270:1*\a)--cycle;
	\end{scope}
}
\foreach \x in {3}
{
	\begin{scope}
		\clip ({2.1*(\x-1)*\a},-0.3*\a)rectangle({2.1*\x*\a},0.3*\a);
		\filldraw[color=orange!40]({(2.1*(\x-1)+0.017)*\a},1*\a)arc(90:-90:1*\a)--({(2.1*(\x-1)+0.017)*\a},-1*\a)--({(2.1*\x-0.17)*\a},-1*\a)--({(2.1*\x-0.017)*\a},-1*\a)arc(-90:-270:1*\a)--cycle;
	\end{scope}
}
\foreach \x in {1,2,3,4,5,6}
{
	\node [scale=\a] at ({2.1*(\x-1)*\a},-0.7*\a) {$D_\x$};
}

\node [scale=1.2*\a] at (0,0) {1};
\node [scale=1.2*\a] at (2.1*\a,0) {0};
\node [scale=1.2*\a] at (4.2*\a,0) {$-1$};
\node [scale=1.2*\a] at (6.3*\a,0) {$-1$};
\node [scale=1.2*\a] at (8.4*\a,0) {0};
\node [scale=1.2*\a] at (10.5*\a,0) {1};

\node  at  (-1.1*\a,-1.2*\a){};
\node  at  (11.6*\a,1.2*\a){};

\node  at (5.25*\a,-1.5*\a) {(c) $u^3$};
\end{tikzpicture}
\begin{tikzpicture}
\def\a{0.57}

\foreach \x in {1,2,3,4,5,6}
{
\shade[shading=ball, ball color=blue!30!white, opacity=0.8]   ({2.1*(\x-1)*\a},0) circle[radius=\a];
}

\foreach \x in {1,3,5}
{
	\begin{scope}
		\clip ({2.1*(\x-1)*\a},-0.3*\a)rectangle({2.1*\x*\a},0.3*\a);
		\filldraw[color=red!60]({(2.1*(\x-1)+0.017)*\a},1*\a)arc(90:-90:1*\a)--({(2.1*(\x-1)+0.017)*\a},-1*\a)--({(2.1*\x-0.17)*\a},-1*\a)--({(2.1*\x-0.017)*\a},-1*\a)arc(-90:-270:1*\a)--cycle;
	\end{scope}
}
\foreach \x in {2,4}
{
	\begin{scope}
		\clip ({2.1*(\x-1)*\a},-0.3*\a)rectangle({2.1*\x*\a},0.3*\a);
		\filldraw[color=orange!40]({(2.1*(\x-1)+0.017)*\a},1*\a)arc(90:-90:1*\a)--({(2.1*(\x-1)+0.017)*\a},-1*\a)--({(2.1*\x-0.17)*\a},-1*\a)--({(2.1*\x-0.017)*\a},-1*\a)arc(-90:-270:1*\a)--cycle;
	\end{scope}
}

\foreach \x in {1,2,3,4,5,6}
{
	\node [scale=\a] at ({2.1*(\x-1)*\a},-0.7*\a) {$D_\x$};
}

\node [scale=1.2*\a] at (0,0) {1};
\node [scale=1.2*\a] at (2.1*\a,0) {$-1$};
\node [scale=1.2*\a] at (4.2*\a,0) {$-1$};
\node [scale=1.2*\a] at (6.3*\a,0) {1};
\node [scale=1.2*\a] at (8.4*\a,0) {1};
\node [scale=1.2*\a] at (10.5*\a,0) {$-1$};

\node  at  (-1.1*\a,-1.2*\a){};
\node  at  (11.6*\a,1.2*\a){};

\node  at (5.25*\a,-1.5*\a) {(d) $u^4$};
\end{tikzpicture}
\begin{tikzpicture}
\def\a{0.57}

\foreach \x in {1,2,3,4,5,6}
{
\shade[shading=ball, ball color=blue!30!white, opacity=0.8]   ({2.1*(\x-1)*\a},0) circle[radius=\a];
}

\foreach \x in {1,2,4,5}
{
	\begin{scope}
		\clip ({2.1*(\x-1)*\a},-0.3*\a)rectangle({2.1*\x*\a},0.3*\a);
		\filldraw[color=red!60]({(2.1*(\x-1)+0.017)*\a},1*\a)arc(90:-90:1*\a)--({(2.1*(\x-1)+0.017)*\a},-1*\a)--({(2.1*\x-0.17)*\a},-1*\a)--({(2.1*\x-0.017)*\a},-1*\a)arc(-90:-270:1*\a)--cycle;
	\end{scope}
}
\foreach \x in {3}
{
	\begin{scope}
		\clip ({2.1*(\x-1)*\a},-0.3*\a)rectangle({2.1*\x*\a},0.3*\a);
		\filldraw[color=orange!40]({(2.1*(\x-1)+0.017)*\a},1*\a)arc(90:-90:1*\a)--({(2.1*(\x-1)+0.017)*\a},-1*\a)--({(2.1*\x-0.17)*\a},-1*\a)--({(2.1*\x-0.017)*\a},-1*\a)arc(-90:-270:1*\a)--cycle;
	\end{scope}
}

\foreach \x in {1,2,3,4,5,6}
{
	\node [scale=\a] at ({2.1*(\x-1)*\a},-0.7*\a) {$D_\x$};
}

\node [scale=1.2*\a] at (0,0) {1};
\node [scale=1.2*\a] at (2.1*\a,0) {$-2$};
\node [scale=1.2*\a] at (4.2*\a,0) {1};
\node [scale=1.2*\a] at (6.3*\a,0) {1};
\node [scale=1.2*\a] at (8.4*\a,0) {$-2$};
\node [scale=1.2*\a] at (10.5*\a,0) {1};

\node  at  (-1.1*\a,-1.2*\a){};
\node  at  (11.6*\a,1.2*\a){};

\node  at (5.25*\a,-1.5*\a) {(e) $u^5$};
\end{tikzpicture}
\begin{tikzpicture}
\def\a{0.57}

\foreach \x in {1,2,3,4,5,6}
{
\shade[shading=ball, ball color=blue!30!white, opacity=0.8]   ({2.1*(\x-1)*\a},0) circle[radius=\a];
}

\foreach \x in {1,2,3,4,5}
{
	\begin{scope}
		\clip ({2.1*(\x-1)*\a},-0.3*\a)rectangle({2.1*\x*\a},0.3*\a);
		\filldraw[color=red!60]({(2.1*(\x-1)+0.017)*\a},1*\a)arc(90:-90:1*\a)--({(2.1*(\x-1)+0.017)*\a},-1*\a)--({(2.1*\x-0.17)*\a},-1*\a)--({(2.1*\x-0.017)*\a},-1*\a)arc(-90:-270:1*\a)--cycle;
	\end{scope}
}

\foreach \x in {1,2,3,4,5,6}
{
	\node [scale=\a] at ({2.1*(\x-1)*\a},-0.7*\a) {$D_\x$};
}
\node [scale=1.2*\a] at (0,0) {1};
\node [scale=1.2*\a] at (2.1*\a,0) {$-1-\sqrt{3}$};
\node [scale=1.2*\a] at (4.2*\a,0) {$2+\sqrt{3}$};
\node [scale=1.2*\a] at (6.3*\a,0) {$-2-\sqrt{3}$};
\node [scale=1.2*\a] at (8.4*\a,0) {$1+\sqrt{3}$};
\node [scale=1.2*\a] at (10.5*\a,0) {$-1$};

\node  at  (-1.1*\a,-1.2*\a){};
\node  at  (11.6*\a,1.2*\a){};

\node  at (5.25*\a,-1.5*\a) {(f) $u^6$};
\end{tikzpicture}
\caption{Resonant modes of chain arrangement for $N=6$.}\label{f_c6}
\end{figure}
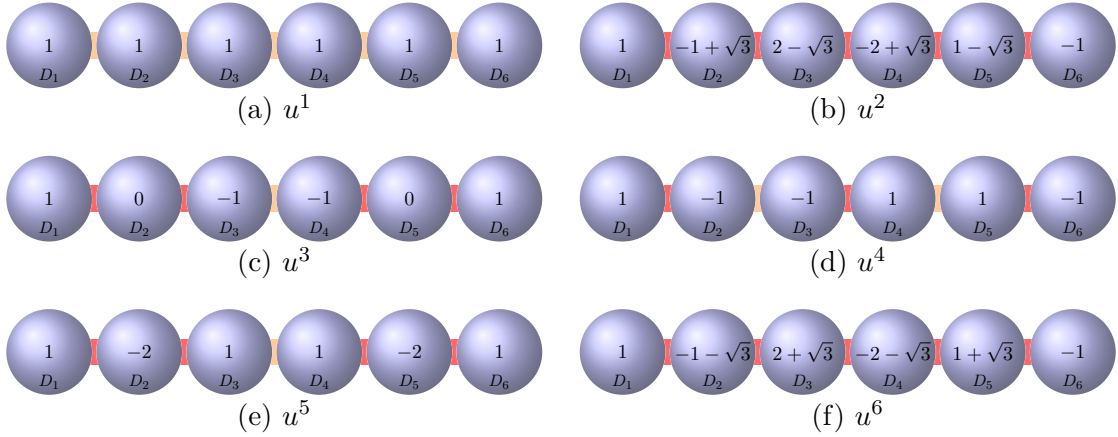

\subsubsection{Ring arrangement} 
When $N=6$, this ring arrangement is also known as a honeycomb arrangement \cite{ahy}. By Theorem \ref{Th_ring}, we have 
$\hat{a}_6^2=1,~\hat{a}_6^3=3,~\hat{a}_6^4=4$.
Hence the resonant frequencies are $\omega_1=\sqrt{\frac{3v_b^2M}{4\pi R^3}\delta}\big(1+o(1)\big)$ and
$$\omega_2, \omega_3=\eta+O(\delta^{1-\beta/2}),
~~\omega_4,\omega_5=\sqrt{3}\eta+O(\delta^{1-\beta/2}),~~\omega_6=2\eta+O(\delta^{1-\beta/2}).$$
By Theorem \ref{Th_mode_r}, up to an error of order $O(\delta^{\gamma})$, the resonant mode $u^1$ associated with the resonant frequency $\omega_1$ satisfies $u_l^1=1$. Similarly, up to an error of order $O(\delta^{\gamma})$, the resonant modes $u^i$, which correspond to the resonant frequency $\omega_i$ for $2\le i\le 6$, satisfy the following boundary conditions:
\begin{align*}
&u^2_1=k_1^2+k_2^2,&&\!\!\!u^2_2=k_2^2,&&\!\!\!u^2_3=-k_1^2,&&\!\!\!u^2_4=-k_1^2-k_2^2,&&\!\!\!u^2_5=-k_2^2,&&\!\!\!u^2_6=k_1^2;\\
&u^3_1=k_1^3+k_2^3,&&\!\!\!u^3_2=k_2^3,&&\!\!\!u^3_3=-k_1^3,&&\!\!\!u^3_4=-k_1^3-k_2^3,&&\!\!\!u^3_5=-k_2^3,&&\!\!\!u^3_6=k_1^3;\\
&u^4_1=k_1^4-k_2^4,&&\!\!\!u^4_2=k_2^4,&&\!\!\!u^4_3=-k_1^4,&&\!\!\!u^4_4=k_1^4-k_2^4,&&\!\!\!u^4_5=k_2^4,&&\!\!\!u^4_6=-k_1^4;\\
&u^5_1=k_1^5-k_2^5,&&\!\!\!u^5_2=k_2^5,&&\!\!\!u^5_3=-k_1^5,&&\!\!\!u^5_4=k_1^5-k_2^5,&&\!\!\!u^5_5=k_2^5,&&\!\!\!u^5_6=-k_1^5;\\
&u^6_1=1,&&\!\!\!u^6_2=-1,&&\!\!\!u^6_3=1,&&\!\!\!u^6_4=-1,&&\!\!\!u^6_5=1,&&\!\!\!u^6_6=-1,
\end{align*}
where $k_t^s$ are some constants for $t=1,2$ and $s=2,\dots,5$. See Figure \ref{f_r6}.
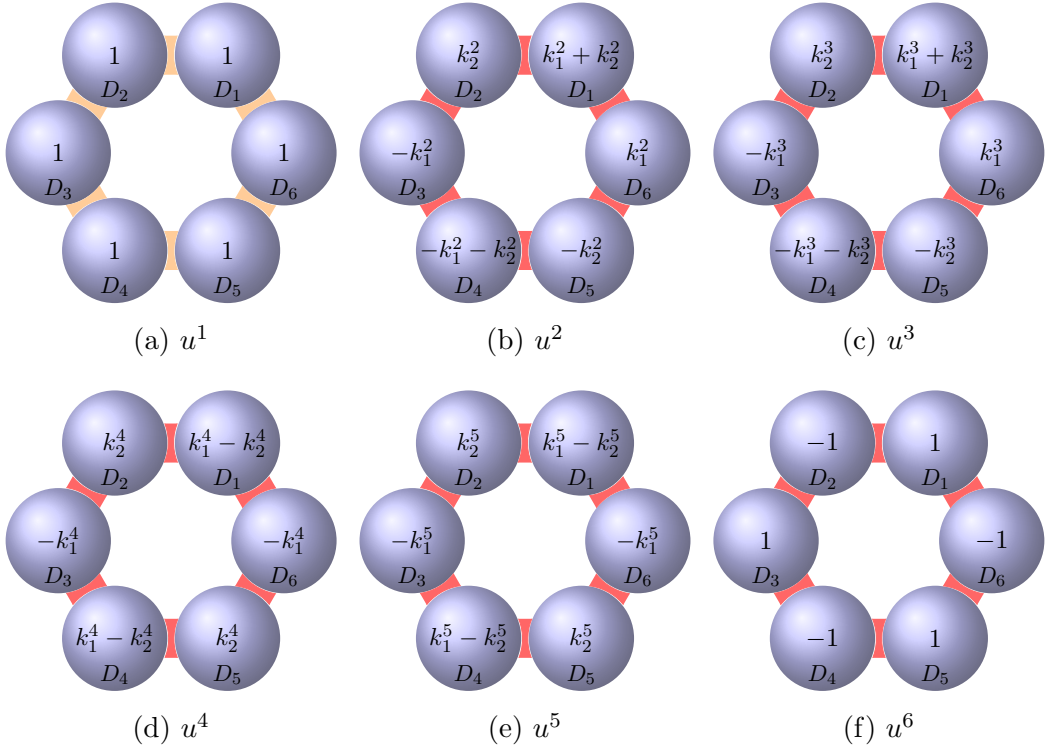
\begin{figure}[h]
\begin{tikzpicture}
\def\a{1}
\def\b{0.9}

\foreach \x in {1,2,3,4,5,6}
{
\shade[shading=ball, ball color=blue!30!white, opacity=0.8]    ({1.5*\a*cos (60*\x)},{1.5*\a*sin (60*\x)}) circle[radius=0.7*\a];
\node [scale=0.8*\a] at ({1.5*\a*cos (60*\x)},{1.5*\a*sin (60*\x)-0.48*\a}) {$D_{\x}$};
}

\foreach \x in {1,...,6} 
{
	\begin{scope}
		\clip({1.2*\a*cos (60*\x)},{1.2*\a*sin (60*\x)})--({1.8*\a*cos (60*\x)},{1.8*\a*sin (60*\x)})--({1.8*\a*cos (60*(1+\x))}, {1.8*\a*sin(60*(1+\x))})--({1.2*\a*cos (60*(1+\x))}, {1.2*\a*sin(60*(1+\x))})--cycle;
		\filldraw[color=orange!40]({0.8*\a*cos (60*\x)-0.02*\a*sin(60*\x)},{0.8*\a*sin (60*\x)+0.01*\a*cos(60*\x)})arc(-180+60*\x:-360+60*\x:0.7*\a)--({2.3*\a*cos (60*\x)-0.02*\a*sin(60*\x)},{2.2*\a*sin (60*\x)+0.01*\a*cos(60*\x)})--({2.2*\a*cos (60*(1+\x))+0.015*\a*sin(60*(1+\x))}, {2.2*\a*sin(60*(1+\x))-0.015*\a*cos(60*(1+\x))})--({2.2*\a*cos (60*(1+\x))+0.015*\a*sin(60*(1+\x))}, {2.2*\a*sin(60*(1+\x))-0.015*\a*cos(60*(1+\x))})arc(60*(1+\x):-180+60*(1+\x):0.7*\a)--cycle;
	\end{scope}
}

\node [scale=0.9*\a] at ({1.5*\a*cos (60*(1))},{1.5*\a*sin (60*(1))}) {1};
\node [scale=0.9*\a] at ({1.5*\a*cos (60*(2))},{1.5*\a*sin (60*(2))}) {1};
\node [scale=0.9*\a] at ({1.5*\a*cos (60*(3))},{1.5*\a*sin (60*(3))}) {1};
\node [scale=0.9*\a] at ({1.5*\a*cos (60*(4))},{1.5*\a*sin (60*(4))}) {1};
\node [scale=0.9*\a] at ({1.5*\a*cos (60*(5))},{1.5*\a*sin (60*(5))}) {1};
\node [scale=0.9*\a] at ({1.5*\a*cos (60*(6))},{1.5*\a*sin (60*(6))}) {1};

\node at (2.4*\b,2.4*\b){};
\node at (-2.4*\b,-2.4*\b){};

\node at (0,-2.5*\a) {(a) $u^1$};
\end{tikzpicture}
\begin{tikzpicture}
\def\a{1}
\def\b{0.9}
\foreach \x in {1,2,3,4,5,6}
{
	\shade[shading=ball, ball color=blue!30!white, opacity=0.8]    ({1.5*\a*cos (60*\x)},{1.5*\a*sin (60*\x)}) circle[radius=0.7*\a];
	\node [scale=0.8*\a] at ({1.5*\a*cos (60*\x)},{1.5*\a*sin (60*\x)-0.48*\a}) {$D_{\x}$};
}

\foreach \x in {1,...,6} 
{
	\begin{scope}
		\clip({1.2*\a*cos (60*\x)},{1.2*\a*sin (60*\x)})--({1.8*\a*cos (60*\x)},{1.8*\a*sin (60*\x)})--({1.8*\a*cos (60*(1+\x))}, {1.8*\a*sin(60*(1+\x))})--({1.2*\a*cos (60*(1+\x))}, {1.2*\a*sin(60*(1+\x))})--cycle;
		\filldraw[color=red!60]({0.8*\a*cos (60*\x)-0.02*\a*sin(60*\x)},{0.8*\a*sin (60*\x)+0.01*\a*cos(60*\x)})arc(-180+60*\x:-360+60*\x:0.7*\a)--({2.3*\a*cos (60*\x)-0.02*\a*sin(60*\x)},{2.2*\a*sin (60*\x)+0.01*\a*cos(60*\x)})--({2.2*\a*cos (60*(1+\x))+0.015*\a*sin(60*(1+\x))}, {2.2*\a*sin(60*(1+\x))-0.015*\a*cos(60*(1+\x))})--({2.2*\a*cos (60*(1+\x))+0.015*\a*sin(60*(1+\x))}, {2.2*\a*sin(60*(1+\x))-0.015*\a*cos(60*(1+\x))})arc(60*(1+\x):-180+60*(1+\x):0.7*\a)--cycle;
	\end{scope}
}

\node [scale=0.85*\a] at ({1.5*\a*cos (60*(1))},{1.5*\a*sin (60*(1))}) {$k_1^2+k_2^2$};
\node [scale=0.85*\a] at ({1.5*\a*cos (60*(2))},{1.5*\a*sin (60*(2))}) {$k_2^2$};
\node [scale=0.85*\a] at ({1.5*\a*cos (60*(3))},{1.5*\a*sin (60*(3))}) {$-k_1^2$};
\node [scale=0.85*\a] at ({1.5*\a*cos (60*(4))},{1.5*\a*sin (60*(4))}) {$-k_1^2-k_2^2$};
\node [scale=0.85*\a] at ({1.5*\a*cos (60*(5))},{1.5*\a*sin (60*(5))}) {$-k_2^2$};
\node [scale=0.85*\a] at ({1.5*\a*cos (60*(6))},{1.5*\a*sin (60*(6))}) {$k_1^2$};

\node at (2.4*\b,2.4*\b){};
\node at (-2.4*\b,-2.4*\b){};

\node at (0,-2.5*\a) {(b) $u^2$};
\end{tikzpicture}
\begin{tikzpicture}
\def\a{1}
\def\b{0.9}
\foreach \x in {1,2,3,4,5,6}
{
	\shade[shading=ball, ball color=blue!30!white, opacity=0.8]    ({1.5*\a*cos (60*\x)},{1.5*\a*sin (60*\x)}) circle[radius=0.7*\a];
	\node [scale=0.8*\a] at ({1.5*\a*cos (60*\x)},{1.5*\a*sin (60*\x)-0.48*\a}) {$D_{\x}$};
}

\foreach \x in {1,...,6} 
{
	\begin{scope}
		\clip({1.2*\a*cos (60*\x)},{1.2*\a*sin (60*\x)})--({1.8*\a*cos (60*\x)},{1.8*\a*sin (60*\x)})--({1.8*\a*cos (60*(1+\x))}, {1.8*\a*sin(60*(1+\x))})--({1.2*\a*cos (60*(1+\x))}, {1.2*\a*sin(60*(1+\x))})--cycle;
		\filldraw[color=red!60]({0.8*\a*cos (60*\x)-0.02*\a*sin(60*\x)},{0.8*\a*sin (60*\x)+0.01*\a*cos(60*\x)})arc(-180+60*\x:-360+60*\x:0.7*\a)--({2.3*\a*cos (60*\x)-0.02*\a*sin(60*\x)},{2.2*\a*sin (60*\x)+0.01*\a*cos(60*\x)})--({2.2*\a*cos (60*(1+\x))+0.015*\a*sin(60*(1+\x))}, {2.2*\a*sin(60*(1+\x))-0.015*\a*cos(60*(1+\x))})--({2.2*\a*cos (60*(1+\x))+0.015*\a*sin(60*(1+\x))}, {2.2*\a*sin(60*(1+\x))-0.015*\a*cos(60*(1+\x))})arc(60*(1+\x):-180+60*(1+\x):0.7*\a)--cycle;
	\end{scope}
}
\node [scale=0.85*\a] at ({1.5*\a*cos (60*(1))},{1.5*\a*sin (60*(1))}) {$k_1^3+k_2^3$};
\node [scale=0.85*\a] at ({1.5*\a*cos (60*(2))},{1.5*\a*sin (60*(2))}) {$k_2^3$};
\node [scale=0.85*\a] at ({1.5*\a*cos (60*(3))},{1.5*\a*sin (60*(3))}) {$-k_1^3$};
\node [scale=0.85*\a] at ({1.5*\a*cos (60*(4))},{1.5*\a*sin (60*(4))}) {$-k_1^3-k_2^3$};
\node [scale=0.85*\a] at ({1.5*\a*cos (60*(5))},{1.5*\a*sin (60*(5))}) {$-k_2^3$};
\node [scale=0.85*\a] at ({1.5*\a*cos (60*(6))},{1.5*\a*sin (60*(6))}) {$k_1^3$};

\node at (2.4*\b,2.4*\b){};
\node at (-2.4*\b,-2.4*\b){};

\node at (0,-2.5*\a) {(c) $u^3$};
\end{tikzpicture}
\begin{tikzpicture}
\def\a{1}
\def\b{0.9}
\foreach \x in {1,2,3,4,5,6}
{
	\shade[shading=ball, ball color=blue!30!white, opacity=0.8]    ({1.5*\a*cos (60*\x)},{1.5*\a*sin (60*\x)}) circle[radius=0.7*\a];
	\node [scale=0.8*\a] at ({1.5*\a*cos (60*\x)},{1.5*\a*sin (60*\x)-0.48*\a}) {$D_{\x}$};
}

\foreach \x in {1,...,6} 
{
	\begin{scope}
		\clip({1.2*\a*cos (60*\x)},{1.2*\a*sin (60*\x)})--({1.8*\a*cos (60*\x)},{1.8*\a*sin (60*\x)})--({1.8*\a*cos (60*(1+\x))}, {1.8*\a*sin(60*(1+\x))})--({1.2*\a*cos (60*(1+\x))}, {1.2*\a*sin(60*(1+\x))})--cycle;
		\filldraw[color=red!60]({0.8*\a*cos (60*\x)-0.02*\a*sin(60*\x)},{0.8*\a*sin (60*\x)+0.01*\a*cos(60*\x)})arc(-180+60*\x:-360+60*\x:0.7*\a)--({2.3*\a*cos (60*\x)-0.02*\a*sin(60*\x)},{2.2*\a*sin (60*\x)+0.01*\a*cos(60*\x)})--({2.2*\a*cos (60*(1+\x))+0.015*\a*sin(60*(1+\x))}, {2.2*\a*sin(60*(1+\x))-0.015*\a*cos(60*(1+\x))})--({2.2*\a*cos (60*(1+\x))+0.015*\a*sin(60*(1+\x))}, {2.2*\a*sin(60*(1+\x))-0.015*\a*cos(60*(1+\x))})arc(60*(1+\x):-180+60*(1+\x):0.7*\a)--cycle;
	\end{scope}
}

\node [scale=0.85*\a] at ({1.5*\a*cos (60*(1))},{1.5*\a*sin (60*(1))}) {$k_1^4-k_2^4$};
\node [scale=0.85*\a] at ({1.5*\a*cos (60*(2))},{1.5*\a*sin (60*(2))}) {$k_2^4$};
\node [scale=0.85*\a] at ({1.5*\a*cos (60*(3))},{1.5*\a*sin (60*(3))}) {$-k_1^4$};
\node [scale=0.85*\a] at ({1.5*\a*cos (60*(4))},{1.5*\a*sin (60*(4))}) {$k_1^4-k_2^4$};
\node [scale=0.85*\a] at ({1.5*\a*cos (60*(5))},{1.5*\a*sin (60*(5))}) {$k_2^4$};
\node [scale=0.85*\a] at ({1.5*\a*cos (60*(6))},{1.5*\a*sin (60*(6))}) {$-k_1^4$};

\node at (2.4*\b,2.4*\b){};
\node at (-2.4*\b,-2.4*\b){};

\node at (0,-2.5*\a) {(d) $u^4$};
\end{tikzpicture}
\begin{tikzpicture}
\def\a{1}
\def\b{0.9}
\foreach \x in {1,2,3,4,5,6}
{
	\shade[shading=ball, ball color=blue!30!white, opacity=0.8]    ({1.5*\a*cos (60*\x)},{1.5*\a*sin (60*\x)}) circle[radius=0.7*\a];
	\node [scale=0.8*\a] at ({1.5*\a*cos (60*\x)},{1.5*\a*sin (60*\x)-0.48*\a}) {$D_{\x}$};
}

\foreach \x in {1,...,6} 
{
	\begin{scope}
		\clip({1.2*\a*cos (60*\x)},{1.2*\a*sin (60*\x)})--({1.8*\a*cos (60*\x)},{1.8*\a*sin (60*\x)})--({1.8*\a*cos (60*(1+\x))}, {1.8*\a*sin(60*(1+\x))})--({1.2*\a*cos (60*(1+\x))}, {1.2*\a*sin(60*(1+\x))})--cycle;
		\filldraw[color=red!60]({0.8*\a*cos (60*\x)-0.02*\a*sin(60*\x)},{0.8*\a*sin (60*\x)+0.01*\a*cos(60*\x)})arc(-180+60*\x:-360+60*\x:0.7*\a)--({2.3*\a*cos (60*\x)-0.02*\a*sin(60*\x)},{2.2*\a*sin (60*\x)+0.01*\a*cos(60*\x)})--({2.2*\a*cos (60*(1+\x))+0.015*\a*sin(60*(1+\x))}, {2.2*\a*sin(60*(1+\x))-0.015*\a*cos(60*(1+\x))})--({2.2*\a*cos (60*(1+\x))+0.015*\a*sin(60*(1+\x))}, {2.2*\a*sin(60*(1+\x))-0.015*\a*cos(60*(1+\x))})arc(60*(1+\x):-180+60*(1+\x):0.7*\a)--cycle;
	\end{scope}
}
\node [scale=0.85*\a] at ({1.5*\a*cos (60*(1))},{1.5*\a*sin (60*(1))}) {$k_1^5-k_2^5$};
\node [scale=0.85*\a] at ({1.5*\a*cos (60*(2))},{1.5*\a*sin (60*(2))}) {$k_2^5$};
\node [scale=0.85*\a] at ({1.5*\a*cos (60*(3))},{1.5*\a*sin (60*(3))}) {$-k_1^5$};
\node [scale=0.85*\a] at ({1.5*\a*cos (60*(4))},{1.5*\a*sin (60*(4))}) {$k_1^5-k_2^5$};
\node [scale=0.85*\a] at ({1.5*\a*cos (60*(5))},{1.5*\a*sin (60*(5))}) {$k_2^5$};
\node [scale=0.85*\a] at ({1.5*\a*cos (60*(6))},{1.5*\a*sin (60*(6))}) {$-k_1^5$};

\node at (2.4*\b,2.4*\b){};
\node at (-2.4*\b,-2.4*\b){};

\node at (0,-2.5*\a) {(e) $u^5$};
\end{tikzpicture}
\begin{tikzpicture}
\def\a{1}
\def\b{0.9}
\foreach \x in {1,2,3,4,5,6}
{
	\shade[shading=ball, ball color=blue!30!white, opacity=0.8]    ({1.5*\a*cos (60*\x)},{1.5*\a*sin (60*\x)}) circle[radius=0.7*\a];
	\node [scale=0.8*\a] at ({1.5*\a*cos (60*\x)},{1.5*\a*sin (60*\x)-0.48*\a}) {$D_{\x}$};
}

\foreach \x in {1,...,6} 
{
	\begin{scope}
		\clip({1.2*\a*cos (60*\x)},{1.2*\a*sin (60*\x)})--({1.8*\a*cos (60*\x)},{1.8*\a*sin (60*\x)})--({1.8*\a*cos (60*(1+\x))}, {1.8*\a*sin(60*(1+\x))})--({1.2*\a*cos (60*(1+\x))}, {1.2*\a*sin(60*(1+\x))})--cycle;
		\filldraw[color=red!60]({0.8*\a*cos (60*\x)-0.02*\a*sin(60*\x)},{0.8*\a*sin (60*\x)+0.01*\a*cos(60*\x)})arc(-180+60*\x:-360+60*\x:0.7*\a)--({2.3*\a*cos (60*\x)-0.02*\a*sin(60*\x)},{2.2*\a*sin (60*\x)+0.01*\a*cos(60*\x)})--({2.2*\a*cos (60*(1+\x))+0.015*\a*sin(60*(1+\x))}, {2.2*\a*sin(60*(1+\x))-0.015*\a*cos(60*(1+\x))})--({2.2*\a*cos (60*(1+\x))+0.015*\a*sin(60*(1+\x))}, {2.2*\a*sin(60*(1+\x))-0.015*\a*cos(60*(1+\x))})arc(60*(1+\x):-180+60*(1+\x):0.7*\a)--cycle;
	\end{scope}
}
\node [scale=0.9*\a] at ({1.5*\a*cos (60*(1))},{1.5*\a*sin (60*(1))}) {1};
\node [scale=0.9*\a] at ({1.5*\a*cos (60*(2))},{1.5*\a*sin (60*(2))}) {$-1$};
\node [scale=0.9*\a] at ({1.5*\a*cos (60*(3))},{1.5*\a*sin (60*(3))}) {1};
\node [scale=0.9*\a] at ({1.5*\a*cos (60*(4))},{1.5*\a*sin (60*(4))}) {$-1$};
\node [scale=0.9*\a] at ({1.5*\a*cos (60*(5))},{1.5*\a*sin (60*(5))}) {1};
\node [scale=0.9*\a] at ({1.5*\a*cos (60*(6))},{1.5*\a*sin (60*(6))}) {$-1$};

\node at (2.4*\b,2.4*\b){};
\node at (-2.4*\b,-2.4*\b){};

\node at (0,-2.5*\a) {(f) $u^6$};
\end{tikzpicture}
\caption{Resonant modes of ring arrangement for $N=6$.}\label{f_r6}
\end{figure}
\subsubsection{Matrix arrangement}
For $N=6$, with $m=2$ and $n=3$. By Theorem \ref{Th_matrix}, we obtain
$\bar{a}_6^2=1,~\bar{a}_6^3=2,~\bar{a}_6^4=\bar{a}_6^5=3,~\bar{a}_6^6=5.$
Therefore, the resonant frequencies are $\omega_1=\sqrt{\frac{3v_b^2M}{4\pi R^3}\delta}\big(1+o(1)\big)$, and 
\begin{align*}
&\omega_2=\eta+O(\delta^{1-\beta/2}),&&\hspace{-15mm}\omega_3=\sqrt{2}\eta+O(\delta^{1-\beta/2}),\\
&\omega_4,\omega_5=\sqrt{3}\eta+O(\delta^{1-\beta/2}),
&&\hspace{-15mm}\omega_6=\sqrt{5}\eta+O(\delta^{1-\beta/2}).
\end{align*}
We compute the resonant modes for the case $m=2$ and $n=3$. By solving equation $(\mathbb{C}-\tilde{\lambda}_i\mathbb{I})\boldsymbol{\alpha}^i=0$ as previously done, with an error of order $O(\delta^\gamma)$, for the resonant mode $u^1$ associated with the resonant frequency $\omega_1$, we have $u_l^1=1$. For the resonant modes $u^i$, $2\le i\le6$, we have
\begin{align*}
&u^2_1=1, &&u^2_2=0, &&u^2_3=-1, &&u^2_4=1, &&u^2_5=0,  &&u^2_6=-1;\\
&u^3_1=1, &&u^3_2=1, &&u^3_3=1, &&u^3_4=-1, &&u^3_5=-1,  &&u^3_6=-1;\\
&u^4_1=k_1^4+k_2^4, &&u^4_2=-k_1^4, &&u^4_3=-k_2^4, &&u^4_4=-k_2^4, &&u^4_5=-k_1^4,  &&u^4_6=k_1^4+k_2^4;\\
&u^5_1=k_1^5+k_2^5, &&u^5_2=-k_1^5, &&u^5_3=-k_2^5, &&u^5_4=-k_2^5, &&u^5_5=-k_1^5,  &&u^5_6=k_1^5+k_2^5;\\
&u^6_1=1, &&u^6_2=-2, &&u^6_3=1, &&u^6_4=-1, &&u^6_5=2,  &&u^6_6=-1,
\end{align*}
where $k_t^s$ are some constants for $t=1,2$ and $s=4,5$. 

From the above analysis, we can observe that $\nabla u^2$ exhibits a lower blow--up rate between the first and second rows. $\nabla u^3$ shows a lower blow--up rate between each pair of columns. Additionally, $\nabla u^4$ and $\nabla u^5$ display a lower blow--up rate within the region bounded by $D_2$ and $D_5$, see Figure \ref{f_m6}.
\begin{figure}[h]
\begin{tikzpicture}
\def\a{.7}

\shade[shading=ball, ball color=blue!30!white, opacity=0.8]   (0,0) circle[radius=\a];
\shade[shading=ball, ball color=blue!30!white, opacity=0.8]   (2.1*\a,0) circle[radius=\a];
\shade[shading=ball, ball color=blue!30!white, opacity=0.8]   (4.2*\a,0) circle[radius=\a];
\shade[shading=ball, ball color=blue!30!white, opacity=0.8]   (0,2.1*\a) circle[radius=\a];
\shade[shading=ball, ball color=blue!30!white, opacity=0.8]   (2.1*\a,2.1*\a) circle[radius=\a];
\shade[shading=ball, ball color=blue!30!white, opacity=0.8]   (4.2*\a,2.1*\a) circle[radius=\a];

\begin{scope}
	\clip (0.5*\a,-0.3*\a)rectangle(1.6*\a,0.3*\a);
	\filldraw[color=orange!40](0.01,1*\a)arc(90:-90:1*\a)--(0.01,-1*\a)--(2.087*\a,-1*\a)--(2.087*\a,-1*\a)arc(-90:-270:1*\a)--cycle;
\end{scope}
\begin{scope}
	\clip (2.5*\a,-0.3*\a)rectangle(3.6*\a,0.3*\a);
	\filldraw[color=orange!40](2.11*\a,1*\a)arc(90:-90:1*\a)--(2.11*\a,-1*\a)--(4.187*\a,-1*\a)--(4.187*\a,-1*\a)arc(-90:-270:1*\a)--cycle;
\end{scope}
\begin{scope}
	\clip (0.5*\a,1.8*\a)rectangle(1.6*\a,2.4*\a);
	\filldraw[color=orange!40](0.01,3.1*\a)arc(90:-90:1*\a)--(0.01,1.1*\a)--(2.087*\a,1.1*\a)--(2.087*\a,1.1*\a)arc(-90:-270:1*\a)--cycle;
\end{scope}
\begin{scope}
	\clip (2.5*\a,1.8*\a)rectangle(3.6*\a,2.4*\a);
	\filldraw[color=orange!40](2.11*\a,3.1*\a)arc(90:-90:1*\a)--(2.11*\a,1.1*\a)--(4.187*\a,1.1*\a)--(4.187*\a,1.1*\a)arc(-90:-270:1*\a)--cycle;
\end{scope}
\begin{scope}
	\clip (-0.3*\a,0.5*\a)rectangle(0.3*\a,1.6*\a);
	\filldraw[color=orange!40](1*\a,0.01)arc(0:180:1*\a)--(-1*\a,0.01)--(-1*\a,2.087*\a)--(-1*\a,2.087*\a)arc(-180:0:1*\a)--cycle;
\end{scope}
\begin{scope}
	\clip (1.8*\a,0.5*\a)rectangle(2.4*\a,1.6*\a);
	\filldraw[color=orange!40](3.1*\a,0.01)arc(0:180:1*\a)--(1.1*\a,0.01)--(1.1*\a,2.087*\a)--(1.1*\a,2.087*\a)arc(-180:0:1*\a)--cycle;
\end{scope}
\begin{scope}
	\clip (3.9*\a,0.5*\a)rectangle(4.5*\a,1.6*\a);
	\filldraw[color=orange!40](5.2*\a,0.01)arc(0:180:1*\a)--(3.2*\a,0.01)--(3.2*\a,2.087*\a)--(3.2*\a,2.087*\a)arc(-180:0:1*\a)--cycle;
\end{scope}

\node [scale=1.2*\a] at (0,0) {1};
\node [scale=1.2*\a] at (2.1*\a,0) {1};
\node [scale=1.2*\a] at (4.2*\a,0) {1};
\node [scale=1.2*\a] at (0,2.1*\a) {1};
\node [scale=1.2*\a] at (2.1*\a,2.1*\a) {1};
\node [scale=1.2*\a] at (4.2*\a,2.1*\a) {1};

\node [scale=\a] at (0,-0.7*\a) {$D_1$};
\node [scale=\a] at (2.1*\a,-0.7*\a) {$D_2$};
\node [scale=\a] at (4.2*\a,-0.7*\a) {$D_3$};
\node [scale=\a] at (0,1.4*\a) {$D_4$};
\node [scale=\a] at (2.1*\a,1.4*\a) {$D_5$};
\node [scale=\a] at (4.2*\a,1.4*\a) {$D_6$};

\node at (4.4*\a,3.3*\a) {};
\node at (-1.3*\a,-1.2*\a) {};
\node at  (2.1*\a,-1.55*\a) {(a) $u^1$};

\end{tikzpicture}
\begin{tikzpicture}
\def\a{.7}

\shade[shading=ball, ball color=blue!30!white, opacity=0.8]   (0,0) circle[radius=\a];
\shade[shading=ball, ball color=blue!30!white, opacity=0.8]   (2.1*\a,0) circle[radius=\a];
\shade[shading=ball, ball color=blue!30!white, opacity=0.8]   (4.2*\a,0) circle[radius=\a];
\shade[shading=ball, ball color=blue!30!white, opacity=0.8]   (0,2.1*\a) circle[radius=\a];
\shade[shading=ball, ball color=blue!30!white, opacity=0.8]   (2.1*\a,2.1*\a) circle[radius=\a];
\shade[shading=ball, ball color=blue!30!white, opacity=0.8]   (4.2*\a,2.1*\a) circle[radius=\a];

\begin{scope}
	\clip (0.5*\a,-0.3*\a)rectangle(1.6*\a,0.3*\a);
	\filldraw[color=red!60](0.01,1*\a)arc(90:-90:1*\a)--(0.01,-1*\a)--(2.087*\a,-1*\a)--(2.087*\a,-1*\a)arc(-90:-270:1*\a)--cycle;
\end{scope}
\begin{scope}
	\clip (2.5*\a,-0.3*\a)rectangle(3.6*\a,0.3*\a);
	\filldraw[color=red!60](2.11*\a,1*\a)arc(90:-90:1*\a)--(2.11*\a,-1*\a)--(4.187*\a,-1*\a)--(4.187*\a,-1*\a)arc(-90:-270:1*\a)--cycle;
\end{scope}
\begin{scope}
	\clip (0.5*\a,1.8*\a)rectangle(1.6*\a,2.4*\a);
	\filldraw[color=red!60](0.01,3.1*\a)arc(90:-90:1*\a)--(0.01,1.1*\a)--(2.087*\a,1.1*\a)--(2.087*\a,1.1*\a)arc(-90:-270:1*\a)--cycle;
\end{scope}
\begin{scope}
	\clip (2.5*\a,1.8*\a)rectangle(3.6*\a,2.4*\a);
	\filldraw[color=red!60](2.11*\a,3.1*\a)arc(90:-90:1*\a)--(2.11*\a,1.1*\a)--(4.187*\a,1.1*\a)--(4.187*\a,1.1*\a)arc(-90:-270:1*\a)--cycle;
\end{scope}
\begin{scope}
	\clip (-0.3*\a,0.5*\a)rectangle(0.3*\a,1.6*\a);
	\filldraw[color=orange!40](1*\a,0.01)arc(0:180:1*\a)--(-1*\a,0.01)--(-1*\a,2.087*\a)--(-1*\a,2.087*\a)arc(-180:0:1*\a)--cycle;
\end{scope}
\begin{scope}
	\clip (1.8*\a,0.5*\a)rectangle(2.4*\a,1.6*\a);
	\filldraw[color=orange!40](3.1*\a,0.01)arc(0:180:1*\a)--(1.1*\a,0.01)--(1.1*\a,2.087*\a)--(1.1*\a,2.087*\a)arc(-180:0:1*\a)--cycle;
\end{scope}
\begin{scope}
	\clip (3.9*\a,0.5*\a)rectangle(4.5*\a,1.6*\a);
	\filldraw[color=orange!40](5.2*\a,0.01)arc(0:180:1*\a)--(3.2*\a,0.01)--(3.2*\a,2.087*\a)--(3.2*\a,2.087*\a)arc(-180:0:1*\a)--cycle;
\end{scope}

\node [scale=1.2*\a] at (0,0) {1};
\node [scale=1.2*\a] at (2.1*\a,0) {0};
\node [scale=1.2*\a] at (4.2*\a,0) {-1};
\node [scale=1.2*\a] at (0,2.1*\a) {1};
\node [scale=1.2*\a] at (2.1*\a,2.1*\a) {0};
\node [scale=1.2*\a] at (4.2*\a,2.1*\a) {-1};

\node [scale=\a] at (0,-0.7*\a) {$D_1$};
\node [scale=\a] at (2.1*\a,-0.7*\a) {$D_2$};
\node [scale=\a] at (4.2*\a,-0.7*\a) {$D_3$};
\node [scale=\a] at (0,1.4*\a) {$D_4$};
\node [scale=\a] at (2.1*\a,1.4*\a) {$D_5$};
\node [scale=\a] at (4.2*\a,1.4*\a) {$D_6$};

\node at (4.4*\a,3.3*\a) {};
\node at (-1.3*\a,-1.2*\a) {};
\node at  (2.1*\a,-1.55*\a) {(b) $u^2$};

\end{tikzpicture}
\begin{tikzpicture}
\def\a{.7}

\shade[shading=ball, ball color=blue!30!white, opacity=0.8]   (0,0) circle[radius=\a];
\shade[shading=ball, ball color=blue!30!white, opacity=0.8]   (2.1*\a,0) circle[radius=\a];
\shade[shading=ball, ball color=blue!30!white, opacity=0.8]   (4.2*\a,0) circle[radius=\a];
\shade[shading=ball, ball color=blue!30!white, opacity=0.8]   (0,2.1*\a) circle[radius=\a];
\shade[shading=ball, ball color=blue!30!white, opacity=0.8]   (2.1*\a,2.1*\a) circle[radius=\a];
\shade[shading=ball, ball color=blue!30!white, opacity=0.8]   (4.2*\a,2.1*\a) circle[radius=\a];

\begin{scope}
	\clip (0.5*\a,-0.3*\a)rectangle(1.6*\a,0.3*\a);
	\filldraw[color=orange!40](0.01,1*\a)arc(90:-90:1*\a)--(0.01,-1*\a)--(2.087*\a,-1*\a)--(2.087*\a,-1*\a)arc(-90:-270:1*\a)--cycle;
\end{scope}
\begin{scope}
	\clip (2.5*\a,-0.3*\a)rectangle(3.6*\a,0.3*\a);
	\filldraw[color=orange!40](2.11*\a,1*\a)arc(90:-90:1*\a)--(2.11*\a,-1*\a)--(4.187*\a,-1*\a)--(4.187*\a,-1*\a)arc(-90:-270:1*\a)--cycle;
\end{scope}
\begin{scope}
	\clip (0.5*\a,1.8*\a)rectangle(1.6*\a,2.4*\a);
	\filldraw[color=orange!40](0.01,3.1*\a)arc(90:-90:1*\a)--(0.01,1.1*\a)--(2.087*\a,1.1*\a)--(2.087*\a,1.1*\a)arc(-90:-270:1*\a)--cycle;
\end{scope}
\begin{scope}
	\clip (2.5*\a,1.8*\a)rectangle(3.6*\a,2.4*\a);
	\filldraw[color=orange!40](2.11*\a,3.1*\a)arc(90:-90:1*\a)--(2.11*\a,1.1*\a)--(4.187*\a,1.1*\a)--(4.187*\a,1.1*\a)arc(-90:-270:1*\a)--cycle;
\end{scope}
\begin{scope}
	\clip (-0.3*\a,0.5*\a)rectangle(0.3*\a,1.6*\a);
	\filldraw[color=red!60](1*\a,0.01)arc(0:180:1*\a)--(-1*\a,0.01)--(-1*\a,2.087*\a)--(-1*\a,2.087*\a)arc(-180:0:1*\a)--cycle;
\end{scope}
\begin{scope}
	\clip (1.8*\a,0.5*\a)rectangle(2.4*\a,1.6*\a);
	\filldraw[color=red!60](3.1*\a,0.01)arc(0:180:1*\a)--(1.1*\a,0.01)--(1.1*\a,2.087*\a)--(1.1*\a,2.087*\a)arc(-180:0:1*\a)--cycle;
\end{scope}
\begin{scope}
	\clip (3.9*\a,0.5*\a)rectangle(4.5*\a,1.6*\a);
	\filldraw[color=red!60](5.2*\a,0.01)arc(0:180:1*\a)--(3.2*\a,0.01)--(3.2*\a,2.087*\a)--(3.2*\a,2.087*\a)arc(-180:0:1*\a)--cycle;
\end{scope}

\node [scale=1.2*\a] at (0,0) {1};
\node [scale=1.2*\a] at (2.1*\a,0) {1};
\node [scale=1.2*\a] at (4.2*\a,0) {1};
\node [scale=1.2*\a] at (0,2.1*\a) {-1};
\node [scale=1.2*\a] at (2.1*\a,2.1*\a) {-1};
\node [scale=1.2*\a] at (4.2*\a,2.1*\a) {-1};

\node [scale=\a] at (0,-0.7*\a) {$D_1$};
\node [scale=\a] at (2.1*\a,-0.7*\a) {$D_2$};
\node [scale=\a] at (4.2*\a,-0.7*\a) {$D_3$};
\node [scale=\a] at (0,1.4*\a) {$D_4$};
\node [scale=\a] at (2.1*\a,1.4*\a) {$D_5$};
\node [scale=\a] at (4.2*\a,1.4*\a) {$D_6$};

\node at (4.4*\a,3.3*\a) {};
\node at (-1.3*\a,-1.2*\a) {};
\node at  (2.1*\a,-1.55*\a) {(c) $u^3$};

\end{tikzpicture}
\begin{tikzpicture}
\def\a{.7}

\shade[shading=ball, ball color=blue!30!white, opacity=0.8]   (0,0) circle[radius=\a];
\shade[shading=ball, ball color=blue!30!white, opacity=0.8]   (2.1*\a,0) circle[radius=\a];
\shade[shading=ball, ball color=blue!30!white, opacity=0.8]   (4.2*\a,0) circle[radius=\a];
\shade[shading=ball, ball color=blue!30!white, opacity=0.8]   (0,2.1*\a) circle[radius=\a];
\shade[shading=ball, ball color=blue!30!white, opacity=0.8]   (2.1*\a,2.1*\a) circle[radius=\a];
\shade[shading=ball, ball color=blue!30!white, opacity=0.8]   (4.2*\a,2.1*\a) circle[radius=\a];

\begin{scope}
	\clip (0.5*\a,-0.3*\a)rectangle(1.6*\a,0.3*\a);
	\filldraw[color=red!60](0.01,1*\a)arc(90:-90:1*\a)--(0.01,-1*\a)--(2.087*\a,-1*\a)--(2.087*\a,-1*\a)arc(-90:-270:1*\a)--cycle;
\end{scope}
\begin{scope}
	\clip (2.5*\a,-0.3*\a)rectangle(3.6*\a,0.3*\a);
	\filldraw[color=red!60](2.11*\a,1*\a)arc(90:-90:1*\a)--(2.11*\a,-1*\a)--(4.187*\a,-1*\a)--(4.187*\a,-1*\a)arc(-90:-270:1*\a)--cycle;
\end{scope}
\begin{scope}
	\clip (0.5*\a,1.8*\a)rectangle(1.6*\a,2.4*\a);
	\filldraw[color=red!60](0.01,3.1*\a)arc(90:-90:1*\a)--(0.01,1.1*\a)--(2.087*\a,1.1*\a)--(2.087*\a,1.1*\a)arc(-90:-270:1*\a)--cycle;
\end{scope}
\begin{scope}
	\clip (2.5*\a,1.8*\a)rectangle(3.6*\a,2.4*\a);
	\filldraw[color=red!60](2.11*\a,3.1*\a)arc(90:-90:1*\a)--(2.11*\a,1.1*\a)--(4.187*\a,1.1*\a)--(4.187*\a,1.1*\a)arc(-90:-270:1*\a)--cycle;
\end{scope}
\begin{scope}
	\clip (-0.3*\a,0.5*\a)rectangle(0.3*\a,1.6*\a);
	\filldraw[color=red!60](1*\a,0.01)arc(0:180:1*\a)--(-1*\a,0.01)--(-1*\a,2.087*\a)--(-1*\a,2.087*\a)arc(-180:0:1*\a)--cycle;
\end{scope}
\begin{scope}
	\clip (1.8*\a,0.5*\a)rectangle(2.4*\a,1.6*\a);
	\filldraw[color=orange!40](3.1*\a,0.01)arc(0:180:1*\a)--(1.1*\a,0.01)--(1.1*\a,2.087*\a)--(1.1*\a,2.087*\a)arc(-180:0:1*\a)--cycle;
\end{scope}
\begin{scope}
	\clip (3.9*\a,0.5*\a)rectangle(4.5*\a,1.6*\a);
	\filldraw[color=red!60](5.2*\a,0.01)arc(0:180:1*\a)--(3.2*\a,0.01)--(3.2*\a,2.087*\a)--(3.2*\a,2.087*\a)arc(-180:0:1*\a)--cycle;
\end{scope}

\node [scale=1.2*\a] at (0,0) {$k_1^4+k_2^4$};
\node [scale=1.2*\a] at (2.1*\a,0) {$-k_1^4$};
\node [scale=1.2*\a] at (4.2*\a,0) {$-k_2^4$};
\node [scale=1.2*\a] at (0,2.1*\a) {$-k_2^4$};
\node [scale=1.2*\a] at (2.1*\a,2.1*\a) {$-k_1^4$};
\node [scale=1.2*\a] at (4.2*\a,2.1*\a) {$k_1^4+k_2^4$};

\node [scale=\a] at (0,-0.7*\a) {$D_1$};
\node [scale=\a] at (2.1*\a,-0.7*\a) {$D_2$};
\node [scale=\a] at (4.2*\a,-0.7*\a) {$D_3$};
\node [scale=\a] at (0,1.4*\a) {$D_4$};
\node [scale=\a] at (2.1*\a,1.4*\a) {$D_5$};
\node [scale=\a] at (4.2*\a,1.4*\a) {$D_6$};

\node at (4.4*\a,3.3*\a) {};
\node at (-1.3*\a,-1.2*\a) {};
\node at  (2.1*\a,-1.55*\a) {(d) $u^4$};

\end{tikzpicture}
\begin{tikzpicture}
\def\a{.7}

\shade[shading=ball, ball color=blue!30!white, opacity=0.8]   (0,0) circle[radius=\a];
\shade[shading=ball, ball color=blue!30!white, opacity=0.8]   (2.1*\a,0) circle[radius=\a];
\shade[shading=ball, ball color=blue!30!white, opacity=0.8]   (4.2*\a,0) circle[radius=\a];
\shade[shading=ball, ball color=blue!30!white, opacity=0.8]   (0,2.1*\a) circle[radius=\a];
\shade[shading=ball, ball color=blue!30!white, opacity=0.8]   (2.1*\a,2.1*\a) circle[radius=\a];
\shade[shading=ball, ball color=blue!30!white, opacity=0.8]   (4.2*\a,2.1*\a) circle[radius=\a];

\begin{scope}
	\clip (0.5*\a,-0.3*\a)rectangle(1.6*\a,0.3*\a);
	\filldraw[color=red!60](0.01,1*\a)arc(90:-90:1*\a)--(0.01,-1*\a)--(2.087*\a,-1*\a)--(2.087*\a,-1*\a)arc(-90:-270:1*\a)--cycle;
\end{scope}
\begin{scope}
	\clip (2.5*\a,-0.3*\a)rectangle(3.6*\a,0.3*\a);
	\filldraw[color=red!60](2.11*\a,1*\a)arc(90:-90:1*\a)--(2.11*\a,-1*\a)--(4.187*\a,-1*\a)--(4.187*\a,-1*\a)arc(-90:-270:1*\a)--cycle;
\end{scope}
\begin{scope}
	\clip (0.5*\a,1.8*\a)rectangle(1.6*\a,2.4*\a);
	\filldraw[color=red!60](0.01,3.1*\a)arc(90:-90:1*\a)--(0.01,1.1*\a)--(2.087*\a,1.1*\a)--(2.087*\a,1.1*\a)arc(-90:-270:1*\a)--cycle;
\end{scope}
\begin{scope}
	\clip (2.5*\a,1.8*\a)rectangle(3.6*\a,2.4*\a);
	\filldraw[color=red!60](2.11*\a,3.1*\a)arc(90:-90:1*\a)--(2.11*\a,1.1*\a)--(4.187*\a,1.1*\a)--(4.187*\a,1.1*\a)arc(-90:-270:1*\a)--cycle;
\end{scope}
\begin{scope}
	\clip (-0.3*\a,0.5*\a)rectangle(0.3*\a,1.6*\a);
	\filldraw[color=red!60](1*\a,0.01)arc(0:180:1*\a)--(-1*\a,0.01)--(-1*\a,2.087*\a)--(-1*\a,2.087*\a)arc(-180:0:1*\a)--cycle;
\end{scope}
\begin{scope}
	\clip (1.8*\a,0.5*\a)rectangle(2.4*\a,1.6*\a);
	\filldraw[color=orange!40](3.1*\a,0.01)arc(0:180:1*\a)--(1.1*\a,0.01)--(1.1*\a,2.087*\a)--(1.1*\a,2.087*\a)arc(-180:0:1*\a)--cycle;
\end{scope}
\begin{scope}
	\clip (3.9*\a,0.5*\a)rectangle(4.5*\a,1.6*\a);
	\filldraw[color=red!60](5.2*\a,0.01)arc(0:180:1*\a)--(3.2*\a,0.01)--(3.2*\a,2.087*\a)--(3.2*\a,2.087*\a)arc(-180:0:1*\a)--cycle;
\end{scope}

\node [scale=1.2*\a] at (0,0) {$k_1^5+k_2^5$};
\node [scale=1.2*\a] at (2.1*\a,0) {$-k_1^5$};
\node [scale=1.2*\a] at (4.2*\a,0) {$-k_2^5$};
\node [scale=1.2*\a] at (0,2.1*\a) {$-k_2^5$};
\node [scale=1.2*\a] at (2.1*\a,2.1*\a) {$-k_1^5$};
\node [scale=1.2*\a] at (4.2*\a,2.1*\a) {$k_1^5+k_2^5$};

\node [scale=\a] at (0,-0.7*\a) {$D_1$};
\node [scale=\a] at (2.1*\a,-0.7*\a) {$D_2$};
\node [scale=\a] at (4.2*\a,-0.7*\a) {$D_3$};
\node [scale=\a] at (0,1.4*\a) {$D_4$};
\node [scale=\a] at (2.1*\a,1.4*\a) {$D_5$};
\node [scale=\a] at (4.2*\a,1.4*\a) {$D_6$};

\node at (4.4*\a,3.3*\a) {};
\node at (-1.3*\a,-1.2*\a) {};
\node at  (2.1*\a,-1.55*\a) {(e) $u^5$};

\end{tikzpicture}
\begin{tikzpicture}
\def\a{.7}

\shade[shading=ball, ball color=blue!30!white, opacity=0.8]   (0,0) circle[radius=\a];
\shade[shading=ball, ball color=blue!30!white, opacity=0.8]   (2.1*\a,0) circle[radius=\a];
\shade[shading=ball, ball color=blue!30!white, opacity=0.8]   (4.2*\a,0) circle[radius=\a];
\shade[shading=ball, ball color=blue!30!white, opacity=0.8]   (0,2.1*\a) circle[radius=\a];
\shade[shading=ball, ball color=blue!30!white, opacity=0.8]   (2.1*\a,2.1*\a) circle[radius=\a];
\shade[shading=ball, ball color=blue!30!white, opacity=0.8]   (4.2*\a,2.1*\a) circle[radius=\a];

\begin{scope}
	\clip (0.5*\a,-0.3*\a)rectangle(1.6*\a,0.3*\a);
	\filldraw[color=red!60](0.01,1*\a)arc(90:-90:1*\a)--(0.01,-1*\a)--(2.087*\a,-1*\a)--(2.087*\a,-1*\a)arc(-90:-270:1*\a)--cycle;
\end{scope}
\begin{scope}
	\clip (2.5*\a,-0.3*\a)rectangle(3.6*\a,0.3*\a);
	\filldraw[color=red!60](2.11*\a,1*\a)arc(90:-90:1*\a)--(2.11*\a,-1*\a)--(4.187*\a,-1*\a)--(4.187*\a,-1*\a)arc(-90:-270:1*\a)--cycle;
\end{scope}
\begin{scope}
	\clip (0.5*\a,1.8*\a)rectangle(1.6*\a,2.4*\a);
	\filldraw[color=red!60](0.01,3.1*\a)arc(90:-90:1*\a)--(0.01,1.1*\a)--(2.087*\a,1.1*\a)--(2.087*\a,1.1*\a)arc(-90:-270:1*\a)--cycle;
\end{scope}
\begin{scope}
	\clip (2.5*\a,1.8*\a)rectangle(3.6*\a,2.4*\a);
	\filldraw[color=red!60](2.11*\a,3.1*\a)arc(90:-90:1*\a)--(2.11*\a,1.1*\a)--(4.187*\a,1.1*\a)--(4.187*\a,1.1*\a)arc(-90:-270:1*\a)--cycle;
\end{scope}
\begin{scope}
	\clip (-0.3*\a,0.5*\a)rectangle(0.3*\a,1.6*\a);
	\filldraw[color=red!60](1*\a,0.01)arc(0:180:1*\a)--(-1*\a,0.01)--(-1*\a,2.087*\a)--(-1*\a,2.087*\a)arc(-180:0:1*\a)--cycle;
\end{scope}
\begin{scope}
	\clip (1.8*\a,0.5*\a)rectangle(2.4*\a,1.6*\a);
	\filldraw[color=red!60](3.1*\a,0.01)arc(0:180:1*\a)--(1.1*\a,0.01)--(1.1*\a,2.087*\a)--(1.1*\a,2.087*\a)arc(-180:0:1*\a)--cycle;
\end{scope}
\begin{scope}
	\clip (3.9*\a,0.5*\a)rectangle(4.5*\a,1.6*\a);
	\filldraw[color=red!60](5.2*\a,0.01)arc(0:180:1*\a)--(3.2*\a,0.01)--(3.2*\a,2.087*\a)--(3.2*\a,2.087*\a)arc(-180:0:1*\a)--cycle;
\end{scope}

\node [scale=1.2*\a] at (0,0) {1};
\node [scale=1.2*\a] at (2.1*\a,0) {-2};
\node [scale=1.2*\a] at (4.2*\a,0) {1};
\node [scale=1.2*\a] at (0,2.1*\a) {-1};
\node [scale=1.2*\a] at (2.1*\a,2.1*\a) {2};
\node [scale=1.2*\a] at (4.2*\a,2.1*\a) {-1};

\node [scale=\a] at (0,-0.7*\a) {$D_1$};
\node [scale=\a] at (2.1*\a,-0.7*\a) {$D_2$};
\node [scale=\a] at (4.2*\a,-0.7*\a) {$D_3$};
\node [scale=\a] at (0,1.4*\a) {$D_4$};
\node [scale=\a] at (2.1*\a,1.4*\a) {$D_5$};
\node [scale=\a] at (4.2*\a,1.4*\a) {$D_6$};

\node at (4.4*\a,3.3*\a) {};
\node at (-1.3*\a,-1.2*\a) {};
\node at  (2.1*\a,-1.55*\a) {(f) $u^6$};

\end{tikzpicture}
\caption{Resonant modes of matrix arrangement for $N=6$.}\label{f_m6}
\end{figure}
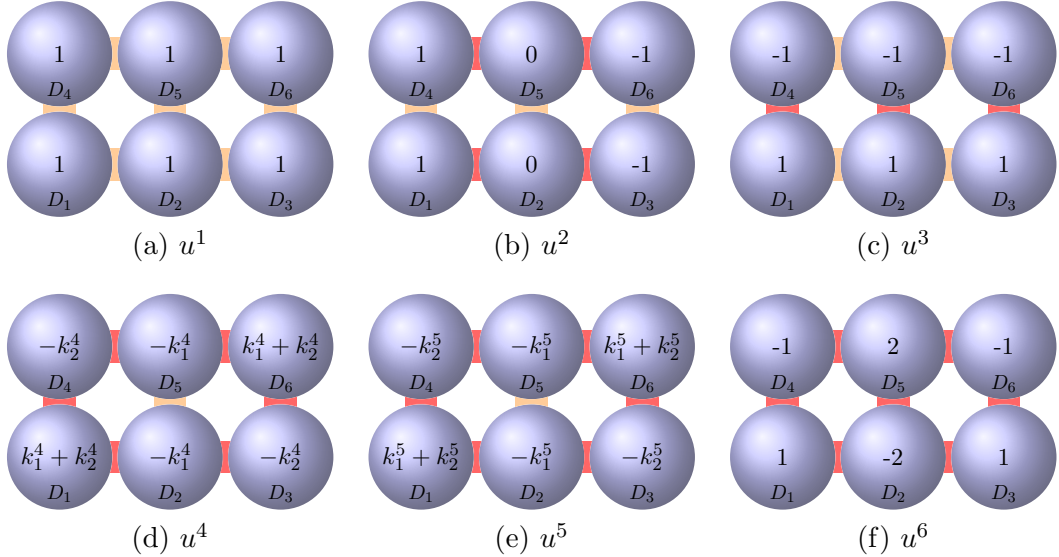

\section{Concluding remarks}\label{sec_con}

In this paper, we investigate how the number and spatial arrangement of resonators in a subwavelength acoustic resonance system influence its resonant frequencies. This is achieved by analyzing the eigenvalues and eigenvectors of the capacitance matrix. We derive precise analytical expressions to quantify these effects and rigorously examine the asymptotic behavior of resonant modes associated with each frequency. Our results demonstrate that resonator arrangement critically governs the system’s response range and frequency distribution. Notably, resonant modes exhibit increasingly diverse asymptotic behaviors as the number of resonators or their spatial configuration varies.

While our study focuses on spherical resonators, the methodology generalizes readily to other geometries, such as 
$m$-convex shapes. Crucially, our analysis reveals that frequency distribution depends not on resonator symmetry or relative positioning, but on the number of neighboring resonators surrounding each unit. Consequently, our conclusions extend to any configuration where the local resonator density (i.e., the count of adjacent resonators) remains consistent. The method’s generalizability enables efficient derivation of resonant frequency distributions and modal behaviors for complex arrangements. Furthermore, this framework applies to elastic wave resonance systems with multiple inclusions—a direction we plan to explore in future work.

These insights advance the design of functional acoustic materials. Key recommendations include: Matrix arrangements to broaden system response ranges; Ring/matrix configurations to maximize frequency gaps; Chain arrangements to isolate distinct response frequencies. By correlating resonant mode asymptotics with arrangement-dependent effects, our work empowers tailored material design. For instance, designers can leverage modal behavior analysis to engineer materials with target spectral or spatial properties.


\begin{thebibliography}{99}
	\bibitem{abcf}H. Ammari, S. Barandun, J. Cao, and F. Feppon, Edge modes in subwavelength resonators in one dimension, Multiscale Model. Simul., 21 (2023), pp. 964--992.
	
	\bibitem{ackly}H. Ammari, G. Ciraolo, H. Kang, H. Lee, and K. Yun, Spectral analysis of the Neumann--Poincar\'e operator and characterization of the stress concentration in anti-plane elasticity, Arch. Ration. Mech. Anal., 208 (2013), pp. 275--304.
	
	\bibitem{adfms}H. Ammari, A. Dabrowski, B. Fitzpatrick, P. Millien, and M. Sini, Subwavelength resonant dielectric nanoparticles with high refractive indices, Math. Methods Appl. Sci., 42:18 (2019), pp. 6567--6579.
	
	\bibitem{ad1}H. Ammari and B. Davies, A fully--coupled subwavelength resonance approach to filtering auditory signals, Proc. A, 475 (2019), 20190049.
	
	\bibitem{ad2}H. Ammari and B. Davies, Mimicking the active cochlea with a fluid-coupled array of subwavelength Hopf resonators, Proc. A, 476 (2020), 20190870.
	
	\bibitem{adh} H. Ammari, B. Davies and E.~O. Hiltunen, Functional analytic methods for discrete approximations of subwavelength resonator systems, Pure Appl. Anal., 6 (2024), pp. 873--939.
	
	\bibitem{adhly} H. Ammari, B. Davies, E. O. Hiltunen, H. Lee, and S. Yu, “High-order exceptional points and enhanced sensing in subwavelength resonator arrays”, Stud. Appl. Math. 146:2 (2021), 440-462.
	
	\bibitem{adhy}H. Ammari, B. Davies, E.~O. Hiltunen, and S. Yu, Topologically protected edge modes in one-dimensional chains of subwavelength resonators, J. Math. Pures Appl., 144 (2020), pp. 17--49.
	
	\bibitem{ady}H. Ammari, B. Davies, and S. Yu, Close--to--touching acoustic subwavelength resonators: Eigenfrequency separation and gradient blow--up, Multiscale Model. Simul., 18 (2020), pp. 1299--1317.
	
	\bibitem{afglz2}H. Ammari, B. Fitzpatrick, D. Gontier, H. Lee, and H. Zhang, Sub--wavelength focusing of acoustic waves in bubbly media, Proc. A, 473 (2017), 20170469.
	
	\bibitem{afglz} H. Ammari, B. Fitzpatrick, D. Gontier, H. Lee, and H. Zhang, Minnaert resonances for acoustic waves in bubbly media, Ann. Inst. H. Poincar\'e Anal. Non Lin\'eaire, 35 (2018), pp. 1975--1998.
	
	
	
	\bibitem{aflyz}H. Ammari, B. Fitzpatrick, H. Lee, S. Yu, and H. Zhang, Double--negative acoustic metamaterials, Quart. Appl. Math., 77 (2019), pp. 767--791.
    \bibitem{ahy}H. Ammari, E.~O. Hiltunen and S. Yu, A high-frequency homogenization approach near the Dirac points in bubbly honeycomb crystals, Arch. Ration. Mech. Anal. 238 (2020), pp. 1559--1583.
    
	\bibitem{akky}H. Ammari, H. Kang, D. Kim, S. Yu, Quantitative estimates for stress concentration of the Stokes flow between adjacent circular cylinders. SIAM J. Math. Anal. 55 (2023), no. 4, pp. 3755--3806.
	
	\bibitem{aklll}H. Ammari, H. Kang, H. Lee, J. Lee, and M. Lim, Optimal estimates for the electric field in two dimensions, J. Math. Pures Appl., 88 (2007), pp. 307--324.
	
	\bibitem{alz}H. Ammari, B. Li, and J. Zou, Mathematical analysis of electromagnetic scattering by dielectric nanoparticles with high refractive indices, Trans. Amer. Math. Soc., 376:1 (2023), pp. 39--90.
	
	\bibitem{amrz}H. Ammari, P. Millien, M. Ruiz, and H. Zhang, Mathematical analysis of plasmonic nanoparticles: the scalar case, Arch. Ration. Mech. Anal., 224:2 (2017), pp. 597--658.
	
	\bibitem{aryz}H. Ammari, M. Ruiz, S. Yu, and H. Zhang, Mathematical analysis of plasmonic resonances for nanoparticles: the full
	Maxwell equations, J. Differential Equations, 261:6 (2016), pp. 3615--3669.
	
	\bibitem{az}H. Ammari and H. Zhang, A mathematical theory of super--resolution by using a system of sub-wavelength Helmholtz resonators, Comm. Math. Phys., 337:1 (2015), pp. 379--428.
	
	\bibitem{b}V. F. K. Bjerknes, Fields of Force, The Columbia University Press, New York, 1906.
	
	\bibitem{bll}J. Bao, H. Li, and Y. Li, Gradient estimates for solutions of the Lam\'e system with partially infinite coefficients, Arch. Ration. Mech. Anal., 215 (2015), pp. 307--351.
	
	\bibitem{bll2}J. Bao, H. Li, and Y. Li, Gradient estimates for solutions of the Lam\'e system with partially infinite coefficients in dimensions greater than two, Adv. Math., 305 (2017), pp. 298--338.
	
	\bibitem{bly}E. Bao, Y. Li, and B. Yin, Gradient estimates for the perfect conductivity problem, Arch. Ration. Mech. Anal., 193 (2009), pp. 195--226.
	
	\bibitem{bly2}E. Bao, Y. Li and B. Yin, Gradient estimates for the perfect and insulated conductivity problems with multiple inclusions, Comm. Partial Differential Equations, 35 (2010), pp. 1982--2006.
	
	\bibitem{bt}E. Bonnetier and F. Triki, On the spectrum of the Poincar\'e variational problem for two close-to-touching inclusions in 2D, Arch. Ration. Mech. Anal., 209 (2013), pp. 541--567.
	
	\bibitem{ck}D.~L. Colton and R. Kress,  Integral equation methods in scattering theory, Pure and Applied Mathematics (New York) A Wiley-Interscience Publication, Wiley, New York, 1983.
	
	\bibitem{c}L. A. Crum, Bjerknes forces on bubbles in a stationary sound field, J. Acoust. Soc. Amer., 57 (1975), pp. 1363--1370.
	
	\bibitem{dhbl}M. Devaud, T. Hocquet, J. C. Bacri, and V. Leroy, The Minnaert bubble: An acoustic approach, Eur. J. Phys., 29 (2008), 1263.
	
	\bibitem{dll}H. Dong, H. Li and L. Xu, On higher regularity of Stokes systems with piecewise H\"older continuous coefficients, Trans. Amer. Math. Soc., 377 (2024), pp. 8477--8513.
	
	\bibitem{dly} H. Dong, Y. Li, Z. Yang, Gradient estimates for the insulated conductivity problem: the non-umbilical case. J. Math. Pures Appl. (9) 189 (2024), Paper No. 103587, 37 pp.
	
	%\bibitem{dz} S. Dyatlov and M. Zworski, Mathematical theory of scattering resonances, Graduate Studies in Mathematics 200,
	%American Mathematical Society, Providence, RI, 2019.
	
	\bibitem{fa} F. Feppon, and H. Ammari, Modal decompositions and point scatterer approximations near the Minnaert resonance
	frequencies, Stud. Appl. Math., 149 (2022), pp. 164--229.
	
	\bibitem{fca} F. Feppon, Z. Cheng and H. Ammari, Subwavelength resonances in one-dimensional high-contrast acoustic media, SIAM J. Appl. Math.,  83 (2023), pp. 625--665.
	
	\bibitem{ky} H. Kang, S. Yu, Quantitative characterization of stress concentration in the presence of closely spaced hard inclusions in two-dimensional linear elasticity. Arch. Ration. Mech. Anal. 232 (2019), no. 1, 121–196.
	
	\bibitem{kbs}H. K. Khattak, P. Bianucci, and A. D. Slepkov, Linking plasma formation in grapes to microwave resonances of aqueous dimers, Proc. Natl. Acad. Sci. USA, 116 (2019), pp. 4000--4005.
	
	\bibitem{kl}J. Kim and M. Lim, Electric field concentration in the presence of an inclusion with eccentric core-shell geometry, Math. Ann., 373 (2019), pp. 517--551.
	
	\bibitem{kdd}M. Kushwaha, B. Djafari-Rouhani, and L. Dobrzynski, Sound isolation from cubic arrays
	of air bubbles in water, Phys. Lett. A, 248 (1998), pp. 252--256.
	
	\bibitem{lbfwtd} V. Leroy, A. Bretagne, M. Fink, H. Willaime, P. Tabeling, and A. Tourin, Design and characterization of bubble phononic crystals, Appl. Phys. Lett., 95 (2009), 171904.
	
	\bibitem{l1} H. Li, Lower bounds of gradient's blow-up for the Lamé system with partially infinite coefficients, J. Math. Pures Appl. (9) 149 (2021), 98--134.
	
	\bibitem{lly}H. Li, Y. Li and Z. Yang, Asymptotics of the gradient of solutions to the perfect conductivity problem, Multiscale Model. Simul., 17 (2019), pp. 899--925.
	
	\bibitem{lx}H. Li and L. Xu, Estimates for stress concentration between two adjacent rigid inclusions in Stokes flow, J. Funct. Anal., 286 (2024), 110313.
	
	\bibitem{lxz2}H. Li, L. Xu and P. Zhang, Stress blowup analysis when a suspending rigid particle approaches the boundary in Stokes flow: 2-dimensional case, SIAM J. Math. Anal., 55 (2023), pp. 4493--4536.
	
	\bibitem{lxz}H. Li, L. Xu and P. Zhang, Stress blow-up analysis when suspending rigid particles approach boundary in 3D Stokes flow, Discrete Contin. Dyn. Syst., 44 (2024), pp. 1134--1165.
	
	\bibitem{lz}H. Li, and Y. Zhao, The interaction between two close--to--touching convex acoustic subwavelength resonators, Multiscale Model. Simul., 21 (2023), pp. 804--826.
	
	\bibitem{lyz} Y. Li, Z. Yang, Gradient estimates of solutions to the insulated conductivity problem in dimension greater than two. Math. Ann. 385 (2023), no. 3-4, 1775--1796.
	
	\bibitem{ly}M. Lim and K. Yun, Blow-up of electric fields between closely spaced spherical perfect conductors, Comm. Partial Differential Equations, 34 (2009), pp. 1287--1315.
	
	\bibitem{mms} T. Meklachi, S. Moskow, and J. C. Schotland, Asymptotic analysis of resonances of small volume high contrast linear and nonlinear scatterers, J. Math. Phys., 59:8 (2018), 083502.
	
	\bibitem{m}M. Minnaert, On musical air--bubbles and the sounds of running water, Philos. Mag., 16 (1933), pp. 235--248.
	
	\bibitem{p2}V. Pandey, Asymmetricity and sign reversal of secondary Bjerknes force from strong nonlinear coupling in cavitation bubble pairs, Phys. Rev. E, 99 (2019), 042209.
	
	\bibitem{p}L. Poladian, Asymptotic behaviour of the effective dielectric constants of composite materials, Proc. A, 426 (1989), pp. 343--359.
	
	\bibitem{wen} B. Weinkove, The insulated conductivity problem, effective gradient estimates and the maximum principle, Math. Ann. 385 (1–2) (2023) 1--16.
	
	\bibitem{ya}S. Yu and H. Ammari, Plasmonic interaction between nanospheres, SIAM Rev., 60 (2018), pp. 356--385.
	
\end{thebibliography}
\end{document}